\documentclass{smfart}
\usepackage{lmodern}
\usepackage[latin1]{inputenc}
\usepackage[OT2,T1]{fontenc}
\usepackage[francais]{babel}
\usepackage{amsmath,amssymb,amsfonts,amsthm}
\usepackage[pdfborder={0 0 0}]{hyperref}
\usepackage{mathrsfs}
\newcommand{\mylabel}[1]{{\phantomsection\label{#1}}}
\newcommand{\mylangle}{\xy(0,0),(1.17,1.485)**[|(1.025)]@{-},(0,0),(1.17,-1.485)**[|(1.025)]@{-}\endxy\mkern2mu}
\newcommand{\myrangle}{\mkern2mu\xy(0,0),(-1.17,1.485)**[|(1.025)]@{-},(0,0),(-1.17,-1.485)**[|(1.025)]@{-}\endxy}
\makeatletter
\@mparswitchfalse
\makeatother
\newcount\myxyoldpv
\newcount\myxyoldpe
\def\myxyin{\myxyoldpv=\catcode`;\catcode`;=12\myxyoldpe=\catcode`!\catcode`!=12}
\def\myxyout{\catcode`;=\myxyoldpv\catcode`!=\myxyoldpe}
\makeatletter
\newenvironment{proofnoskip}[1][\proofname]{\par \normalfont
  \topsep6\p@\@plus6\p@ \trivlist \itemindent\z@ 
  \def\@proofhead{\normalfont\itshape #1}%
  \sbox\@tempboxa{\@proofhead}%
  \item[\hskip\labelsep
          \unhbox\@tempboxa\pointrait]%
  \ignorespaces
}{%
  \MakeQed\endtrivlist
}
\makeatother
\newcommand{\demicrochet}{
}

\mathcode`A="7041 \mathcode`B="7042 \mathcode`C="7043 \mathcode`D="7044
\mathcode`E="7045 \mathcode`F="7046 \mathcode`G="7047 \mathcode`H="7048
\mathcode`I="7049 \mathcode`J="704A \mathcode`K="704B \mathcode`L="704C
\mathcode`M="704D \mathcode`N="704E \mathcode`O="704F \mathcode`P="7050
\mathcode`Q="7051 \mathcode`R="7052 \mathcode`S="7053 \mathcode`T="7054
\mathcode`U="7055 \mathcode`V="7056 \mathcode`W="7057 \mathcode`X="7058
\mathcode`Y="7059 \mathcode`Z="705A

\usepackage[ps,dvips,arrow,matrix,tips,line,curve]{xy}
\entrymodifiers={+!!<0pt,\the\fontdimen22\textfont2>}
\SelectTips{cm}{10}
\newdir^{ (}{{}*!/-3.5pt/\dir^{(}}
\newdir_{ (}{{}*!/-3.5pt/\dir_{(}}

\title[Zéro-cycles sur les fibrations]{Zéro-cycles sur les fibrations au-dessus d'une courbe de genre quelconque}

\date{18 octobre 2010; révisé le 21 janvier 2012}

\author[O. Wittenberg]{Olivier Wittenberg}
\email{wittenberg@dma.ens.fr}
\address{D\'epartement de math\'ematiques et applications\\
\'Ecole normale sup\'erieure\\
45~rue d'Ulm\\
75230 Paris Cedex 05\\
France}

\renewcommand{\phi}{\varphi}
\newcommand{\Htilde}{\widetilde{H}}
\newcommand{\RHom}{R\mathscr{H}\mkern-4muom}

\newcommand{\Sm}{\mathrm{Sm}}
\newcommand{\pr}{\mathrm{pr}}
\newcommand{\abs}[1]{\left| \mskip1mu #1 \right|}
\newcommand{\fchapeau}{{\mkern2mu\widehat{\mkern-2muf\mkern2mu}\mkern-.1mu}}
\newcommand{\uchapeau}{{\widehat{u}}}
\newcommand{\zchapeau}{{\widehat{z}}}
\newcommand{\ychapeau}{{\widehat{y}}}

\newcommand{\Psub}[1]{P_{\mkern-2mu#1}}
\newcommand{\Ysub}[1]{Y_{\mkern-2mu#1}}
\newcommand{\Wsub}[1]{W_{\mkern-2mu#1}}
\newcommand{\Tsub}[1]{T_{\mkern-2mu#1}}
\newcommand{\eff}{{\mathrm{eff}}}
\newcommand{\ab}{{\mathrm{ab}}}
\newcommand{\CH}{{\mathrm{CH}}}
\newcommand{\Gm}{\mathbf{G}_\mathrm{m}}
\newcommand{\sU}{{\mathscr{U}}}

\newcommand{\sF}{{\mathscr{F}}}
\newcommand{\sG}{{\mathscr{G}}}
\newcommand{\sR}{{\mathscr{R}}}
\newcommand{\sH}{{\mathscr{H}}}
\newcommand{\sB}{{\mathscr{B}}}
\newcommand{\sD}{{\mathscr{D}}}
\newcommand{\sC}{{\mathscr{C}}}
\newcommand{\sX}{{\mathscr{X}}}
\newcommand{\sS}{{\mathscr{S}}}
\newcommand{\sV}{{\mathscr{V}}}
\newcommand{\sO}{{\mathscr{O}}}
\newcommand{\sM}{{\mathscr{M}}}

\newcommand{\kbar}{{\mkern1mu\overline{\mkern-1mu{}k\mkern-1mu}\mkern1mu}}
\newcommand{\etaT}{{\eta_{\mkern1muT}}}
\newcommand{\etabar}{{\mkern1mu\overline{\mkern-1mu{}\eta\mkern-1mu}\mkern1mu}}
\newcommand{\mmu}{{\boldsymbol{\mu}}}
\newcommand{\tors}[2]{{\vphantom{#2}}_{#1}{#2}}
\newcommand{\Spec}{\mathop\mathrm{Spec}\nolimits}
\newcommand{\Gal}{\mathop\mathrm{Gal}\nolimits}

\newcommand{\Card}{\mathop\mathrm{Card}\nolimits}
\newcommand{\Pic}{\mathop\mathrm{Pic}\nolimits}
\newcommand{\Sym}{\mathop\mathrm{Sym}\nolimits}

\newcommand{\Br}{\mathop\mathrm{Br}\nolimits}
\newcommand{\Hom}{\mathop\mathrm{Hom}\nolimits}
\renewcommand{\Im}{\mathop\mathrm{Im}\nolimits}
\newcommand{\Ker}{\mathop\mathrm{Ker}\nolimits}
\newcommand{\Coker}{\mathop\mathrm{Coker}\nolimits}
\newcommand{\Res}{{\mathrm{Res}}}
\newcommand{\A}{{\mathbf A}}
\newcommand{\Q}{{\mathbf Q}}
\newcommand{\F}{{\mathbf F}}
\newcommand{\Fv}{{\F_{\mkern-2muv}}}
\newcommand{\Z}{{\mathbf Z}}
\newcommand{\R}{{\mathbf R}}
\renewcommand{\vert}{{\mathrm{vert}}}
\newcommand{\uplet}[2]{#1, \mskip2.5mu \ldots \mskip-1mu, \mskip2.5mu #2}
\renewcommand{\P}{{\mathbf P}}
\renewcommand{\leq}{\leqslant}
\renewcommand{\geq}{\geqslant}
\renewcommand{\emptyset}{\varnothing}
\DeclareMathOperator{\inv}{inv}
\newcommand{\isoto}{\myxrightarrow{\,\sim\,}}

\makeatletter
\def\myrightarrow{{\setbox\z@\hbox{$\rightarrow$}\dimen0\ht\z@\multiply\dimen0 6\divide\dimen0 10\ht\z@\dimen0\box\z@}}
\def\myrightarrowfill@{\arrowfill@\relbar\relbar\myrightarrow}
\newcommand{\myxrightarrow}[2][]{\ext@arrow 0359\myrightarrowfill@{#1}{#2}}
\makeatother

\newcommand{\chapeau}{{\rlap{\smash{\hbox{\lower4pt\hbox{\hskip1pt$\widehat{\phantom{u}}$}}}}}}
\newcommand{\CHzchapeau}{\CH_{0\phantom{,}}^\chapeau}
\newcommand{\CHzAchapeau}{\CH_{0,\A}^\chapeau}
\newcommand{\CHzA}{\CH_{0,\A}}
\newcommand{\petitplus}{+}
\newcommand{\Picchapeau}{\Pic^{{\smash{\hbox{\lower4pt\hbox{\hskip0.4pt$\widehat{\phantom{u}}$}}}}}}

\newcommand{\PicAchapeau}{\Picchapeau_\A}
\newcommand{\Picpluschapeau}{\Pic_\petitplus^{{\smash{\hbox{\lower4pt\hbox{\hskip0.4pt$\widehat{\phantom{u}}$}}}}}}
\newcommand{\PicplusAchapeau}{\Pic_{\petitplus,\A}^{{\smash{\hbox{\lower4pt\hbox{\hskip.4pt$\widehat{\phantom{u}}$}}}}}}
\newcommand{\Picplus}{\Pic_\petitplus}

\newcommand{\Brplus}{\Br_\petitplus}
\newcommand{\vplus}{\underline{v}_\petitplus}

\numberwithin{equation}{section}
\theoremstyle{plain}
\newtheorem{prop}{Proposition}[section]
\newtheorem{thm}[prop]{Théorème}
\newtheorem{cor}[prop]{Corollaire}
\newtheorem{lem}[prop]{Lemme}
\newtheorem{rappel}[prop]{Rappel}
\theoremstyle{remark}

\newtheorem*{conjE*}{Conjecture~$(E)$}
\newtheorem*{conjEzero*}{Conjecture~$(E_0)$}
\newtheorem*{conjEun*}{Conjecture~$(E_1)$}
\newtheorem*{rmq*}{Remarque}
\newtheorem{rmq}[prop]{Remarque}
\newtheorem{rmqs}[prop]{Remarques}

\input cyracc.def
\DeclareFontFamily{U}{russian}{}
\DeclareFontShape{U}{russian}{m}{n}
        { <5><6> wncyr5
        <7><8><9> wncyr7
        <10><10.95><12><14.4><17.28><20.74><24.88> wncyr10 }{}
\DeclareSymbolFont{Russian}{U}{russian}{m}{n}
\DeclareSymbolFontAlphabet{\mathcyr}{Russian}
\makeatletter
\let\@math@cyr\mathcyr
\renewcommand{\mathcyr}[1]{\@math@cyr{\cyracc #1}}
\makeatother
\newcommand{\Sha}{{\mathcyr{Sh}}}
\edef\textsection{\textsection\penalty10000\hskip3.4pt}

{\setbox0\hbox{$ $}}\fontdimen16\textfont2=\fontdimen17\textfont2

\makeatletter
\def\@setaddresses{\par\nobreak
  \begingroup
  \parindent-2em\leftskip2em
  \rightskip=0pt plus 20pt
  \emergencystretch .5\textwidth
  \exhyphenpenalty=-100
  \interlinepenalty\@M
  \def\baselinestretch{1}\normalfont\footnotesize
  \def\\{\unskip, \penalty-10\ignorespaces}%
  \def\cond@bullet {\unskip
         {\discretionary{}{}{\hbox{\ $\bullet$\ }}}}
  \def\author##1{\ifhmode\par\nobreak \vskip\smallskipamount\fi
      {\scshape ##1}\let\address\firstaddress}%
  \def\firstaddress##1##2{\unskip, \let\address\otheraddress
         \penalty-20\ignorespaces##2}%
  \def\otheraddress##1##2{\cond@bullet \ignorespaces##2}%
  \def\curraddr{\address}
  \let\address\firstaddress
  \def\email##1##2{\@ifnotempty{##2}%
        {\cond@bullet
         \hbox{\itshape Courriel~:}~{\ttfamily\ignorespaces ##2}}}%
  \def\urladdr##1##2{\@ifnotempty{##2}%
        {\cond@bullet
         {\itshape Url~:}~{\ttfamily\ignorespaces ##2}\par}}%
  \addresses
  \par\endgroup
}
\makeatother

\begin{document}
\begin{abstract}
Soit~$X$ une vari\'et\'e propre et lisse sur un corps de nombres~$k$.
Des conjectures sur l'image du groupe de Chow des z\'ero-cycles de~$X$ dans le produit des m\^emes groupes sur tous les compl\'et\'es de~$k$ ont \'et\'e propos\'ees par Colliot-Th\'el\`ene, Kato et Saito.
Nous d\'emontrons ces conjectures pour l'espace total de fibrations en vari\'et\'es rationnellement connexes v\'erifiant l'approximation faible, au-dessus de courbes dont le groupe de Tate--Shafarevich est fini, sous une
hypoth\`ese d'ab\'elianit\'e sur les fibres singuli\`eres. 
\end{abstract}

\maketitle

\section*{Introduction}

Soit~$X$ une variété propre et lisse sur un corps de nombres~$k$.
À~défaut de comprendre la structure du groupe de Chow~$\CH_0(X)$
des zéro-cycles sur~$X$ à équivalence
rationnelle près, dont on ne sait pas, notamment, s'il est de type fini (conjecture de Beilinson et Bloch),
on peut s'interroger sur l'image et le noyau de l'application naturelle $\CH_0(X) \to \prod_{v \in \Omega} \CH_0(X \otimes_k k_v)$, où~$\Omega$ désigne l'ensemble des places de~$k$ et~$k_v$ le complété de~$k$ en $v \in \Omega$.

Le premier résultat dans cette direction est dû à Cassels.  Celui-ci démontra dans~\cite{casselsdual} que si~$E$ est une courbe elliptique sur~$k$
et si le groupe de Tate--Shafarevich de~$E$ ne contient pas d'élément infiniment divisible non nul,
alors l'adhérence de~$E(k)$ dans le groupe des points adéliques $E(\A_k)$
coïncide avec le noyau de l'homomorphisme naturel $E(\A_k) \to \Hom(H^1(k,E),\Q/\Z)$,
à condition de munir~$E(\A_k)$ d'une topologie rendue plus grossière aux places infinies.
Rappelons que dans ce contexte on a $\CH_0(E)=E(k)\oplus \Z$ et $\CH_0(E \otimes_k k_v)=E(k_v)\oplus \Z$
et que le groupe de Tate--Shafarevich de~$E$ est conjecturalement fini.

Soit $A_0(X) \subset \CH_0(X)$ le sous-groupe des classes de $0$-cycles de degré~$0$.
Lorsque~$X$ est une surface (géométriquement) rationnelle, une conjecture de Colliot-Thélène et Sansuc~\cite[Conjecture~A]{ctsduke} postule l'existence d'une suite exacte canonique de groupes abéliens finis
\begin{align}
\mylabel{eq:sects}
\xymatrix@C=1.5em{
0 \ar[r] & \Sha^1(k,S) \ar[r] & A_0(X) \ar[r] & \displaystyle \bigoplus_{v \in \Omega} A_0(X \otimes_k k_v) \ar[r] & \Hom(H^1(k,S^*), \Q/\Z)
}
\end{align}
où $\Sha^1(k,S)=\Ker(H^1(k,S) \to \prod_{v \in \Omega} H^1(k_v,S))$
et
où~$S$ désigne le tore sur~$k$ dont le groupe des caractères est $S^*=\Pic(X \otimes_k \kbar)$.
De plus, si $H^1(k,S^*)=0$, l'existence d'un $0$\nobreakdash-cycle de degré~$1$ sur $X \otimes_k k_v$ pour chaque $v \in \Omega$ devrait entraîner celle d'un $0$\nobreakdash-cycle de degré~$1$ sur~$X$
\cite[Conjecture~C]{ctsduke}.

Notons $\Br(X)=H^2(X,\Gm)$ le groupe de Brauer cohomologique de~$X$.
Quelle que soit la variété~$X$, il résulte de la loi de réciprocité de la théorie du corps de classes global
que l'image de l'application $\CH_0(X) \to \prod_{v \in \Omega} \CH_0(X \otimes_k k_v)$ est incluse dans le noyau à droite de l'accouplement
dit \og{}de Brauer--Manin\fg{}
\begin{align}
\mylabel{eq:accbm}
\Br(X) \times \prod_{v \in \Omega} \CH_0(X \otimes_k k_v) \to \Q/\Z
\end{align}
défini par Manin~\cite{maninicm}.
En termes de cet accouplement, des conjectures généralisant aux variétés propres et lisses arbitraires aussi bien le théorème de Cassels que les conjectures~A et~C de~\cite{ctsduke}, exception
faite de la description du noyau de la flèche $A_0(X) \to \prod_{v \in \Omega} A_0(X\otimes_k k_v)$, furent proposées par Kato et Saito~\cite[\textsection7]{katosaitocontemp}
et par Colliot-Thélène~\cite[Conjectures~1.5]{ctbordeaux}.  Celles-ci affirment notamment:

\begin{conjEzero*}
Soit~$X$ une variété propre et lisse sur un corps de nombres~$k$.
Le complexe
$$
\xymatrix{
\displaystyle\varprojlim_{n} A_0(X)/n \ar[r] & \displaystyle\prod_{v \in \Omega} \varprojlim_{n} A_0(X \otimes_k k_v)/n \ar[r] & \Hom(\Br(X),\Q/\Z)\rlap{\text{,}}
}
$$
dans lequel la seconde flèche est induite par~(\ref{eq:accbm}), est une suite exacte.
\end{conjEzero*}

\begin{conjEun*}
Soit~$X$ une variété propre et lisse sur un corps de nombres~$k$.  S'il existe une famille $(z_v)_{v \in \Omega} \in \prod_{v \in \Omega} \CH_0(X \otimes_k k_v)$
orthogonale à~$\Br(X)$ pour l'accouplement de Brauer--Manin et telle que $\deg(z_v)=1$ pour tout $v \in \Omega$,
alors il existe un $0$\nobreakdash-cycle de degré~$1$ sur~$X$.
\end{conjEun*}

\begin{rmq*}
Lorsque~$X$ est une surface (géométriquement) rationnelle,
 la suite spectrale de Hochschild--Serre induit une surjection canonique $\Br(X) \to H^1(k,S^*)$,
les
groupes $A_0(X)$ et $A_0(X \otimes_k k_v)$ sont finis
et $A_0(X \otimes_k k_v)$ s'annule pour tout~$v$ hors d'un ensemble fini de places
(cf.~\cite[Théorème~2.9]{ctbordeaux}).
La conjecture~$(E_0)$ équivaut donc, dans ce cas, à l'exactitude de la partie droite de~(\ref{eq:sects}).
\end{rmq*}

On ne dispose pas à l'heure actuelle de conjecture généralisant à toutes les variétés propres et lisses la partie gauche de la suite exacte~(\ref{eq:sects}). Voir~\cite{suresh} à ce sujet,
ainsi que~\cite{gonzalezaviles}.

Les conjectures~$(E_0)$ et~$(E_1)$ furent prouvées pour les courbes par Saito~\cite{saito}
lorsque le groupe de Tate--Shafarevich de la jacobienne ne contient pas d'élément infiniment divisible non nul.
Une démonstration simplifiée est donnée dans~\cite[\textsection3]{ctconj}.
Le cas des courbes de genre~$1$, au moins pour ce qui est de~$(E_1)$, était déjà connu grâce aux travaux de Cassels reformulés par Manin~\cite{maninicm}.

En dimension supérieure, Salberger~\cite{salberger} établit, dans un article fondateur, les conjectures~$(E_0)$ et~\cite[Conjecture~A]{ctsduke}
pour les surfaces rationnelles fibrées en coniques au-dessus de~$\P^1_k$, ainsi que la conjecture~$(E_1)$ sous l'hypothèse additionnelle que $H^1(k,\Pic(X \otimes_k \kbar))=0$.
Cette hypothèse additionnelle est supprimée dans~\cite{salbergertorseurs} et, indépendamment, dans~\cite{ctsd94}.

Le résultat de Salberger fut généralisé dans deux directions.
Tout d'abord, la technique employée dans~\cite{salberger} permit de confirmer les conjectures~$(E_0)$ et~$(E_1)$ pour les variétés fibrées au-dessus de~$\P^1_k$ en quadriques de dimension~$2$ ou en
variétés de Severi--Brauer (cf.~\cite[\textsection6]{ctsd94}).
Colliot-Thélène, Skorobogatov et Swinnerton-Dyer~\cite[\textsection4]{ctsksd98} dégagèrent par la suite un énoncé tout général portant sur les fibrations au-dessus de~$\P^1_k$:
si $f:X \to \P^1_k$ est un morphisme dont la fibre générique est géométriquement irréductible,
la conjecture~$(E_1)$, ainsi qu'une version affaiblie de~$(E_0)$, vaut pour~$X$ dès que les fibres lisses de~$f$ vérifient
le principe de Hasse pour l'existence de $0$\nobreakdash-cycles de degré~$1$ et que les fibres singulières de~$f$ contiennent chacune une composante irréductible de multiplicité~$1$
déployée par une extension abélienne de leur corps de base.
Ces deux hypothèses sont notamment satisfaites lorsque la fibre générique de~$f$ est une quadrique de dimension~$2$ ou une variété de Severi--Brauer
(cf.~\emph{op.\ cit.}, Examples~1.6).

Une seconde série d'articles étudia les fibrations au-dessus de courbes de genre non nul.
La division euclidienne des polynômes dans~$k[t]$ joue un rôle décisif dans l'argument de~\cite{salberger}.
Colliot-Thélène~\cite{ctreglees} montra qu'il est possible de lui substituer le théorème de Riemann--Roch (cf.~\emph{op.\ cit.}, Lemme principal~5.2).
Il~prouva ainsi la conjecture~$(E_1)$ lorsque~$X$ est une surface munie d'une fibration en coniques $f:X \to C$ et que le corps~$k$ est totalement imaginaire,
sous les hypothèses supplémentaires
que l'application $f^*:\Br(C)\to\Br(X)$ est surjective,
que la courbe~$C$ est de genre~$1$ et
que le groupe de Tate--Shafarevich de sa jacobienne n'a pas d'élément infiniment divisible non nul.
En outre, sans supposer ces hypothèses supplémentaires satisfaites, il établit une version relative de la conjecture~$(E_0)$.
Frossard~\cite{frossard} étendit ces résultats en démontrant les conjectures~$(E_0)$ et~$(E_1)$, ainsi que la version relative de~$(E_0)$ apparue dans~\cite{ctreglees},
lorsque~$X$ est l'espace total d'une fibration en variétés de Severi--Brauer d'indice sans facteur carré
au-dessus d'une courbe~$C$ de genre quelconque sur un corps de nombres totalement imaginaire, en supposant, pour~$(E_0)$ et~$(E_1)$ mais non pour la version relative de~$(E_0)$,
que le groupe de Tate--Shafarevich de la jacobienne de~$C$ ne contient
pas d'élément infiniment divisible non nul.  Enfin, à l'aide d'un formalisme cohomologique élaboré,
van~Hamel~\cite{vanhamel} supprima de ces théorèmes l'hypothèse que~$k$ est totalement imaginaire.

Les arguments de~\cite{ctreglees} et~\cite{frossard} reposent sur
les liens qui unissent variétés de Severi--Brauer et cohomologie galoisienne
et sur
plusieurs outils spécifiques aux variétés de Severi--Brauer.
Interviennent notamment:
un théorème de Merkurjev et~Suslin
concernant les algèbres simples centrales d'indice sans facteur carré
combiné à un théorème de Kato affirmant la nullité du groupe de cohomologie non ramifiée $H^3_{\mathrm{nr}}(k_v(C),\Z/n\Z(2))$ pour $v \in \Omega$ finie (cf.~\cite[p.~88]{frossard});
l'étude de la géométrie et du groupe de Chow des $0$\nobreakdash-cycles
des fibres singulières des modèles d'Artin des variétés de Severi--Brauer d'indice sans facteur carré sur le corps des fonctions d'une courbe (cf.~\emph{idem}, p.~85 et p.~91);
un théorème d'effectivité, dû à Salberger, pour les $0$\nobreakdash-cycles de degré assez grand sur l'espace total d'une fibration en variétés de Severi--Brauer d'indice premier au-dessus d'une courbe (cf.~\emph{id.}, p.~91).

Dans cet article, faisant table rase de l'approche mise en \oe{}uvre dans~\cite{ctreglees} et dans~\cite{frossard},
nous montrons que l'on peut se dispenser à la fois des outils
mentionnés au paragraphe précédent et du substitut de division euclidienne introduit dans~\cite{ctreglees}:
étant donnée une fibration $f:X\to C$ au-dessus d'une courbe de genre quelconque, nous
réduisons les conjectures~$(E_0)$ et~$(E_1)$ pour~$X$ aux mêmes conjectures pour l'espace total de certaines fibrations au-dessus de~$\P^1_k$
construites à partir de~$f$.
Cette réduction repose
sur l'étude de la restriction des scalaires à la Weil de~$f$ le long d'un morphisme
$\pi:C \to \P^1_k$ bien choisi.  Pour traiter la conjecture~$(E_0)$ en toute généralité, nous aurons besoin d'un énoncé de finitude
concernant le groupe de Chow des $0$\nobreakdash-cycles de variétés sur les corps $p$\nobreakdash-adiques issu des résultats récents
de Saito et Sato~\cite{saitosato}.  Des idées inspirées de van Hamel~\cite{vanhamel} nous permettront d'autre part de résoudre les difficultés liées aux places réelles.
Ainsi, grâce à l'article~\cite[\textsection4]{ctsksd98}, qui traite le cas où $C=\P^1_k$,
nous obtiendrons
des résultats pour l'espace total de fibrations au-dessus d'une courbe de genre quelconque
dont les fibres ne sont soumises qu'aux hypothèses générales de \emph{loc.\ cit.}, Theorem~4.1.
Ces résultats contiennent tous ceux cités jusqu'ici au sujet des conjectures~$(E_0)$ et~$(E_1)$;
ils vont également plus loin puisqu'ils s'appliquent par exemple
aux fibrations en quadriques ou en variétés de Severi--Brauer d'indice quelconque au-dessus d'une courbe de genre quelconque, sur un corps de nombres quelconque.

Plus précisément, soit $f:X \to C$ un morphisme dont la fibre générique est géométriquement irréductible,
dont les fibres lisses vérifient le principe de Hasse pour l'existence de
$0$\nobreakdash-cycles de degré~$1$ et dont les fibres singulières
contiennent chacune une composante irréductible de multiplicité~$1$ déployée par
une extension abélienne de leur corps de base.
Si le groupe de Tate--Shafarevich de la jacobienne de~$C$ ne contient pas d'élément infiniment divisible non nul,
nous prouvons la conjecture~$(E_1)$ pour~$X$ (théorème~\ref{th:principalEun}).
Si de plus la fibre générique géométrique~$X_\etabar$ de~$f$ est simplement connexe et vérifie $A_0(X_\etabar \otimes K)=0$ pour tout corps~$K$ algébriquement clos
(hypothèses satisfaites par exemple lorsque~$X_\etabar$ est rationnellement connexe
au sens de Kollár, Miyaoka et Mori~\cite{kollarlivre})
et si les fibres lisses vérifient la propriété d'approximation faible pour les $0$\nobreakdash-cycles de degré~$1$,
nous établissons la conjecture~$(E_0)$ pour~$X$ (théorème~\ref{th:principalE}).
De façon parallèle à~\cite{ctreglees} et~\cite{frossard},
nous obtenons également la version relative de la conjecture~$(E_0)$ sans faire d'hypothèse sur le groupe de Tate--Shafarevich de la jacobienne de~$C$ (théorème~\ref{th:principalErelatif}).
Les résultats de~\cite[\textsection5 et~6]{ctsd94}, \cite[\textsection4]{ctsksd98} et ceux de~\cite{ctreglees}, \cite{frossard}, \cite{vanhamel} sont donc simultanément couverts.

Indiquons brièvement la structure des preuves des théorèmes~\ref{th:principalEun}, \ref{th:principalE} et~\ref{th:principalErelatif}.
La~première difficulté consiste à se réduire à une question portant sur des familles $(z_v)_{v \in \Omega}
\in \prod_{v \in \Omega}Z_0(X\otimes_kk_v)$ de
$0$\nobreakdash-cycles locaux \emph{effectifs}.
Nous devrons aussi nous assurer de l'existence d'un diviseur $y \in Z_0(C)$
tel que les images directes $f_*z_v$ soient linéairement équivalentes à~$y$ et soient toutes situées sur une même
droite~$L$ à l'intérieur de l'espace projectif $|y|=\P(H^0(C,\sO(y)))$, cette droite étant de plus en position générale.
Une fois cela accompli, notons
 $p:W \to \P^1_k$ la restriction des scalaires à la Weil
de~$f$ le long du revêtement $\pi:C \to \P^1_k$ correspondant au pinceau~$L$,
puis $[z_v]\in W(k_v)$, pour $v \in \Omega$, le point défini par le $0$\nobreakdash-cycle effectif~$z_v$.
La~seconde difficulté
consiste à vérifier que
si la famille $(z_v)_{v \in \Omega}$ est orthogonale à~$\Br(X)$ pour l'accouplement de Brauer--Manin,
alors la famille $([z_v])_{v \in \Omega}$ est orthogonale au groupe de Brauer d'une compactification lisse~$W'$ de~$W$.
Nous montrerons seulement que $([z_v])_{v \in \Omega}$ est orthogonale au sous-groupe $\Br_\vert(W/\P^1_k) = \Br(W') \cap p^*\Br(k(\P^1))$, et ce, seulement si~$k$ est totalement imaginaire.
Lorsque le corps de nombres~$k$ est formellement réel,
les idées de~\cite{vanhamel}
nous permettront d'établir cette propriété à condition d'avoir choisi les~$z_v$ aux places~$v$ réelles avec soin au début de l'argument.
De l'orthogonalité de $([z_v])_{v \in \Omega}$ à $\Br_\vert(W/\P^1_k)$, on déduit enfin, grâce à~\cite{ctsksd98}, l'existence d'un $0$\nobreakdash-cycle de degré~$1$ sur~$W'$ vérifiant,
le cas échéant, des conditions locales en un nombre fini de places.  Celui-ci donne formellement naissance
à un $0$\nobreakdash-cycle $z \in Z_0(X)$, soumis aux mêmes conditions locales et tel que~$f_*z$ soit linéairement équivalent à~$y$.
La construction de~$z$ est le point clef des preuves des théorèmes~\ref{th:principalEun}, \ref{th:principalE} et~\ref{th:principalErelatif}.

Le texte est organisé comme suit. Le \textsection\ref{sec:prelim} introduit les notations, énonce les théorèmes principaux et établit
divers lemmes généraux concernant d'une part l'équivalence rationnelle des $0$\nobreakdash-cycles pour une variété projective et lisse
définie sur un corps local
et d'autre part
les groupes abéliens.  Au \textsection\ref{sec:groupechow}, à l'aide des résultats de Saito et Sato~\cite{saitosato}, nous démontrons que si
$C$ est une courbe propre et lisse sur un corps de nombres et si $f:X \to C$ est un morphisme propre
dont la fibre générique géométrique~$X_\etabar$ est irréductible, simplement connexe et vérifie
$A_0(X_\etabar \otimes K)=0$ pour tout corps~$K$ algébriquement clos (ces hypothèses étant notamment satisfaites lorsque~$X_\etabar$ est rationnellement connexe),
alors l'application $f_*:\CH_0(X \otimes_k k_v) \to \CH_0(C \otimes_k k_v)$ est un isomorphisme pour toute place~$v$ hors d'un ensemble fini.
Cette assertion est un ingrédient des preuves des théorèmes~\ref{th:principalE} et~\ref{th:principalErelatif}; en revanche, le théorème~\ref{th:principalEun} n'en dépend pas.
Les démonstrations des théorèmes~\ref{th:principalEun}, \ref{th:principalE} et~\ref{th:principalErelatif} sont données
aux \textsection\ref{sec:reductionP} et~\ref{sec:demPS}, à l'exception de la proposition~\ref{prop:pointorthogonal}, que nous ne
prouvons au \textsection\ref{sec:demPS} que si~$k$ est totalement imaginaire.
Au \textsection\ref{sec:vanhamel},
après avoir développé un formalisme cohomologique inspiré de~\cite{vanhamel}, nous établissons la proposition~\ref{prop:pointorthogonal} sans restriction sur~$k$.

On renvoie aux articles récents de Yongqi Liang pour des progrès subséquents au présent travail.  Liang établit
dans~\cite{liang} une version des théorèmes~\ref{th:principalEun} et~\ref{th:principalE}
pour des fibrations au-dessus de~$\P^n_k$.  D'autre part, dans le cas des fibrations $f:X \to C$
au-dessus d'une courbe de genre quelconque, les
théorèmes~\ref{th:principalEun} et~\ref{th:principalE} contiennent une hypothèse
de nature arithmétique
portant sur les
fibres lisses de~$f$;
en adaptant les preuves de ces deux théorèmes données ci-dessous,
Liang montre dans~\emph{op.\ cit.} que leur conclusion est encore valable si cette hypothèse n'est supposée satisfaite que par les fibres de~$f$ au-dessus d'un ensemble hilbertien généralisé de points fermés de~$C$.

\bigskip
Je suis reconnaissant à Jean-Louis Colliot-Thélène et au rapporteur pour de nombreuses suggestions ayant permis d'améliorer la rédaction de cet article.
Je remercie Yongqi Liang pour une remarque sur la démonstration du théorème~\ref{th:variantectsksd}.

\section{Préliminaires}
\mylabel{sec:prelim}

\subsection{Notations et conjectures}

Soit~$M$ un groupe abélien.
Si~$n$ est un entier, notons respectivement $\tors{n}M$, $nM$ et $M/n$ le noyau, l'image et le conoyau de l'endomorphisme de~$M$ de multiplication par~$n$.
Posons $\widehat{M}=\varprojlim_{n>0} M/n$.
Nous dirons
 que le groupe~$M$ est \emph{divisible} si $M/n=0$ pour tout~$n>0$
et qu'un élément de~$M$ est \emph{infiniment divisible} si son image dans~$\widehat{M}$ est nulle.

Tous les groupes de cohomologie considérés dans cet article sont des groupes de cohomologie étale (ou galoisienne).
Notons $\Br(X)=H^2(X,\Gm)$ le groupe de Brauer cohomologique d'un schéma~$X$.  Rappelons que si~$X$ est intègre et régulier de corps des fonctions~$K$, la flèche de restriction $\Br(X) \to \Br(K)$ est injective
(cf.~\cite[II, Example~2.22]{milneet}); par conséquent~$\Br(X)$ est un groupe de torsion.
Si $f:X\to Y$ est un morphisme entre schémas intègres et réguliers et si~$L$ désigne le corps des fonctions de~$Y$,
notons $\Br_\vert(X/Y)=\Br(X) \cap \mkern2muf^*\mkern-2mu\Br(L)$ le \emph{groupe de Brauer vertical} de~$X$ sur~$Y$, l'intersection étant prise dans~$\Br(K)$.

Soit~$X$ une variété lisse sur un corps~$k$. Notons $Z_0(X)$ le groupe des $0$\nobreakdash-cycles sur~$X$ et $Z_0^\eff(X) \subset Z_0(X)$ l'ensemble des $0$\nobreakdash-cycles effectifs.

Supposons dorénavant~$X$ propre sur~$k$. Notons $\CH_0(X)$ le quotient de $Z_0(X)$ par le sous-groupe des $0$\nobreakdash-cycles rationnellement équivalents à~$0$
et~$A_0(X)$ le noyau de l'application degré $\CH_0(X) \to \Z$ (cf.~\cite[\textsection1]{fulton}).
Il existe un accouplement canonique
\begin{align}
\mylabel{eq:accbrch}
\mylangle -, - \myrangle : \Br(X) \times \CH_0(X) \to \Br(k)
\end{align}
caractérisé par la propriété suivante: pour tout $A \in \Br(X)$ et tout point fermé $P \in X$, la classe $\mylangle A,P\myrangle \in \Br(k)$
est égale à la corestriction de~$k(P)$ à~$k$ de $A(P) \in \Br(k(P))$
(cf.~\cite[Théorème~9]{maninicm}).

Supposons que~$k$ soit un corps de nombres. Notons~$\Omega_f$ (resp.~$\Omega_\infty$) l'ensemble de ses places finies (resp.~infinies) et posons $\Omega=\Omega_f \amalg \Omega_\infty$.
Pour $v \in \Omega$, notons~$k_v$ le complété de~$k$ en~$v$ et~$\kbar_v$ une clôture algébrique de~$k_v$.
Pour $v \in \Omega_f$, notons~$\sO_v$ l'anneau des entiers de~$k_v$ et~$\Fv$ le corps résiduel de~$\sO_v$.
Enfin, pour $S \subset \Omega$, notons~$\sO_S$ l'anneau des $S$\nobreakdash-entiers de~$k$
(éléments de~$k$ entiers en toute place finie hors de~$S$).

Soit $\CHzA(X)$ le groupe
$$
\prod_{v \in \Omega_f} \CH_0(X \otimes_k k_v) \times \prod_{v \in \Omega_\infty}\Coker\mkern-1mu\left(N_{\kbar_v/k_v}\!: \CH_0(X \otimes_k \kbar_v) \to \CH_0(X \otimes_k k_v)\right)\!\rlap{\text{,}}
$$
où $N_{\kbar_v/k_v}$ désigne l'application norme.
Remarquons que les facteurs correspondant aux places infinies sont tués par~$2$
(cf.~\cite[Example~13.12]{fulton})
et que si~$k$ est totalement imaginaire, on a simplement $\CHzA(X)=\prod_{v \in \Omega_f} \CH_0(X\otimes_k k_v)$.
Remarquons enfin que, conformément à la notation employée, le groupe $\CHzA(X)$ ne dépend pas du corps~$k$ mais seulement du schéma~$X$; on aurait pu remplacer~$k$ par~$\Q$ dans la définition ci-dessus.

Pour $v \in \Omega$, notons $\inv_v:\Br(k_v)  \hookrightarrow \Q/\Z$
l'invariant de la théorie du corps de classes local.
La formule $(A,(z_v)_{v \in \Omega}) \mapsto \sum_{v\in \Omega}\inv_v \mylangle A, z_v \myrangle$ définit un accouplement
\begin{align}
\mylabel{eq:brchzaacc}
\Br(X) \times \CHzA(X) \to \Q/\Z
\end{align}
(cf.~\cite[\textsection3.1]{ctsd94}).
Posons $\CHzchapeau(X)=\widehat{\CH_0(X)}$ et $\CHzAchapeau(X)=\widehat{\CHzA(X)}$.
Comme~$\Br(X)$ est de torsion, l'accouplement~(\ref{eq:brchzaacc}) induit à son tour
un accouplement
\begin{align}
\mylabel{eq:brchzacacc}
\Br(X) \times \CHzAchapeau(X) \to \Q/\Z\rlap{\text{.}}
\end{align}

Le groupe $\CHzAchapeau(X)$ intervient dans la formulation, due à van~Hamel (cf.~\cite[Theorem~0.2]{vanhamel}),
d'un énoncé englobant les conjectures~$(E_0)$ et~$(E_1)$ et néanmoins
plausible sur un corps de nombres formellement réel:

\begin{conjE*}
Pour toute variété~$X$ propre et lisse sur un corps de nombres, le complexe
\begin{align}
\mylabel{eq:seconje}
\xymatrix{
\CHzchapeau(X) \ar[r] & \CHzAchapeau(X) \ar[r] & \Hom(\Br(X),\Q/\Z)\rlap{\text{,}}
}
\end{align}
dans lequel la seconde flèche est induite par~(\ref{eq:brchzacacc}),
est une suite exacte.
\end{conjE*}

\begin{rmqs}
\mylabel{rqs:conje}
(i) La conjecture~$(E)$ est vraie pour $X=\Spec(k)$ puisque
le complexe~(\ref{eq:seconje}) s'identifie dans ce cas à celui obtenu en appliquant le foncteur $\Hom(-,\Q/\Z)$ à la suite exacte
\begin{align}
\mylabel{eq:sebrglobal}
\xymatrix{
0 \ar[r] & \Br(k) \ar[r] & \displaystyle\bigoplus_{v \in \Omega} \Br(k_v) \ar[r] & \Q/\Z \ar[r] & 0 \rlap{\text{.}}
}
\end{align}
On peut donc voir la conjecture~$(E)$ comme une généralisation partielle, en dimension supérieure, de la suite exacte~(\ref{eq:sebrglobal}).

(ii) Pour $v \in \Omega_\infty$, quel que soit~$X$,
le groupe $\varprojlim A_0(X \otimes_k k_v)/n$ s'identifie
au conoyau de $N_{\kbar_v/k_v}:A_0(X \otimes_k \kbar_v)\to A_0(X \otimes_k k_v)$ (cf.~\cite[Théorèmes~1.1~(a) et~1.3~(a)(b)]{ctbordeaux}).
D'autre part,
le noyau de l'application degré $\CHzchapeau(X) \to \widehat{\Z}$ s'identifie à~$\widehat{A_0(X)}$.
Il s'ensuit que~$(E)$ implique~$(E_0)$.

(iii) Voici comment déduire~$(E_1)$ de~$(E)$.
Soit $(z_v)_{v \in \Omega}$ une famille de $0$\nobreakdash-cycles de degré~$1$ orthogonale à~$\Br(X)$.
Soit $P \in X$ un point fermé.  Posons $n=\deg(P)$.
Si la conjecture~$(E)$ est vraie, l'image de $(z_v)_{v \in \Omega}$ dans $\CHzAchapeau(X)$ provient de $\CHzchapeau(X)$;
en particulier, il existe
$z \in Z_0(X)$ ayant même image que~$z_v$ dans $\CH_0(X \otimes_k k_v)/n$ pour toute place finie~$v$.
Comme $\deg(z)$ est congru à~$1$ modulo~$n$, une combinaison linéaire de~$z$ et de~$P$ est de degré~$1$.

(iv) Si~$X$ est une courbe
et si le groupe de Tate--Shafarevich de la jacobienne de~$X$ ne contient pas d'élément infiniment divisible non nul,
la conjecture~$(E)$ est vraie pour~$X$, d'après Saito~\cite[(7-1) et~(7-5)]{saitoarithmeticsurfaces} et Colliot-Thélène~\cite[\textsection3]{ctconj}.
Comme cet énoncé ne figure pas explicitement dans la littérature, nous indiquons ici comment le déduire de~\cite{ctconj}.
Soit $\zchapeau_\A \in \CHzAchapeau(X)$ un élément orthogonal à~$\Br(X)$.
La fonctorialité covariante de la suite exacte~(\ref{eq:seconje}) par rapport au morphisme structural $X \to \Spec(k)$ et la remarque~\ref{rqs:conje}~(i)
entraînent l'existence d'un $d \in \widehat{\Z}$ tel que la composante $v$\nobreakdash-adique de~$\zchapeau_\A$ soit de degré~$d$ (resp.~$d$ modulo~$2$)
pour toute place~$v$ finie (resp.~réelle).  Soit $P \in X$ un point fermé.
Quitte à remplacer $\zchapeau_\A$ par $\zchapeau_\A + rP$ pour un $r \in \widehat{\Z}$ bien choisi, on peut supposer que $d \in \Z$.
D'après le rappel~\ref{rappel:chcourbe} ci-dessous, il existe alors une famille $(z_v)_{v \in \Omega} \in \prod_{v \in \Omega} \CH_0(X\otimes_k k_v)$ ayant
$\zchapeau_\A$
pour image
dans $\CHzAchapeau(X)$.  La démonstration de la proposition~3.3 de \emph{op.\ cit.} appliquée à $(z_v)_{v \in \Omega}$ fournit maintenant un $\zchapeau \in \CHzchapeau(X)$
qui s'envoie sur~$\zchapeau_\A$ dans $\CHzAchapeau(X)$.

\gdef\citemilneadtrappel{\cite[Ch.~I, Lemma~3.3]{milneadt}}
\begin{rappel}[{\rm{cf.}}~\citemilneadtrappel]
\mylabel{rappel:chcourbe}
Soit~$C$ une courbe propre et lisse sur un corps $p$\nobreakdash-adique.
Le sous-groupe de torsion de $\CH_0(C)$ est fini.
L'application $\CH_0(C) \to \CHzchapeau(C)$ est injective et son image coïncide avec
l'image réciproque de~$\Z$ par l'application degré $\CHzchapeau(C) \to \widehat{\Z}$.
\end{rappel}

(v) Si~$X$ est une courbe, l'application $\CHzchapeau(X) \to \CHzAchapeau(X)$ est injective.
En effet, notant~$J$ la jacobienne de~$X$, la flèche
$$
\widehat{J(k)} \longrightarrow \prod_{v \in \Omega_f} \widehat{J(k_v)}
$$
est injective
d'après~\cite[Ch.~I, Corollary~6.23~(b)]{milneadt}; de plus, comme~$J(k)$ est un groupe abélien de type fini (théorème de Mordell--Weil), la flèche
$\widehat{A_0(X)} \to \widehat{J(k)}$ induite par
l'inclusion $A_0(X) \subset J(k)$ est elle aussi injective.
En revanche, en dimension supérieure,
l'application $\CHzchapeau(X) \to \CHzAchapeau(X)$
peut avoir un noyau non nul, même lorsque~$X$ est une surface fibrée en coniques au-dessus de~$\P^1_k$
 (cf.~(\ref{eq:sects}) et~\cite{salberger}).

(vi) Si~$k$ est un corps de caractéristique~$0$, les groupes~$\Br(V)$ et~$\CH_0(V)$ sont des invariants birationnels des variétés~$V$ propres et lisses sur~$k$
(cf.~\cite[\textsection7]{grbr3}, \cite[Example~16.1.11]{fulton}).
Les conjectures~$(E)$, $(E_0)$ et~$(E_1)$ ne dépendent donc de~$X$ qu'à équivalence birationnelle près.
\end{rmqs}

\subsection{Énoncés principaux}

Cet article est consacré à la démonstration des trois théorèmes suivants.

\begin{thm}
\mylabel{th:principalEun}
Soit~$X$ une variété irréductible, propre et lisse sur un corps de nombres~$k$,
munie d'un morphisme $f:X \to C$
de fibre générique géométriquement irréductible, où~$C$ est une courbe lisse et géométriquement irréductible.
Supposons que le groupe de Tate--Shafarevich de la jacobienne de~$C$ ne contienne pas d'élément infiniment divisible non nul.
Supposons de plus
les deux hypothèses suivantes satisfaites
pour tout point fermé $c \in C$:
\begin{enumerate}
\item[(a)] la fibre~$X_c$ possède
une composante irréductible~$Y$ de multiplicité~$1$ dans le corps des fonctions de laquelle la fermeture algébrique de~$k(c)$ est une extension abélienne
de~$k(c)$;
\item[(b)] si~$X_c$ est lisse et si $X_c(\A_{k(c)})\neq \emptyset$ alors~$X_c$ admet un $0$\nobreakdash-cycle de degré~$1$ sur~$k(c)$.
\end{enumerate}
Alors~$X$ vérifie l'énoncé de la conjecture~$(E_1)$.
\end{thm}

\begin{thm}
\mylabel{th:principalE}
Mêmes hypothèses que dans le théorème~\ref{th:principalEun}, à ceci près que~(b) est renforcée en:
\begin{enumerate}
\item[(b')] si~$X_c$ est lisse, alors pour tout $n>0$ et tout ensemble fini~$S$ de places de~$k(c)$,
l'image de $\CH_0(X_c)$ dans $\prod_{w \in S} \CH_0(X_c \otimes_{k(c)} k(c)_w)/n$
contient l'image de $X_c(\A_{k(c)})$ dans ce même groupe.
\end{enumerate}
Supposons de plus que $A_0(X_\etabar \otimes K)=0$ et $H^1(X_\etabar,\Q/\Z)=0$,
où~$K$ désigne une clôture algébrique du corps des fonctions de la fibre générique géométrique~$X_\etabar$ de~$f$.
Alors~$X$ vérifie l'énoncé de la conjecture~$(E)$.
\end{thm}

\begin{rmqs}
\mylabel{rq:hypothesesth}
(i) Un argument similaire à celui de la remarque~\ref{rqs:conje}~(iii) montre que l'hypothèse~(b') implique~(b) (prendre pour~$n$ le degré d'un point fermé de~$X_c$).

(ii) Il résulte du lemme~\ref{lem:continuite} ci-dessous que
si $X_c(k(c))$ est dense dans $X_c(\A_{k(c)})$ (propriété d'approximation faible), alors
l'hypothèse~(b') est satisfaite.

(iii) L'hypothèse~(b') devrait aussi être satisfaite par des classes de variétés qui ne vérifient pas l'approximation faible.
Ce devrait par exemple être le cas si~$X_c$ est une intersection complète lisse de dimension~$\geq 3$ dans un espace projectif.
En effet, l'application naturelle $\Br(k(c)) \to \Br(X_c)$ est alors surjective (cf.~\cite[Appendix~A]{poonenvoloch}), de sorte
que l'hypothèse~(b') serait impliquée par la validité de la conjecture~$(E)$ pour~$X_c$.  D'après les conjectures de Lang, de telles variétés
ne devraient pas vérifier l'approximation faible si leur degré n'est pas trop petit (cf.~\cite[Conjecture~5.7]{langconj}).

(iv)
Si la fibre générique géométrique~$X_\etabar$ est rationnellement connexe (par exemple si elle est unirationnelle, cf.~\cite[Chapter~4]{debarrehigherdim}),
les conditions $A_0(X_\etabar \otimes K)=0$ et $H^1(X_\etabar,\Q/\Z)=0$ sont satisfaites (cf.~\emph{op.\ cit.}, Corollary~4.18~(b) pour la seconde).

(v)
D'après la remarque~\ref{rqs:conje}~(vi) et le théorème des fonctions implicites, les conditions~(b) et~(b') des théorèmes~\ref{th:principalEun}
et~\ref{th:principalE} sont des invariants birationnels de la variété propre et lisse~$X_c$.
De plus, les groupes $A_0(X_\etabar \otimes K)$ et $H^1(X_\etabar,\Q/\Z)$ sont des invariants birationnels de la variété propre et lisse~$X_\etabar$
(cf.~remarque~\ref{rqs:conje}~(vi) et \cite[Exp.~X, Corollaire~3.4 et Exp.~XI, \textsection5]{sga1}).
Notons que la condition $H^1(X_\etabar,\Q/\Z)=0$ équivaut à ce que $\Pic(X_\etabar)$ soit sans torsion.
\end{rmqs}

Conjecturalement, le groupe de Tate--Shafarevich d'une variété abélienne sur un corps de nombres est toujours fini
et ne contient donc jamais d'élément infiniment divisible non nul.
Des calculs numériques permettent de vérifier sur de nombreux exemples de courbes~$C$ qu'au moins le sous-groupe de torsion $\ell$\nobreakdash-primaire du groupe de Tate--Shafarevich
de la jacobienne de~$C$ est fini, pour de petits nombres premiers~$\ell$.  Lorsque cette condition est satisfaite pour un~$\ell$ fixé,
la preuve du théorème~\ref{th:principalE} entraîne encore, inconditionnellement, l'exactitude du complexe obtenu en remplaçant dans~(\ref{eq:seconje})
les groupes $\CHzchapeau(X)$ et $\CHzAchapeau(X)$ par les complétés $\ell$\nobreakdash-adiques de $\CH_0(X)$ et de $\CHzA(X)$.
Il convient d'autre part de souligner que l'absence
d'élément infiniment divisible non nul dans le groupe de Tate--Shafarevich de la jacobienne de~$C$
est une condition \emph{nécessaire} à la validité de la conjecture~$(E)$ pour~$X=C$,
du moins lorsque $C(k)\neq\emptyset$ (cf.~\cite[Ch.~I, Theorem~6.26~(b)]{milneadt}).

En l'absence de toute hypothèse sur le groupe de Tate--Shafarevich de la jacobienne de~$C$, nous démontrerons:

\begin{thm}
\mylabel{th:principalErelatif}
Soit~$X$ une variété irréductible, propre et lisse sur un corps de nombres~$k$,
munie d'un morphisme $f:X \to C$
de fibre générique géométriquement irréductible, où~$C$ est une courbe lisse et géométriquement irréductible.
Supposons l'hypothèse~(a) du théorème~\ref{th:principalEun}
et l'hypothèse~(b') du théorème~\ref{th:principalE} satisfaites
pour tout point fermé $c \in C$.
Supposons de plus que $A_0(X_\etabar \otimes K)=0$ et $H^1(X_\etabar,\Q/\Z)=0$.
Alors le complexe
\begin{align}
\mylabel{diag:complexeErelatif}
\xymatrix{
\CH_0(X/C) \ar[r] & \CHzA(X/C) \ar[r] & \Hom(\Br(X)/f^*\Br(C),\Q/\Z)
}
\end{align}
est une suite exacte, où $\CH_0(X/C)$ et $\CHzA(X/C)$ désignent les noyaux respectifs de $f_*:\CH_0(X) \to \CH_0(C)$ et de $f_*:\CHzA(X) \to \CHzA(C)$.
\end{thm}

Le corollaire suivant généralise d'une part~\cite[Theorems~5.1, 6.2]{ctsd94} et d'autre part~\cite[Théorèmes~0.3, 0.4, 0.5]{frossard}, \cite[Theorem~0.2, Corollary~0.3]{vanhamel}:

\begin{cor}
\mylabel{cor:corthpr}
Soit~$X$ une variété propre et lisse sur un corps de nombres~$k$.
Supposons que~$X$ admette une structure de
fibration en variétés de Severi--Brauer,
en quadriques ou encore en variétés de Severi--Brauer généralisées au sens de~\cite[\textsection2]{ctsd94},
au-dessus d'une courbe~$C$ propre et lisse sur~$k$.
Alors le complexe~(\ref{diag:complexeErelatif}) est une suite exacte.
Si de plus le groupe de Tate--Shafarevich de la jacobienne de~$C$ ne possède pas d'élément infiniment
divisible non nul, les conjectures~$(E)$, $(E_1)$ et~$(E_0)$ valent pour~$X$.
\end{cor}

\begin{proof}
L'hypothèse~(a) du théorème~\ref{th:principalEun} est satisfaite d'après~\cite[p.~117--118]{owlnm}.
Les autres hypothèses du théorème~\ref{th:principalErelatif} sont satisfaites en vertu des
remarques~\ref{rq:hypothesesth}~(ii) et~(iv) et de~\cite[Proposition~2.1.4~(i), 2.1.5]{ctsd94}.
\end{proof}

\subsection{Un lemme de continuité pour l'équivalence rationnelle des zéro-cycles sur un corps local}

Si~$X$ est une variété sur un corps~$k$, notons $\Sym_{X/k}$ la réunion disjointe des produits symétriques $\smash[t]{\Sym^d_{X/k}}$ pour $d\geq 1$.
Le lemme suivant est probablement bien connu; ne l'ayant pas trouvé dans la littérature, nous en fournissons une preuve ci-dessous.
Il~servira au paragraphe~\ref{sec:preuveP}.

\begin{lem}
\mylabel{lem:continuite}
Soit~$k$ une extension finie de~$\Q_p$ ou de~$\R$.
Soit~$X$ une variété projective et lisse sur~$k$.
Pour tout entier $n>0$, l'application
$\Sym_{X/k}(k) \to \CH_0(X)/n$
qui à un $0$\nobreakdash-cycle effectif sur~$X$ associe sa classe dans $\CH_0(X)/n$ est localement constante
(si l'on munit $\Sym_{X/k}(k)$ de la topologie induite par celle de~$k$).
\end{lem}

\begin{proof}
Nous allons établir le lemme par récurrence sur la dimension de~$X$.  Si $\dim(X)=0$, il n'y a rien à démontrer.  Supposons que $\dim(X)\geq 1$.
Fixons un plongement $X \subset \P^N_k$ et un $0$\nobreakdash-cycle effectif $z \in Z_0^\eff(X)$.
Notons $Z \subset X$ le support de~$z$, vu comme schéma réduit.
D'après Altman et Kleiman~\cite[(7)]{altmankleiman}, on peut supposer,
quitte à remplacer le plongement donné $X \subset \P^N_k$ par sa composée avec un plongement de Veronese de degré assez élevé,
qu'il existe un sous-espace linéaire $L \subset \P^N_k$ de codimension~$\dim(X)-1$, contenant~$Z$,
tel que le schéma $L \cap X$ soit une courbe lisse.  Soit $D \subset L$ un sous-espace linéaire de codimension~$1$ dans~$L$, disjoint de~$Z$,
tel que le schéma $D \cap X$ soit étale sur~$k$.
Soit $H \subset \P^N_k$ un sous-espace linéaire de dimension~$\dim(X)-1$ disjoint de~$D$.
Notons $\pi: X' \to X$ la variété obtenue en faisant éclater $D \cap X$ dans~$X$
et $p:X' \to H$ le morphisme de projection de centre~$D$ dans~$\P^N_k$.
Ses fibres sont les sections de~$X$ par les sous-espaces linéaires de~$\P^N_k$ de codimension~$\dim(X)-1$ contenant~$D$.
Il existe donc un point $h \in H(k)$ tel que
$L \cap X = p^{-1}(h)$.
Comme $L \cap X$ est une courbe lisse, il existe un ouvert $V \subset H$ contenant~$h$
tel que le morphisme $p^{-1}(V) \to V$ induit par~$p$
soit projectif, lisse, de dimension relative~$1$.
La composante neutre du foncteur de Picard relatif $\mathbf{Pic}^0_{p^{-1}(V)/V}$ est alors représentée par un schéma abélien $J \to V$
(cf.~\cite[8.4/3 et~8.4/4]{blr}).

Comme~$p$ est lisse aux points de~$Z$ et que~$Z$ est disjoint de~$D$,
il existe une sous-variété fermée $F \subset p^{-1}(V)$ contenant~$Z$, disjointe de~$D$ et étale sur~$V$ aux points de~$Z$
(cf.~\cite[p.~193]{ega44}).  Quitte à rétrécir~$V$, on peut supposer~$F$ étale sur~$V$,
auquel cas~$F$ est une variété lisse.  Fixons, à l'aide de Hironaka,
une compactification lisse $F \subset \hat{F}$ telle que l'inclusion $F \subset X'$ se prolonge en un morphisme $\hat{F} \to X'$.
La variété~$\hat{F}$ est de dimension $\dim(X)-1$; par hypothèse de récurrence,
il existe donc un ouvert $\sU_F \subset \Sym_{F/k}(k) \subset \Sym_{\hat{F}/k}(k)$ contenant le point~$z$, tel que l'application
$\Sym_{\hat{F}/k}(k) \to \CH_0(\hat{F})/n$ soit constante sur~$\sU_F$.
L'application $\sU_F \to \CH_0(X)/n$ obtenue par composition avec le morphisme d'image directe $\CH_0(\hat{F})/n \to \CH_0(X)/n$ est donc elle aussi
constante.

Soit $K/k$ une extension finie galoisienne dans laquelle se plongent toutes les extensions finies de degré~$\leq d$ de~$k$, où $d=\deg(z)$.
Posons $G=\Gal(K/k)$.
Comme~$F$ est étale sur~$V$, le théorème d'inversion locale~\cite[Part~II, Ch.~III, \textsection9, Theorem~2]{serreharvard}
entraîne l'existence d'un voisinage ouvert $\sV \subset V(K)$ de~$h$
et, pour chaque $x \in Z(K)$, d'un voisinage ouvert $\sB_x \subset X(K)$ de~$x$,
tels que les
ensembles~$\sB_x$ soient deux à deux disjoints
et que les
applications $\sB_x \cap F(K) \to \sV$ induites par~$p$ soient des isomorphismes de variétés analytiques.
Notons $s_x : \sV \to \sB_x \cap F(K)$ les isomorphismes inverses.
Quitte à rétrécir~$\sV$ et à remplacer chaque~$\sB_x$ par l'intersection des $\sigma^{-1}(\sB_{\sigma(x)})$ pour $\sigma\in G$,
on peut supposer que $\sigma(\sB_x)=\sB_{\sigma(x)}$ pour tout $\sigma \in G$ et tout $x \in Z(K)$.
Alors~$\sV$ et la réunion~$\sB$ des~$\sB_x$ sont stables sous l'action de~$G$,
chaque~$\sB_x$ est stable sous l'action du stabilisateur $G_x \subset G$ de~$x$ et enfin l'application~$s_x$ est $G_x$\nobreakdash-équivariante
puisque~$s_x^{-1}$ l'est.
Si $k' \subset K$ est un sous-corps contenant~$k$, notons $nJ(k')$ l'image de~$J(k')$ par l'endomorphisme de multiplication par~$n$ de $J \to V$.
Cet endomorphisme étant étale, le théorème d'inversion locale entraîne que~$nJ(k')$ est un ouvert de~$J(k')$.
Ainsi, quitte à rétrécir à nouveau les ouverts~$\sB_x$ et~$\sV$, on peut supposer que pour tout
corps intermédiaire $k \subset k' \subset K$ et tout $x \in Z(k')$,
l'application continue $\sB_x \cap X(k') \to J(k')$ qui à~$b$ associe $b-s_x(p(b))$ est à valeurs dans $nJ(k')$.  Cette application est bien définie parce que~$s_x$
est $G_x$\nobreakdash-équivariante.

Notons $\phi:\sB \to F(K)$ la réunion des applications $s_x \circ p : \sB_x \to F(K)$.
C'est une application continue et $G$\nobreakdash-équivariante.
Elle induit donc une application continue et $G$\nobreakdash-équivariante
$\Sym^d(\phi):\Sym^d(\sB) \to \Sym^d(F(K))$ entre les produits symétriques de ces espaces topologiques.
Comme les sous-espaces de $\Sym^d(X(K))$ et de $\Sym^d(F(K))$ constitués des éléments invariants par~$G$ s'identifient respectivement
à $\Sym^d_{X/k}(k)$ et à $\Sym^d_{F/k}(k)$,
l'application $\Sym^d(\phi)$ induit
à son tour une application continue $\psi: \sU_0 \to \Sym^d_{F/k}(k)$,
où $\sU_0 \subset \Sym^d_{X/k}(k)$ désigne l'ensemble des éléments de $\Sym^d(\sB)$ invariants par~$G$.
Posons finalement $\sU = \psi^{-1}(\sU_F)$.

L'ensemble~$\sU$ est un ouvert de $\Sym_{X/k}(k)$ contenant le point~$z$.
Vérifions que pour tout $a \in \sU$, la classe de~$a$ dans $\CH_0(X)/n$ est égale à celle de~$z$.
Pour tout $a \in \sU$, on a l'égalité de $0$\nobreakdash-cycles sur~$X$
$$a-z = \left(a - \psi(a)\right) + \left(\psi(a)-z\right)\!\mkern1mu\rlap{\text{.}}$$
Comme $\psi(a)$ et~$z$ appartiennent à~$\sU_F$, la classe dans $\CH_0(X)/n$ de $\psi(a)-z$ est nulle.
Quant au~$0$\nobreakdash-cycle $a-\psi(a)$, il s'écrit comme une somme de cycles sur~$X$ de la forme $N_{k'/k}(b-s_x(p(b)))$
pour divers corps intermédiaires $k \subset k' \subset K$, divers $x \in Z(k')$ et divers $b \in \sB_x \cap X(k')$,
où~$N_{k'/k}$ désigne la norme de~$k'$ à~$k$.
Or le cycle $b-s_x(p(b))$ est supporté par la courbe propre et lisse $p^{-1}(p(b))$ et sa classe dans
$\Pic(p^{-1}(p(b)))$ est divisible par~$n$ par construction; sa classe dans $\CH_0(X \otimes_k k')$ est donc elle aussi divisible par~$n$.
Il s'ensuit que la classe de $a-\psi(a)$ dans $\CH_0(X)/n$ est nulle.
\end{proof}

\Subsection{Quelques lemmes sur les groupes abéliens}

\medskip
\begin{lem}
\mylabel{lem:chapeausurjectif}
Soit $f:A\to B$ un homomorphisme de groupes abéliens.
Si~$f$ est surjectif, alors l'homomorphisme $\fchapeau:\widehat{A} \to \widehat{B}$ induit par~$f$ l'est aussi.
\end{lem}

\begin{proof}
Notons $K_n$ le noyau de $f/n:A/n\to B/n$.
On vérifie sans peine que les morphismes de transition du système projectif formé par les groupes~$K_n$ sont surjectifs.
L'exactitude de la suite
$$
\xymatrix{
0 \ar[r] & K_n \ar[r] & A/n \ar[r] & B/n \ar[r] & 0
}
$$
est donc préservée par passage à la limite projective.
\end{proof}

\begin{lem}
\mylabel{lem:divisible}
Soit~$M$ un groupe abélien.  Le conoyau de la flèche naturelle $M \to \widehat{M}$ est divisible.
\end{lem}

\begin{proof}
Soit $n$ un entier~$>0$.
Notons respectivement $s:M \twoheadrightarrow nM$, $i:nM \hookrightarrow M$ et $p:\widehat{M}\to M/n$ la surjection, l'injection et la projection canoniques.
D'après le lemme~\ref{lem:chapeausurjectif}, le morphisme $\widehat{s} : \widehat{M} \to \widehat{nM}$ induit par~$s$
est surjectif.  Compte tenu de
l'exactitude de la suite
$$
\xymatrix{
\widehat{nM} \ar[r]^{\widehat{i}} & \widehat{M} \ar[r]^(.43)p & M/n
}
$$
et de l'égalité $\widehat{i} \circ \widehat{s} = \widehat{i \circ s} =n$,
il s'ensuit que $\widehat{M}/n$ s'injecte dans $M/n$.  Cela implique finalement que $\Coker(M \to \widehat{M})/n=0$.
\end{proof}

\begin{lem}
\mylabel{lem:conoyauchapeau}
Soit $f:A\to B$ un homomorphisme de groupes abéliens dont le conoyau est d'exposant fini.
Notant $\fchapeau:\widehat{A} \to \widehat{B}$ l'homomorphisme induit par~$f$,
la flèche naturelle $\Coker(f) \to \Coker(\fchapeau)$ est injective.
\end{lem}

\begin{proof}
En effet, si $C=\Coker(f)$, on a un diagramme commutatif
$$
\xymatrix@R=3.5ex{
& B \ar[d] \ar[r] & C \ar[d]^(.4)\wr \\
\widehat{A} \ar[r] & \widehat{B} \ar[r] & \widehat{C}
}
$$
où la seconde ligne est un complexe.
\end{proof}

\begin{lem}
\mylabel{lem:noyauchapeau}
Soit $f:A \to B$ un homomorphisme de groupes abéliens dont le noyau et le conoyau sont d'exposant fini.
Si $\tors{n}B$ est fini pour tout $n>0$, la flèche naturelle $\Ker(f) \to \Ker(\fchapeau)$ est surjective.
\end{lem}

\begin{proof}
Notons $I=\Im(f)$, $K=\Ker(f)$ et $C=\Coker(f)$.
Pour $n>0$, le lemme du serpent fournit des suites exactes
\begin{align}
\mylabel{eq:snknanin}
\xymatrix{
0 \ar[r] & S_n \ar[r] & K/n \ar[r] & A/n \ar[r] & I/n
}
\end{align}
et
\begin{align}
\mylabel{eq:tnncinbn}
\xymatrix{
0 \ar[r] & T_n \ar[r] & \tors{n}C \ar[r] & I/n \ar[r] & B/n
}
\end{align}
où~$S_n$ et~$T_n$ sont des sous-quotients de~$\tors{n}B$.  Comme $\tors{n}B$ est fini, il en va de même
des groupes~$S_n$ et~$T_n$ pour tout~$n$; ainsi le passage à la limite projective préserve-t-il l'exactitude de ces suites
(cf.~\cite[Proposition~3.5.7]{weibel}).
D'autre part, on a $\widehat{K}=K$ et $\varprojlim \tors{n}C=0$
puisque~$K$ et~$C$ sont d'exposant fini.
De~(\ref{eq:snknanin}) et de~(\ref{eq:tnncinbn})
on tire donc, à la limite, l'exactitude des suites $K \to \widehat{A} \to \widehat{I}$
et $0 \to \widehat{I} \to \widehat{B}$; d'où l'exactitude de la suite $K \to \widehat{A} \to \widehat{B}$.
\end{proof}

\begin{rmq}
On ne peut supprimer l'hypothèse que~$\tors{n}B$ est fini pour tout $n>0$ dans le lemme~\ref{lem:noyauchapeau}.
En effet, si $A = \bigoplus_{n \geq 1} \Z/2^n\Z$, $B=2A$ et si~$f$ est l'application induite par la multiplication par~$2$ sur~$A$,
alors $\Ker(f)=\bigoplus_{n \geq 1} \Z/2\Z$ mais $\Ker(\fchapeau)=\prod_{n\geq 1}\Z/2\Z$.
\end{rmq}

\section{Groupe de Chow des zéro-cycles d'une fibration en variétés rationnellement connexes au-dessus d'une courbe}
\mylabel{sec:groupechow}

\smallskip
Si~$X$ est une variété projective, lisse et rationnellement connexe sur un corps de nombres~$k$,
Kollár et Szabó~\cite[Corollary~9]{kollarszabo} ont démontré qu'il existe un ensemble fini $S \subset \Omega$ tel que
$A_0(X \otimes_k k_v)=0$ pour toute place $v \notin S$.
Dans ce paragraphe, nous établissons un énoncé analogue pour les fibrations
en variétés rationnellement connexes au-dessus d'une courbe.
Cet énoncé interviendra dans les preuves des théorèmes~\ref{th:principalE} et~\ref{th:principalErelatif}, qui concernent la conjecture~$(E)$.
Le lecteur intéressé uniquement par la conjecture~$(E_1)$ peut passer directement au \textsection\ref{sec:reductionP}.

\begin{thm}
\mylabel{th:isochow}
Soient~$k$ un corps de nombres et $f:X \to C$ un morphisme de fibre générique géométriquement irréductible entre variétés propres, géométriquement irréductibles et lisses sur~$k$.
Supposons que~$C$ soit une courbe.
Supposons de plus que $A_0(X_\etabar \otimes K)=0$ et $H^1(X_\etabar,\Q/\Z)=0$,
où~$K$ désigne une clôture algébrique du corps des fonctions de la fibre générique géométrique~$X_\etabar$ de~$f$.
Alors il existe un ensemble fini de places $S \subset \Omega$ tel que l'application
$$f_* : \CH_0(X \otimes_k k_v) \to \CH_0(C \otimes_k k_v)$$
soit un isomorphisme pour tout $v \notin S$.
\end{thm}

D'après la remarque~\ref{rq:hypothesesth}~(iv), il s'ensuit en particulier:

\begin{cor}
Soit $f:X \to C$ un morphisme dominant entre variétés propres, irréductibles et lisses sur un corps de nombres~$k$.
Si~$C$ est une courbe et si la fibre générique géométrique de~$f$ est rationnellement connexe,
il existe un ensemble fini de places $S \subset \Omega$ tel que l'application
$$f_* : \CH_0(X \otimes_k k_v) \to \CH_0(C \otimes_k k_v)$$
soit un isomorphisme pour tout $v \notin S$.
\end{cor}

L'injectivité de $f_*:\CH_0(X\otimes_k k_v) \to \CH_0(C\otimes_k k_v)$ pour toute place~$v$ hors d'un ensemble fini était connue,
avant la publication de~\cite{kollarszabo} et de~\cite{saitosato},
dans le cas des fibrations en quadriques de dimension~$2$ et des fibrations en variétés de Severi--Brauer d'indice sans facteur carré
(cf.~\cite[Theorem~3.4~(b)]{ctbordeaux}, \cite[Théorème~4.8]{frossardchow}).

Le lemme suivant donne un renseignement sur le noyau de~$f_*$ aux places $v \in S$.
Il~servira dans les preuves des théorèmes~\ref{th:isochow} et~\ref{prop:reduc}.

\begin{lem}
\mylabel{lem:decompdiag}
Soit $f:X \to Y$ un morphisme propre et dominant entre variétés irréductibles et lisses sur un corps~$k$.
Supposons que la fibre générique géométrique~$X_{\etabar}$ de~$f$ soit irréductible et lisse
et vérifie $A_0(X_{\etabar} \otimes K)=0$, où~$K$ désigne une clôture algébrique de son corps des fonctions.
Alors il existe un entier $n>0$ tel que pour toute extension~$k'/k$,
le noyau de $f_*:\CH_0(X \otimes_k k') \to \CH_0(Y \otimes_k k')$ soit annulé par~$n$.
\end{lem}

La preuve du lemme~\ref{lem:decompdiag} consiste à adapter en famille l'argument de \og{}décomposition de la diagonale\fg{} (cf.~\cite[p.~1.20]{blochlectures}, \cite[Proposition~11]{ctfinitudechow}).

\begin{proof}
Soit $Z \subset X$ une sous-variété fermée génériquement finie sur~$Y$, de degré $d>0$.  Notons $\pi:Z \to Y$ la restriction de~$f$ à~$Z$
et~$Z_\eta$ (resp.~$X_\eta$) la fibre générique de~$\pi$ (resp.~de~$f$).  Ainsi~$Z_\eta$ est un $0$\nobreakdash-cycle de degré~$d$ sur~$X_\eta$.

Soit~$K_0$ le corps des fonctions de~$X_\eta$.
Comme $A_0(X_\etabar \otimes K)=0$,
tout $0$\nobreakdash-cycle de degré~$0$ sur $X_\eta \otimes K_0$ devient rationnellement équivalent à~$0$ sur une extension finie de~$K_0$.
Il s'ensuit que le groupe $A_0(X_\eta \otimes K_0)$ est de torsion (cf.~\cite[Example~1.7.4]{fulton}).
Notant
 $\Delta_\eta \subset X_\eta \times_\eta X_\eta$ la diagonale de~$X_\eta$
et $j: X_\eta \otimes K_0 \to X_\eta \times_\eta X_\eta$ l'inclusion, dans $X_\eta \times_\eta X_\eta$, de la fibre générique de la seconde projection
$X_\eta\times_\eta X_\eta \to X_\eta$,
il existe donc $n>0$ tel que $nj^*(d \Delta_\eta - (Z_\eta \times_\eta X_\eta))$ soit rationnellement équivalent à~$0$ sur $X_\eta \otimes K_0$.
Autrement dit, il existe un
 ouvert dense~$U_\eta$ de~$X_\eta$
et un cycle~$w_\eta$ sur $X_\eta \times_\eta X_\eta$, supporté par le complémentaire
de $X_\eta \times_\eta U_\eta$, tels que l'on ait une
équivalence rationnelle
$$nd \Delta_\eta \sim n (Z_\eta \times_\eta X_\eta) + w_\eta$$
sur $X_\eta \times_\eta X_\eta$.
Cette équivalence rationnelle s'étend et se spécialise au-dessus d'un voisinage de~$\eta$ dans~$Y$:
il existe des ouverts denses $Y^0 \subset Y$ et $U \subset X$ tels que le morphisme~$f$ (resp.~$\pi$) soit lisse (resp.~fini et plat) au-dessus de~$Y^0$ et tels que
pour tout $y \in Y^0$, l'ouvert $U_y \subset X_y$ soit dense et l'on ait une équivalence rationnelle
\begin{align*}
nd \Delta_y \sim n (Z_y \times_{k(y)} X_y) + w_y
\end{align*}
sur $X_y \times_{k(y)} X_y$, où $Z_y=\pi^{-1}(y)$,
où $\Delta_y$ est la diagonale de $X_y \times_{k(y)} X_y$ et où~$w_y$ est un cycle supporté par le complémentaire
de $X_y \times_{k(y)} U_y$.
Cette relation implique
que pour tout $y \in Y^0$ et toute extension $k''/k(y)$,
le groupe $A_0(X_y \otimes_{k(y)} k'')$ est annulé par~$nd$; en effet,
la correspondance $nd \Delta_y - n (Z_y \times_{k(y)} X_y) - w_y$ agit sur
$\CH_0(X_y \otimes_{k(y)} k'')$ par $z \mapsto ndz - n\deg(z)[Z_y]$.

Soient~$k'/k$ une extension et~$z$ un $0$\nobreakdash-cycle sur $X \otimes_k k'$ tel que $f_*z \sim 0$.
Comme~$X$ est lisse, 
un lemme de déplacement bien connu
permet de supposer le support de~$z$
inclus dans $f^{-1}(Y^0 \otimes_k k')$,
quitte à remplacer~$z$ par un $0$\nobreakdash-cycle qui lui est rationnellement équivalent (cf.~\cite[p.~599]{ctfinitudechow}).
Le cycle $dz - \pi^* f_* z$ sur $X \otimes_k k'$ s'écrit alors comme une somme de $0$\nobreakdash-cycles de degré~$0$
supportés par $X_y \otimes_{k(y)} k''$ pour divers $y \in Y^0$ et diverses extensions $k''/k(y)$.
Il~s'ensuit que $nd(dz - \pi^* f_* z)$ est rationnellement équivalent à~$0$ sur $X \otimes_k k'$.
Or $\pi^* f_* z$ est lui-même rationnellement équivalent à~$0$ puisque $f_* z \sim 0$ (cf.~\cite[8.1]{fulton});
d'où $nd^2 z \sim 0$ sur $X \otimes_k k'$. Ainsi l'entier $nd^2$ annule-t-il le noyau de $f_* : \CH_0(X \otimes_k k') \to \CH_0(Y \otimes_k k')$.
\end{proof}

L'assertion analogue pour le conoyau de~$f_*$ ne présente aucune difficulté:

\begin{lem}
\mylabel{lem:conoyaumultisection}
Soit $f:X \to Y$ un morphisme propre et dominant entre variétés irréductibles et lisses sur un corps~$k$.
Il existe un entier $n>0$ tel que pour toute extension~$k'/k$,
le conoyau de $f_*:\CH_0(X \otimes_k k') \to \CH_0(Y \otimes_k k')$ soit annulé par~$n$.
\end{lem}

\begin{proof}
Soit $Z \subset X$ une sous-variété fermée irréductible génériquement finie sur~$Y$.  Notons $n>0$ le degré de~$Z$ sur~$Y$.
D'après la formule de projection (cf.~\cite[Example~8.1.7]{fulton}), on a $f_*(f^*y \cdot [Z])=y \cdot f_*[Z]=ny$ dans $\CH_0(Y \otimes_k k')$ pour tout $y \in \CH_0(Y \otimes_k k')$.
\end{proof}

Nous sommes maintenant en position d'établir le théorème~\ref{th:isochow}.

\begin{proof}[Démonstration du théorème~\ref{th:isochow}]
Grâce au lemme de Chow (cf.~\cite[5.6.1]{ega2}) et à Hironaka, on peut supposer~$X$ projective sur~$k$.
Les hypothèses et la conclusion du théorème~\ref{th:isochow} sont en effet
des invariants birationnels pour les variétés propres et lisses
(cf.~remarques~\ref{rqs:conje}~(vi) et~\ref{rq:hypothesesth}~(v)).

D'après les lemmes~\ref{lem:decompdiag} et~\ref{lem:conoyaumultisection}, il existe un entier $n>0$ annulant le noyau et le conoyau de
$f_*:\CH_0(X \otimes_k k_v) \to \CH_0(C \otimes_k k_v)$
pour toute place~$v$.
Soit~$S$ un ensemble fini de places de~$k$ contenant les places archimédiennes et les places divisant~$n$,
assez grand pour que~$X$ et~$C$ s'étendent en des $\sO_S$\nobreakdash-schémas projectifs et lisses~$\sX$ et~$\sC$ à fibres géométriques irréductibles,
pour que~$f$ s'étende en un $\sO_S$\nobreakdash-morphisme plat $f:\sX \to \sC$ et pour que pour tout $v \notin S$, la fibre de~$f$ au-dessus du point générique de $\sC \otimes_{\sO_S} \Fv$ soit lisse
et géométriquement irréductible.

Soient~$r>0$ un entier et $v \notin S$.  Posons $m=n^r$, $\sX_{\sO_v} = \sX \otimes_{\sO_S} \sO_v$,
$\sC_{\sO_v}=\sC\otimes_{\sO_S}\sO_v$ et notons $\eta_v=\Spec(k_v(C))$ le point générique de $\sC_{\sO_v}$.
Si~$T$ est un schéma, notons $T^{(1)}$ l'ensemble des points de~$T$ de codimension~$1$.
Rappelons que pour tout schéma irréductible régulier~$T$ sur lequel~$m$ est inversible, la suite spectrale de Leray pour
l'inclusion $j:\etaT \to T$ du point générique~$\etaT$ de~$T$ fournit une suite exacte
$$
\xymatrix{
0 \ar[r] & H^1(T,\Z/m\Z) \ar[r] & H^1(\etaT, \Z/m\Z) \ar[r] & \displaystyle\bigoplus_{t \in T^{(1)}} H^0(t, \Z/m\Z(-1))\rlap{\text{,}}
}
$$
où $\Z/m\Z(-1)=\Hom(\mmu_N,\Z/m\Z)$ (cf.~\cite[Ch.~III, Theorem~1.18]{milneet} et \cite[Lemma~2.2]{brusseltengan}).
Considérant cette suite pour $T=\sC_{\sO_v}$ et pour $T=\sX_{\sO_v}$, on obtient un diagramme commutatif
\\\vbox to 23.5ex{
\begin{align}
\raisetag{-10.5ex}
\mylabel{diag:coholoc}
\begin{aligned}
\vbox to 3.4ex{
\demicrochet
\xymatrix@C=3ex@R=8ex{
H^1(\sX_{\sO_v},\Z/m\Z) \ar@{^{ (}->}[r] & H^1(X \times_C \eta_v,\Z/m\Z) \ar[r] \vphantom{\sX_{\sO_v}} &
\vphantom{\sX_{\sO_v}}\smash[b]{\displaystyle\bigoplus_{P \in \sC_{\sO_v}^{(1)}} \; {\bigoplus_{\substack{Q \in \sX_{\sO_v}^{(1)}\\f(Q)=P}} H^0(Q,\Z/m\Z(-1))}}\\
\ar[u] H^1(\sC_{\sO_v},\Z/m\Z) \ar@{^{ (}->}[r] & H^1(\eta_v,\Z/m\Z) \ar[u] \ar[r] & \vphantom{H^1(\sC_{\sO_v})}\smash{\displaystyle\bigoplus_{P \in \sC_{\sO_v}^{(1)}} H^0(P,\Z/m\Z(-1)) \ar[u]_(.46){\times e_{Q/P}}}
}
}
\end{aligned}
\end{align}
}
dont les lignes sont exactes, où~$e_{Q/P}$ désigne la multiplicité, dans la fibre $f^{-1}(P)$, de la composante irréductible contenant~$Q$.

La nullité de $A_0(X_\etabar\otimes K)$ implique
celle de $H^i(X_\eta,\sO_{X_\eta})$ pour $i>0$ (cf.~\cite[Corollaire~22.18]{voisinlivre}).
D'après~\cite[Proposition~7.3~(iii)]{ctvoisin},
il s'ensuit
que pour~$P$ fixé,
le pgcd des entiers $e_{Q/P}$ apparaissant ci-dessus est égal à~$1$.
Par conséquent, la flèche verticale de droite de~(\ref{diag:coholoc}) est injective.
Le morphisme $f_{\sO_v}:\sX_{\sO_v} \to \sC_{\sO_v}$ déduit de~$f$ par changement de base étant plat, surjectif et à fibres géométriquement connexes,
on a $(f_{\sO_v})_*\Z/m\Z=\Z/m\Z$; compte tenu de la suite spectrale de Leray pour~$f_{\sO_v}$, la flèche verticale de gauche de~(\ref{diag:coholoc}) est donc elle aussi injective.
D'autre part,
comme $H^1(X_\etabar,\Q/\Z)=0$, on a $H^1(X_\etabar,\Z/m\Z)=0$.
La suite spectrale de Leray entraîne maintenant que
la flèche verticale du milieu de~(\ref{diag:coholoc}) un isomorphisme.

Il résulte de tout cela que la flèche verticale de gauche de~(\ref{diag:coholoc}) est un isomorphisme.
En vertu du théorème de changement de base propre,
celle-ci s'identifie à l'application $f_{\Fv}^*:H^1(\sC \otimes_{\sO_v}\Fv,\Z/m\Z)\to H^1(\sX\otimes_{\sO_v}\Fv,\Z/m\Z)$.
Par dualité de Pontrjagin, il s'ensuit que la flèche naturelle
$(f_{\Fv})_*:\pi_1^\ab(\sX\otimes_{\sO_v}\Fv)/m \to \pi_1^\ab(\sC\otimes_{\sO_v}\Fv)/m$ est un isomorphisme.

Si~$\sV$ désigne un $\sO_v$-schéma projectif et lisse à fibres géométriquement irréductibles,
un théorème de Saito et Sato~\cite[Corollary~0.10]{saitosato}
assure que le morphisme de spécialisation
$\CH_0(\sV \otimes_{\sO_v} k_v)/m \to \CH_0(\sV \otimes_{\sO_v} \Fv)/m$ est un isomorphisme
(l'inversibilité de~$m$ dans~$\sO_v$ est ici cruciale).
De plus, d'après un théorème de Kato et Saito~\cite[Theorem~1]{katosaito},
il existe un isomorphisme canonique et fonctoriel $\CH_0(\sV \otimes_{\sO_v} \Fv)/m \isoto \pi_1^\ab(\sV \otimes_{\sO_v} \Fv)/m$.
Appliquant cela à $\sV=\sX_{\sO_v}$ et $\sV=\sC_{\sO_v}$, on conclut que l'application
\begin{align}
\mylabel{eq:fetoile}
f_*:\CH_0(X \otimes_k k_v)/m \to \CH_0(C \otimes_k k_v)/m
\end{align}
est elle aussi un isomorphisme.

Notons~$N$ et~$N'$ le noyau et le conoyau de
$f_*:\CH_0(X \otimes_k k_v) \to \CH_0(C \otimes_k k_v)$.
Comme~$m$ annule~$N'$, la surjectivité de~(\ref{eq:fetoile}) assure que $N'=0$.
L'injectivité de~(\ref{eq:fetoile}) et
la suite exacte
$$
\xymatrix{
0 \ar[r] & N \ar[r] & \CH_0(X \otimes_k k_v) \ar[r]^{f_*} & \CH_0(C \otimes_k k_v) \ar[r] & 0
}
$$
fournissent maintenant une surjection $\tors{m}\CH_0(C \otimes_k k_v)\twoheadrightarrow N/m=N$,
grâce au lemme du serpent.
Rappelons que $m=n^r$.  Comme~$C$ est une courbe, le groupe
$\tors{m}\CH_0(C \otimes_k k_v)$ est fini; par conséquent
l'application $\varprojlim_{r>0} \left(\tors{n^r}\CH_0(C\otimes_k k_v)\right) \rightarrow N$
obtenue par passage à la limite est encore surjective.
Or le sous-groupe de torsion de $\CH_0(C \otimes_k k_v)$ ne contient pas d'élément infiniment
divisible non nul puisque~$C$ est une courbe et~$k_v$ un corps $p$\nobreakdash-adique (cf.~rappel~\ref{rappel:chcourbe}).
Ainsi $\varprojlim_{r>0} \left(\tors{n^r}\CH_0(C\otimes_k k_v)\right)=0$, d'où $N=0$.
\end{proof}

\begin{rmqs}
\mylabel{rem:pgcdmult}
(i) Les hypothèses du théorème~\ref{th:isochow} devraient être satisfaites par des classes de fibrations dont la fibre générique n'est pas rationnellement connexe.
Par exemple, le théorème~\ref{th:isochow} devrait s'appliquer si~$X_\etabar$ est une surface de Barlow (cf.~\cite[Ch.~VII, \textsection10.7]{barthetc}), ou plus généralement
une surface de type général simplement connexe telle que $H^i(X_\etabar,\sO_{X_\etabar})=0$ pour $i \in \{1,2\}$.
La nullité de $A_0(X_\etabar \otimes K)$ résulterait dans ce cas d'une conjecture de Bloch~\cite[Lecture~1]{blochlectures}.

(ii) La preuve du théorème~\ref{th:isochow} ne donne aucun contrôle sur l'ensemble fini~$S$ puisque celui-ci dépend de l'entier~$n$ fourni par le lemme~\ref{lem:decompdiag}.
Cependant, 
un raffinement des techniques de déformation de courbes rationnelles employées dans~\cite{kollarszabo} permet d'établir,
au prix d'une démonstration plus délicate que celle du théorème~\ref{th:isochow},
que si la fibre générique de~$f$ est rationnellement connexe,
l'ensemble~$S$ peut être choisi égal à l'ensemble des places archimédiennes et des places finies en lesquelles~$f$ ne se réduit pas
en un morphisme de fibre générique géométrique séparablement rationnellement connexe
entre variétés propres et lisses.
\end{rmqs}

\section{Réduction des théorèmes~\ref{th:principalEun} et~\ref{th:principalE} à l'existence de zéro-cycles au-dessus d'une classe d'équivalence linéaire fixée}
\mylabel{sec:reductionP}

\medskip
Le but de ce paragraphe est de démontrer la

\begin{prop}
\mylabel{prop:reduc}
Soit~$X$ une variété irréductible, propre et lisse sur un corps de nombres~$k$,
munie d'un morphisme $f:X \to C$ de fibre générique géométriquement irréductible, où~$C$ est une courbe lisse et géométriquement irréductible.
Supposons que pour tout $c \in C$, le pgcd des multiplicités des composantes irréductibles de la fibre~$X_c$
soit égal à~$1$.
Considérons la propriété suivante, qui dépend d'un ensemble fini $S \subset \Omega$:
\begin{enumerate}
\item[$(\Psub{S})$]
Soient $y \in \CH_0(C)$ et $(z_v)_{v \in \Omega} \in \prod_{v \in \Omega} \CH_0(X \otimes_k k_v)$.  Si la famille $(z_v)_{v \in \Omega}$ est orthogonale à $\Br_\vert(X/C)$
pour l'accouplement de Brauer--Manin
et si $f_*z_v = y$ dans $\CH_0(C \otimes_k k_v)$ pour tout $v \in \Omega$,
alors pour tout entier $n>0$, il existe
$z \in \CH_0(X)$ tel que $f_*z = y$ dans $\CH_0(C)$ et tel que pour tout $v \in S$, on ait $z=z_v$ dans $\CH_0(X \otimes_k k_v)/n$ si~$v$ est finie
et $z=z_v+N_{\kbar_v/k_v}(u_v)$ dans $\CH_0(X \otimes_k k_v)$ pour un $u_v \in \CH_0(X \otimes_k \kbar_v)$ si~$v$ est réelle.
\end{enumerate}
On a alors:
\begin{enumerate}
\item[(i)] Si la propriété~$(\Psub{\emptyset})$ est satisfaite et si le groupe de Tate--Shafarevich de la jacobienne de~$C$ ne contient pas d'élément infiniment divisible non nul,
la variété~$X$ vérifie
l'énoncé de la conjecture~$(E_1)$.

\item[(ii)]
Notons~$K$ une clôture algébrique du corps des fonctions de la fibre générique géométrique~$X_\etabar$ de~$f$.
Si
$A_0(X_\etabar \otimes K)=0$ et $H^1(X_\etabar,\Q/\Z)=0$
et si la propriété~$(\Psub{S})$ est satisfaite pour tout ensemble fini $S \subset \Omega$,
le complexe~(\ref{diag:complexeErelatif}) est une suite exacte.

\item[(iii)] Si les hypothèses de~{\rm(ii)} sont satisfaites et si de plus
le groupe de Tate--Shafarevich de la jacobienne de~$C$ ne contient pas d'élément infiniment divisible non nul, la variété~$X$ vérifie l'énoncé de la conjecture~$(E)$.
\end{enumerate}
\end{prop}

Afin d'établir le théorème~\ref{th:principalEun}
(resp.~\ref{th:principalE} ou~\ref{th:principalErelatif}), il nous suffira donc d'établir~$(\Psub{\emptyset})$ (resp.~$(\Psub{S})$ pour tout~$S$) sous les
hypothèses dudit théorème.

\begin{proof}[Démonstration de la proposition~\ref{prop:reduc}]
D'après Saito et Colliot-Thélène, la seconde ligne
du diagramme commutatif
\begin{align}
\mylabel{diag:EXEC}
\begin{aligned}
\xymatrix{
\CHzchapeau(X) \ar[r] \ar[d]^{f_*} & \CHzAchapeau(X) \ar[r] \ar[d]^{f_*} & \Hom(\Br(X),\Q/\Z) \ar[d]^{\Hom(f^*,\Q/\Z)} \\
\CHzchapeau(C) \ar[r] & \CHzAchapeau(C) \ar[r] & \Hom(\Br(C),\Q/\Z)
}
\end{aligned}
\end{align}
est exacte si le groupe de Tate--Shafarevich de la jacobienne de~$C$ ne contient pas d'élément infiniment divisible non nul
(cf.~remarque~\ref{rqs:conje}~(iv)).

Afin de démontrer~(i), supposons la propriété $(\Psub{\emptyset})$ satisfaite et la seconde ligne de~(\ref{diag:EXEC}) exacte.  Fixons une famille $(z_v)_{v \in \Omega} \in \prod_{v \in \Omega} \CH_0(X\otimes_k k_v)$
orthogonale à $\Br(X)$, vérifiant $\deg(z_v)=1$ pour tout $v \in \Omega$. Notons~$z_\A$ son
image dans~$\CHzA(X)$.
D'après l'exactitude de la seconde ligne de~(\ref{diag:EXEC}), il existe $\ychapeau \in \CHzchapeau(C)$ tel que
$f_*z_\A=\ychapeau$ dans $\CHzAchapeau(C)$.
Le~conoyau de $f_*:\CHzA(X) \to \CHzA(C)$ est d'exposant fini (cf.~lemme~\ref{lem:conoyaumultisection});
notons~$N_1$ son exposant.

\gdef\citectsk{\cite[Lemma~3.1]{ctskrevisited}}%
\begin{lem}[{\rm{cf.}}~\citectsk]
\mylabel{lem:finitudebrvert}
Soit $f:X\to Y$ un morphisme
de fibre générique géométriquement irréductible entre variétés irréductibles et lisses sur un corps~$k$ de caractéristique~$0$.
Supposons que le pgcd des multiplicités des composantes irréductibles de la fibre de~$f$ au-dessus de chaque point de codimension~$1$ de~$Y$ soit égal à~$1$.
Alors il existe un sous-groupe fini $B \subset \Br_\vert(X/Y)$ tel que $\Br_\vert(X/Y)=B + f^*\Br(Y)$.
\end{lem}

\begin{proof}
Soit $Y^0 \subset Y$ un ouvert dense au-dessus duquel les fibres de~$f$ sont géométriquement irréductibles et lisses.
Notons~$M$ l'ensemble (fini) des points de codimension~$1$ de~$Y$ n'appartenant pas à~$Y^0$. Pour $m \in M$,
notons $(X_{m,i})_{i \in I_m}$ la famille des composantes irréductibles de~$f^{-1}(m)$ et
pour chaque $i \in I_m$, notons~$e_{m,i}$ la multiplicité de $X_{m,i}$ dans $f^{-1}(m)$ et
$\Res:H^1(k(m),\Q/\Z) \to H^1(k(X_{m,i}),\Q/\Z)$
l'application de restriction de~$k(m)$ au corps des fonctions $k(X_{m,i})$ de~$X_{m,i}$.
On dispose d'un diagramme commutatif
\begin{align}
\mylabel{diag:finitudebrvert}
\begin{aligned}
\xymatrix{
0 \ar[r] & \ar[d]^{f^*} \Br(Y) \ar[r] & \ar[d]^{f^*} \Br(Y^0) \ar[r]^(.4)\delta &
\smash[b]{\displaystyle\bigoplus_{m\in M} H^1(k(m),\Q/\Z)}\vphantom{\Br(Y)}
\ar[d]^{\bigoplus e_{m,i}\Res} \\
0 \ar[r] & \Br(X) \ar[r] & \Br(f^{-1}(Y^0)) \ar[r] & \displaystyle \bigoplus_{m\in M} \bigoplus_{\;i \in I_m} H^1(k(X_{m,i}),\Q/\Z)
}
\end{aligned}
\end{align}
dont les lignes sont exactes (cf.~\cite[Théorème~6.1, p.~134]{grbr3}
et~\cite[Proposition~1.1.1]{ctsd94}).
Notons~$K_{m,i}$ la fermeture algébrique de~$k(m)$ dans $k(X_{m,i})$ et fixons, pour chaque $m \in M$, une extension finie $K_m/k(m)$ dans laquelle se plongent les~$K_{m,i}$.
Le noyau de la flèche verticale de droite de~(\ref{diag:finitudebrvert}) s'identifie au noyau de
$$
\bigoplus_{m \in M, \mkern2mu i \in I_m} e_{m,i}\Res : \bigoplus_{m \in M} H^1(k(m),\Q/\Z) \longrightarrow \bigoplus_{m \in M}\bigoplus_{i \in I_m} H^1(K_{m,i},\Q/\Z) \rlap{\text{,}}
$$
qui est inclus dans celui de
$$
\bigoplus_{m \in M}\Res : \bigoplus_{m \in M} H^1(k(m),\Q/\Z) \longrightarrow \bigoplus_{m \in M} H^1(K_m,\Q/\Z)
$$
puisque les $e_{m,i}$ pour~$m$ fixé sont premiers entre eux.
L'extension finie $K_m/k(m)$ ne contenant qu'un nombre fini de sous-extensions cycliques, il s'ensuit que le noyau de la flèche verticale de droite de~(\ref{diag:finitudebrvert}) est fini.
Soit $\uplet{b_1}{b_n} \in \Br(Y^0)$ un système de représentants modulo $\Br(Y)$ de l'image réciproque de ce noyau par~$\delta$.
Comme le groupe $\Br(Y^0)$ est de torsion,
il résulte du diagramme~(\ref{diag:finitudebrvert}) que le sous-groupe $B \subset \Br_\vert(X/Y)$ engendré par les $f^*b_i$ vérifie les conditions requises,
compte tenu que $\Br_\vert(X/Y)=\Br(X) \cap f^*\mkern-2mu\Br(Y^0)$.
\end{proof}

Soit $B \subset \Br_\vert(X/C)$ un sous-groupe fini satisfaisant la conclusion du lemme~\ref{lem:finitudebrvert}.  Soit~$N_2$ son exposant.  Soit enfin~$N_3$ le degré d'un point fermé $P \in X$.
Choisissons un relèvement $y \in \CH_0(C)$ de l'image de $\ychapeau$ dans $\CH_0(C)/N_1N_2N_3$.
On a alors $f_*z_\A=y$ dans $\CHzA(C)/N_1N_2N_3$; d'où
\begin{align}
\mylabel{eq:defzzeroA}
f_*z_\A = y+N_2N_3f_*z_{0,\A}
\end{align}
dans $\CHzA(C)$ pour un $z_{0,\A} \in \CHzA(X)$.
Notons~$d$ le degré de la composante $v$\nobreakdash-adique de~$z_{0,\A}$ pour un $v \in \Omega_f$ (cet entier ne dépend pas de~$v$ d'après~(\ref{eq:defzzeroA}))
puis posons $z_{1,\A}=z_\A-N_2N_3z_{0,\A}+ dN_2P$
et $y_1 = y + dN_2f_*P$, de sorte que
\begin{align}
\mylabel{eq:flzy}
f_*z_{1,\A}=y_1
\end{align}
dans $\CHzA(C)$.

Comme $y_1 \in \CH_0(C)$, il résulte de~(\ref{eq:flzy}) et de la commutativité du diagramme~(\ref{diag:EXEC})
que la famille~$z_{1,\A}$ est orthogonale à $f^*\Br(C)$ pour l'accouplement de Brauer--Manin.
D'autre part elle est orthogonale à~$B$ puisque~$z_\A$ est elle-même orthogonale à~$B$, que~$N_2$ annule~$B$
et que l'image de~$P$ dans $\CHzA(X)$ provient de $\CH_0(X)$.
Par conséquent~$z_{1,\A}$ est orthogonale à $\Br_\vert(X/C)$.
Soit $(z_{1,v})_{v \in \Omega} \in \prod_{v \in \Omega} \CH_0(X \otimes_k k_v)$ un relèvement de $z_{1,\A} \in \CHzA(X)$.
Comme l'application $f_*:\CH_0(X \otimes_k \kbar_v) \to \CH_0(C \otimes_k \kbar_v)$ est surjective pour tout $v \in \Omega_\infty$,
on peut supposer que
$f_*z_{1,v}=y_1$ dans $\CH_0(C \otimes_k k_v)$ pour tout $v \in \Omega$
quitte à modifier le choix des relèvements~$z_{1,v}$ pour $v \in \Omega_\infty$.

La~propriété~($\Psub{\emptyset}$) entraîne maintenant l'existence de $z \in \CH_0(X)$ tel que $f_*z=y_1$,
ce qui conclut la démonstration de~$(E_1)$ puisque $\deg(z)=\deg(y_1)=1$.  L'assertion~(i) de la proposition est donc établie.

Avant de démontrer~(ii) et~(iii), prouvons le

\begin{lem}
\mylabel{lem:ameliorationps}
Supposons la propriété~$(\Psub{S})$ satisfaite pour tout ensemble fini $S \subset \Omega$. Supposons
que $A_0(X_\etabar \otimes K)=0$ et $H^1(X_\etabar,\Q/\Z)=0$.
Soit
$z_\A \in \CHzA(X)$
orthogonale à $\Br_\vert(X/C)$.
Soit $y \in \CH_0(C)$.
Si~$y$ et~$z_\A$ ont même image dans $\CHzA(C)$,
il existe $z \in \CH_0(X)$ ayant pour images~$y$ dans~$\CH_0(C)$ et~$z_\A$ dans $\CHzA(X)$.
\end{lem}

\begin{proof}
D'après le théorème~\ref{th:isochow},
il existe un ensemble fini $S \subset \Omega$ contenant~$\Omega_\infty$ tel que l'application
$f_*:\CH_0(X \otimes_k k_v) \to \CH_0(C \otimes_k k_v)$ soit un isomorphisme pour tout $v \notin S$.
Soit~$N_1$ un entier vérifiant la conclusion du lemme~\ref{lem:decompdiag} pour le morphisme $f:X \to C$.
Soit~$N_2$ un multiple commun des exposants des sous-groupes de torsion des groupes $\CH_0(C \otimes_k k_v)$ pour $v \in S \cap \Omega_f$
(cf.~rappel~\ref{rappel:chcourbe}).
Comme l'application $f_*:\CH_0(X \otimes_k \kbar_v) \to \CH_0(C \otimes_k \kbar_v)$ est surjective pour $v \in \Omega_\infty$
et comme $f_*z_\A=y$ dans $\CHzA(C)$,
il existe un relèvement $(z_v)_{v \in \Omega} \in \prod_{v \in \Omega} \CH_0(X\otimes_k k_v)$ de~$z_\A$
tel que $f_*z_v=y$ dans $\CH_0(C \otimes_k k_v)$ pour tout $v \in \Omega$.
La propriété~$(\Psub{S})$ fournit alors un $z \in \CH_0(X)$
et des $z_{2,v} \in \CH_0(X\otimes_k k_v)$ pour $v \in S \cap \Omega_f$
vérifiant d'une part $f_*z=y$ dans $\CH_0(C)$
et d'autre part
$z=z_v + N_1N_2 z_{2,v}$ dans $\CH_0(X \otimes_k k_v)$ pour tout $v \in S \cap \Omega_f$
et $z=z_v$ dans $\CH_0(X \otimes_k k_v)/N_{\kbar_v/k_v}(\CH_0(X \otimes_k \kbar_v))$ pour tout $v \in \Omega_\infty$.

Pour $v \notin S$, par définition de~$S$, l'égalité $f_*z=f_*z_v$ dans $\CH_0(C \otimes_k k_v)$ entraîne que $z=z_v$ dans $\CH_0(X \otimes_k k_v)$.
Pour $v \in S \cap \Omega_f$, l'égalité $f_*z=f_*z_v$ entraîne que $f_*z_{2,v}$ est de torsion dans $\CH_0(C \otimes_k k_v)$;
il en résulte que $N_2 f_*z_{2,v}=0$.
Ainsi $N_2 z_{2,v}$ appartient-il au noyau de~$f_*$.  D'où
$N_1N_2  z_{2,v}=0$ dans $\CH_0(X \otimes_k k_v)$.  Nous avons donc à nouveau $z=z_v$ dans $\CH_0(X \otimes_k k_v)$ pour tout $v \in S \cap \Omega_f$.
Par conséquent~$z$ et~$(z_v)_{v \in \Omega}$ ont même image dans $\CHzA(X)$.
\end{proof}

Le lemme~\ref{lem:ameliorationps} (avec $y=0$) entraîne immédiatement la validité de l'assertion~(ii) de la proposition~\ref{prop:reduc}.
Il nous reste seulement à établir~(iii); nous pouvons donc supposer que le groupe de Tate--Shafarevich de la jacobienne de~$C$
ne contient pas d'élément infiniment divisible non nul.
Soit $\zchapeau_\A \in \CHzAchapeau(X)$ un élément orthogonal à~$\Br(X)$.  Puisque la seconde ligne de~(\ref{diag:EXEC}) est exacte,
il existe $\ychapeau \in \CHzchapeau(C)$ tel que $f_*\zchapeau_\A=\ychapeau$ dans $\CHzAchapeau(C)$.

\begin{lem}
\mylabel{lem:chCchzXchzC}
L'application $\CH_0(C) \times \CHzchapeau(X) \to \CHzchapeau(C)$, $(y, \zchapeau) \mapsto y + f_*\zchapeau$ est surjective.
\end{lem}

\begin{proof}
Le conoyau de l'application $\CH_0(C) \to \CHzchapeau(C)$ est divisible d'après le lemme~\ref{lem:divisible}.
Celui de $f_*:\CHzchapeau(X) \to \CHzchapeau(C)$ est d'exposant fini; en effet
le lemme~\ref{lem:conoyaumultisection} et sa démonstration restent valables si l'on remplace
chaque occurrence de~$\CH_0$ par~$\CHzchapeau$.
Ainsi le conoyau de l'application apparaissant dans l'énoncé du lemme~\ref{lem:chCchzXchzC} est à la fois divisible et d'exposant fini. Il est donc nul.
\end{proof}

D'après le lemme~\ref{lem:chCchzXchzC}, il existe
 $y \in \CH_0(C)$ et $\zchapeau \in \CHzchapeau(X)$
tels que $\ychapeau=y+f_*\zchapeau$.
Quitte à remplacer $\zchapeau_\A$ par $\zchapeau_\A-\zchapeau$,
on peut supposer que $\zchapeau=0$, de sorte que l'image de~$y$ dans $\CHzAchapeau(C)$ est égale à $f_*\zchapeau_\A$.
Appliquant le lemme~\ref{lem:conoyauchapeau}
à l'homomorphisme $f_*:\CHzA(X) \to \CHzA(C)$,
dont le conoyau est d'exposant fini d'après le lemme~\ref{lem:conoyaumultisection},
on en déduit que l'image de~$y$ dans $\CHzA(C)$ s'écrit $f_*z_{0,\A}$ pour un $z_{0,\A} \in \CHzA(X)$.
Ainsi $f_*(\zchapeau_\A - z_{0,\A})=0$ dans $\CHzAchapeau(C)$.

Pour $v \in \Omega_f$, le noyau et le conoyau de $f_*:\CH_0(X \otimes_k k_v) \to \CH_0(C \otimes_k k_v)$
sont d'exposant fini grâce aux lemmes~\ref{lem:decompdiag} et~\ref{lem:conoyaumultisection}.
Comme~$C$ est une courbe, on peut appliquer le lemme~\ref{lem:noyauchapeau} à cet homomorphisme pour tout $v \in \Omega_f$.
Il en résulte que $\zchapeau_\A - z_{0,\A}$, et par conséquent~$\zchapeau_\A$, appartient à l'image de $\CHzA(X) \to \CHzAchapeau(X)$.
Soit $z_\A \in \CHzA(X)$ un antécédent de~$\zchapeau_\A$.
La flèche naturelle $\CHzA(C) \to \CHzAchapeau(C)$ étant injective (cf.~rappel~\ref{rappel:chcourbe}),
l'égalité $f_*\zchapeau_\A=\ychapeau$ entraîne $f_*z_\A=y$ dans $\CHzA(C)$.
D'après le lemme~\ref{lem:ameliorationps}, il s'ensuit que $z_\A$ provient de $\CH_0(X)$, ce qui conclut la démonstration de la proposition~\ref{prop:reduc}.
\end{proof}

\section{Démonstration des théorèmes~\ref{th:principalEun}, \ref{th:principalE} et~\ref{th:principalErelatif}}
\mylabel{sec:demPS}

\medskip
Nous démontrons dans ce paragraphe le

\begin{thm}
\mylabel{th:propPvraie}
Soit~$X$ une variété irréductible, propre et lisse sur un corps de nombres~$k$,
munie d'un morphisme $f:X \to C$
de fibre générique géométriquement irréductible, où~$C$ est une courbe lisse et géométriquement irréductible.
Si les hypothèses~(a) et~(b) du théorème~\ref{th:principalEun} sont satisfaites,
la propriété~$(\Psub{\emptyset})$ l'est aussi.
Si l'hypothèse~(a) du théorème~\ref{th:principalEun} et l'hypothèse~(b') du théorème~\ref{th:principalE} sont satisfaites,
la propriété~$(\Psub{S})$ l'est aussi pour tout ensemble fini $S \subset \Omega$.
\end{thm}

Grâce à la proposition~\ref{prop:reduc}, la validité des théorèmes~\ref{th:principalEun}, \ref{th:principalE} et~\ref{th:principalErelatif} s'ensuit.

\subsection{Lemmes d'effectivité}

On dira qu'un $0$\nobreakdash-cycle effectif sur~$X$ ou sur~$C$ est \emph{sans multiplicités} s'il est somme de points fermés deux à deux distincts.

\begin{lem}
\mylabel{lem:cyclerelevelocal}
Soit~$R$ un anneau local intègre hensélien, de corps des fractions~$K$ infini et de corps résiduel~$F$.
Soit~$\sC$ une courbe projective et lisse sur~$R$, de genre~$g$.
Posons $C = \sC \otimes_R K$ et fixons un fermé strict $M \subset C$.
Si $\Card(F)\geq \deg(M)$, alors tout diviseur~$D$ sur~$C$ de degré~$>2g$ est linéairement équivalent à un diviseur~$D'$ effectif, sans multiplicités et de support lisse,
tel que l'adhérence dans~$\sC$ du support de~$D'$ ne rencontre pas l'adhérence de~$M$.
\end{lem}

\begin{proof}
Soit~$D$ un diviseur sur~$C$ de degré $d>2g$.

Le produit symétrique relatif $d$\nobreakdash-ème $\Sym^d_{\sC/R}$ et la composante de degré~$d$ du foncteur de Picard relatif $\Pic^d_{\sC/R}$
sont des $R$\nobreakdash-schémas plats dont la formation est compatible aux changements de base (cf.~\cite[II.1]{iversen}, \cite[8.4/2 et~8.4/3]{blr}).
Comme $d>2g-2$, la restriction du morphisme canonique $\rho:\Sym^d_{\sC/R} \to \Pic^d_{\sC/R}$ au-dessus de chaque point de~$R$
est lisse (cf.~\cite[Remark~5.6(c)]{milnejac}).
Il s'ensuit que~$\rho$ est lui-même lisse
(cf.~\cite[11.3.11]{ega43}).
Notons $\sM\subset \sC$ l'adhérence de~$M$ et posons $\sC^0=\sC \setminus \sM$ et $C^0=C \setminus M$.
Notons enfin $U \subset \Sym^d_{C^0/K}$ l'ouvert paramétrant les diviseurs sans multiplicités et de support lisse
et~$\rho^0$ la composée de~$\rho$ et de l'immersion ouverte $\Sym^d_{\sC^0/R} \subset \Sym^d_{\sC/R}$.

Par propreté de $\Pic^d_{\sC/R}$ sur~$R$, le $K$\nobreakdash-point $x_K$ de $\Pic^d_{\sC/R}$ défini par~$D$ s'étend en un $R$\nobreakdash-point $x \in \Pic^d_{\sC/R}(R)$.
Le $R$\nobreakdash-schéma $\Sym^d_{\sC^0/R}$ paramètre les sous-schémas fermés de~$\sC^0$ finis et plats de degré~$d$ sur~$R$
(cf.~\cite[II.3 et~II.4]{iversen}).  Il suffit donc, pour conclure, de vérifier que~$x$ est l'image par~$\rho^0$ d'un $R$\nobreakdash-point de $\Sym^d_{\sC^0/R}$ rencontrant~$U$.

La restriction~$S$ de~$\rho^0$ au-dessus de~$x$ est un
$R$\nobreakdash-schéma lisse. Comme $d>2g$, le diviseur~$D$ est très
ample. La fibre générique de~$S$ rencontre donc~$U$, par le théorème de
Bertini (cf.~\cite[Theorem~8.18]{hartshorne}).
D'après le théorème de Riemann--Roch, il résulte de l'inégalité $d\geq 2g$ que la fibre spéciale de~$S$
est le complémentaire, dans un espace projectif non vide,
de la réunion d'au plus~$\deg(M)$ hyperplans; l'hypothèse $\Card(F) \geq \deg(M)$ entraîne donc que $S(F)\neq \emptyset$.
Or, si~$R\neq K$, tout $F$\nobreakdash-point de~$S$ se relève en un $R$\nobreakdash-point rencontrant~$U$
puisque~$R$ est hensélien et~$S$ est lisse sur~$R$.
\end{proof}

\begin{lem}
\mylabel{lem:effectivitecourbecdn}
Soit~$X$ une variété sur un corps de nombres~$k$,
munie d'un morphisme $f:X \to C$
de fibre générique lisse et géométriquement irréductible, où~$C$ est une courbe propre, lisse et géométriquement irréductible sur~$k$, de genre~$g$.
Il~existe un ensemble fini $S \subset \Omega$ tel que pour toute place $v \notin S$, tout diviseur sur $C \otimes_k k_v$ de degré~$>2g$
soit linéairement équivalent à un diviseur effectif sans multiplicités de la forme $f_*z$ où~$z$ est un $0$\nobreakdash-cycle effectif sur $X\otimes_k k_v$ supporté par des fibres
lisses de~$f$.
\end{lem}

\begin{proof}
Soit $M \subset C$ un ensemble fini assez grand pour que~$f$ soit lisse et à fibres géométriquement irréductibles au-dessus de $C \setminus M$.
Soit $S \subset \Omega$ un ensemble fini assez grand pour que~$C$ s'étende en une courbe projective et lisse~$\sC$ sur~$\sO_S$,
pour que~$X$ s'étende en un $\sO_S$\nobreakdash-schéma plat~$\sX$ et pour que~$f$ s'étende en un $\sO_S$\nobreakdash-morphisme $f:\sX \to \sC$
lisse au-dessus de $\sC \setminus \sM$, où $\sM \subset \sC$ désigne l'adhérence de~$M$.
Quitte à agrandir~$S$, on peut supposer, grâce au théorème de Lang--Weil--Nisnevi\v{c}~\cite{langweil}, que les fibres de~$f$ au-dessus des points fermés
de $\sC\setminus \sM$ possèdent toutes un point rationnel.  On peut supposer de plus que~$S$ contient les places archimédiennes ainsi que les places finies dont le corps résiduel
est de cardinal $<\deg(M)$.

Il résulte du lemme~\ref{lem:cyclerelevelocal}
que pour tout $v\notin S$, tout diviseur sur $C\otimes_kk_v$ de degré~$>2g$
est linéairement équivalent à un diviseur effectif sans multiplicités~$D'$ tel que l'adhérence~$\sD'$ du support de~$D'$ dans $\sC \otimes_{\sO_S} \sO_v$ ne rencontre pas $\sM \otimes_{\sO_S} \sO_v$.
La restriction du morphisme~$f$ au-dessus de $\sD'$ est donc lisse et ses fibres fermées contiennent toutes un point rationnel.
Il s'ensuit, grâce au lemme de Hensel, que~$D'$ s'écrit $f_*z$ pour un $0$\nobreakdash-cycle effectif~$z$ sur $X \otimes_k k_v$.
\end{proof}

Rappelons enfin le lemme suivant, démontré dans \cite[Lemmes~3.1 et~3.2]{ctreglees}.

\begin{lem}
\mylabel{lem:effectiflocal}
Soit~$C$ une courbe propre, lisse et géométriquement irréductible sur un corps parfait infini~$k$.
Soit~$X$ une variété projective, lisse et géométriquement irréductible sur~$k$, munie d'un morphisme surjectif $f:X\to C$.
Soit enfin $M \subset C$ un fermé strict.  Pour tout sous-schéma fermé $F \subset X$ étale sur~$k$,
il existe un entier~$N_0$ et une
courbe propre, lisse et géométriquement irréductible $Z \subset X$ dominant~$C$ et contenant~$F$
tels que pour tout $z \in Z_0(X)$, si~$z$ est supporté par~$Z$ et si $\deg(z)\geq N_0$
alors~$z$ est rationnellement équivalent à un $0$\nobreakdash-cycle effectif~$z'$
supporté par~$Z$ tel que $f_*z'$ soit sans multiplicités et de support disjoint de~$M$.
\end{lem}

\subsection{Restriction des scalaires à la Weil}
\mylabel{sec:restrictionWeil}

Divers résultats généraux concernant la restriction des scalaires à la Weil
le long d'un revêtement de courbes serviront dans la preuve du
théorème~\ref{th:propPvraie}. Par souci de clarté, nous les avons regroupés ci-dessous.

Étant donnés des morphismes de schémas $X' \to S' \to S$, nous noterons $\sR_{S'/S}X'$ la restriction des scalaires à la Weil
du $S'$\nobreakdash-schéma~$X'$ le long de $S' \to S$.  Rappelons que si $S' \to S$ est fini localement libre
et si $X' \to S'$ est quasi-projectif,
alors
 $\sR_{S'/S}X'$ est
un $S$\nobreakdash-schéma (cf.~\cite[7.6/4]{blr}); c'est un $S$\nobreakdash-schéma quasi-projectif si de plus~$S$ est localement noethérien
(cf.~\cite[Proposition~A.5.8]{cgp}; noter que la preuve de \emph{loc.\ cit.} est encore valable lorsque la base n'est pas affine).

Dans tout le~\textsection\ref{sec:restrictionWeil}, nous fixons
une variété~$X$ lisse sur un corps~$k$,
un morphisme fini et plat $\pi:C \to B$ entre courbes lisses sur~$k$
et un morphisme $f:X\to C$ projectif et plat.
Nous supposons que~$B$ est irréductible
et qu'il existe un ensemble fini $M \subset C$ (considéré par la suite comme sous-schéma fermé réduit de~$C$)
tel que les fibres de~$f$ au-dessus de $C \setminus M$ soient lisses et géométriquement irréductibles.
Notons $R \subset C$ le lieu de ramification de~$\pi$.
Posons enfin $W = \sR_{C/B}X$ et notons $p:W \to B$ le morphisme structural.
Comme~$p$ est quasi-projectif, il existe une variété~$W'$ sur~$k$ contenant~$W$ comme ouvert dense, telle que~$p$ se prolonge en un morphisme projectif $p':W'\to B$.

\begin{prop}
\mylabel{prop:rwproprietes}
Supposons que~$\pi$ induise un isomorphisme $M \isoto \pi(M)$ et supposons que $\pi(M) \cap \pi(R)=\emptyset$.
Alors:
\begin{enumerate}
\item[(i)] La variété~$W$ est irréductible et lisse.
\item[(ii)]
Au-dessus de $B \setminus \left(\pi(M) \cup \pi(R)\right)$,
le morphisme~$p'$ est lisse et ses fibres sont géométriquement irréductibles.
\item[(iii)]
La fibre de~$p'$ au-dessus de chaque point de~$\pi(R)$ contient une composante irréductible de multiplicité~$1$ qui est géométriquement irréductible.
\item[(iv)]
Pour tout $c \in M$, il existe
une variété~$V$ lisse et géométriquement irréductible sur~$k(c)$
telle que la fibre de~$p'$ au-dessus de~$\pi(c)$ s'écrive $X_c \times_{k(c)} V$, où~$X_c$ désigne la fibre de~$f$.
\end{enumerate}
\end{prop}

\begin{proof}
Comme~$f$ est lisse et à fibres géométriquement irréductibles
au-dessus de $C \setminus \pi^{-1}(\pi(M))$ et est propre et plat au-dessus de $C \setminus \pi^{-1}(\pi(R))$,
le morphisme~$p$ est lisse et à fibres géométriquement irréductibles
au-dessus de $B \setminus \pi(M)$ et est propre et plat au-dessus de $B \setminus \pi(R)$ (cf.~\cite[Proposition~A.5.11]{cgp}
et~\cite[7.6/5]{blr}).
Par conséquent, la variété $p^{-1}(B\setminus\pi(M))$ est lisse,
le morphisme~$p$ est plat (ce qui entraîne que toute composante irréductible de~$W$ domine~$B$)
et les morphismes~$p$ et~$p'$ s'identifient canoniquement au-dessus de $B \setminus \pi(R)$.
De ces remarques découlent l'irréductibilité de~$W$ et les assertions~(ii) et~(iii).

Pour $c \in M$, la fibre $\pi^{-1}(\pi(c))$ est la réunion
disjointe de~$c$ et d'un sous-schéma fermé~$F$ de $C \setminus M$ étale sur $k(c)=k(\pi(c))$.
Comme la restriction des scalaires à la Weil est compatible au changement de base et transforme union disjointe en produit
(cf.~\cite[(4.2.3) et~(4.2.6)]{scheiderer}), la fibre de~$p$ (et donc de~$p'$) au-dessus de~$\pi(c)$ s'écrit
$X_c \times_{k(c)} V$, où $V=\sR_{F/k(c)} p^{-1}(F)$.
D'après~\cite[Proposition~A.5.11]{cgp},
la lissité et l'irréductibilité géométrique des fibres de $p^{-1}(F) \to F$ entraînent ensemble les mêmes propriétés pour~$V$ sur~$k(c)$, d'où l'assertion~(iv).

Il nous reste seulement à établir la lissité de~$W$ au voisinage de $p^{-1}(\pi(c))$ pour chaque $c \in M$.
Pour cela, il est loisible de remplacer~$B$ par un voisinage étale de~$\pi(c)$.  Comme $\pi(M) \cap \pi(R)=\emptyset$ et $M \isoto \pi(M)$, on peut supposer de cette façon
que $M=\{c\}$ et que $C = C_1 \amalg \dots \amalg C_d$, où $c \in C_1$ et où la restriction de~$\pi$ à chaque~$C_i$ est un isomorphisme.
Le $B$\nobreakdash-schéma $X_i = X \times_C C_i$ est alors lisse pour $i>1$
et la décomposition $X = X_1 \amalg \dots \amalg X_d$ se traduit par $W = X_1 \times_B \dots \times_B X_d$
(cf.~\cite[(4.2.6)]{scheiderer}).
Comme~$X_i$ est lisse sur~$B$ pour $i>1$, la projection $W \to X_1$ est lisse.
D'autre part, la variété~$X_1$ est lisse sur~$k$, étant un ouvert de~$X$;
par conséquent~$W$ est bien lisse sur~$k$.
\end{proof}

Il résulte de la proposition~\ref{prop:rwproprietes} que si les fibres de~$f$ satisfont l'hypothèse~(a) du théorème~\ref{th:principalEun},
les fibres de~$p'$ la satisfont aussi.
Supposons maintenant que~$k$ soit un corps de nombres et vérifions que les hypothèses~(b) et~(b') des théorèmes~\ref{th:principalEun} et~\ref{th:principalE}
se transmettent également de~$f$ à~$p'$.

\begin{prop}
\mylabel{prop:rwtransmissionhyp}
Soit $t \in B \setminus (\pi(M)\cup \pi(R))$ un point fermé.  Notons $\Wsub{t}=p'^{-1}(t)$.
\begin{enumerate}
\item[(i)] Supposons que~$X_c$ vérifie la condition~(b) du théorème~\ref{th:principalEun} pour tout $c \in \pi^{-1}(t)$.
Si $\Wsub{t}(\A_{k(t)})\neq\emptyset$, alors~$\Wsub{t}$ admet un $0$\nobreakdash-cycle de degré~$1$ sur~$k(t)$.
\item[(ii)]
Supposons que~$X_c$ vérifie la condition~(b') du théorème~\ref{th:principalE} pour tout $c \in \pi^{-1}(t)$.
Pour tout $n>0$ et tout ensemble fini~$T$ de places de~$k(t)$,
l'image de $\CH_0(\Wsub{t})$ dans $$\prod_{w \in T} \CH_0(\Wsub{t} \otimes_{k(t)} k(t)_w)/n$$
contient l'image de $\Wsub{t}(\A_{k(t)})$ dans ce même groupe.
\end{enumerate}
\end{prop}

\begin{proof}
Remarquons que si $L/K$ est une extension finie de corps de caractéristique nulle, si~$V$ est une variété projective sur~$L$ et si
$P \in V$ est un point fermé de degré~$d$ sur~$L$, alors $\sR_{L/K}P$ est un $0$\nobreakdash-cycle sur $\sR_{L/K}V$ de degré $d^{[L:K]}$ sur~$K$.
Ainsi, l'existence d'un $0$\nobreakdash-cycle sur~$V$ de degré~$1$ sur~$L$ entraîne-t-elle celle d'un $0$\nobreakdash-cycle sur $\sR_{L/K}V$ de degré~$1$ sur~$K$.

De l'égalité
\begin{align}
\mylabel{eq:wtxc}
\Wsub{t} = \prod_{c \in \pi^{-1}(t)}\!\sR_{k(c)/k(t)}X_c
\end{align}
on tire que
\begin{align}
\mylabel{eq:wtxca}
\Wsub{t}(\A_{k(t)})=\prod_{c \in \pi^{-1}(t)}\!X_c(\A_{k(c)})\rlap{\text{.}}
\end{align}
Si $\Wsub{t}(\A_{k(t)})\neq\emptyset$
et si~$X_c$ satisfait, 
pour chaque $c \in \pi^{-1}(t)$,
la condition~(b) du théorème~\ref{th:principalEun},
il suit de~(\ref{eq:wtxca}) que~$X_c$ possède un $0$\nobreakdash-cycle de degré~$1$ sur~$k(c)$
pour chaque $c \in \pi^{-1}(t)$.
D'après la remarque précédente
et l'égalité~(\ref{eq:wtxc}),
cela implique que~$\Wsub{t}$ admet un $0$\nobreakdash-cycle de degré~$1$ sur~$k(t)$.
L'assertion~(i) du lemme est donc établie.

Soient~$n>0$ un entier et~$T$ un ensemble fini de places de~$k(t)$.
D'après Karpenko \cite[\textsection4 et Proposition~4.4]{karpenkoweil}, il existe des flèches verticales (non additives) rendant commutatif le diagramme suivant:
$$
\xymatrix@C=1.924em{
\vphantom{C_)}\CH_0(\Wsub{t}) \ar[r]^(.32)\alpha &
\vphantom{C_)}\smash[b]{\displaystyle\prod_{w \in T} \CH_0(\Wsub{t} \otimes_{k(t)} k(t)_w)/n}
& \Wsub{t}(\A_{k(t)}) \ar[l]_(.32)\beta \\
\displaystyle\prod_{c \in \pi^{-1}(t)}\!\CH_0(X_c) \ar[u] \ar[r]^(.32)\gamma & \ar[u]
\vphantom{\displaystyle\prod}\smash[t]{\displaystyle\prod_{w\in T}\left(\prod_{c \in \pi^{-1}(t)}\!\CH_0(X_c\otimes_{k(t)}k(t)_w)/n\right)} &
\displaystyle\prod_{c \in \pi^{-1}(t)}\!X_c(\A_{k(c)})\rlap{\text{.}} \ar[l]_(.32)\delta \ar[u]_\wr
}
$$
Si~$X_c$ satisfait la condition~(b') du théorème~\ref{th:principalE}
pour tout $c \in \pi^{-1}(t)$, l'image de~$\gamma$ contient celle de~$\delta$;
comme la flèche verticale de droite est bijective, il s'ensuit que l'image de~$\alpha$ contient celle de~$\beta$.
\end{proof}

\subsection{Démonstration du théorème~\texorpdfstring{\ref{th:propPvraie}}{4.1}}
\mylabel{sec:preuveP}

Soit $f:X \to C$ un morphisme vérifiant les hypothèses du théorème~\ref{th:propPvraie} et les hypothèses~(a) et~(b) du théorème~\ref{th:principalEun}.
Nous devons établir
la propriété~$(\Psub{\emptyset})$, ainsi que~$(\Psub{S})$ pour tout $S \subset \Omega$ fini si l'hypothèse~(b') du théorème~\ref{th:principalE} est satisfaite.

D'après le lemme de Chow et le théorème de Hironaka, il existe
une variété~$X'$ irréductible, projective et lisse sur~$k$, munie d'un morphisme birationnel $X' \to X$.
Notons $f':X' \to C$ la composée de ce morphisme avec~$f$.
Quitte à remplacer~$X'$ par la variété obtenue en faisant éclater dans~$X'$ un nombre fini de points fermés contenus dans des fibres lisses de~$f'$,
on peut supposer que pour tout $c \in C$, si~$X'_c$ est lisse alors~$X_c$ est lisse et est birationnellement équivalente à~$X'_c$.
La remarque~\ref{rq:hypothesesth}~(v) entraîne alors
que les fibres lisses de~$f'$ vérifient l'hypothèse~(b) du théorème~\ref{th:principalEun},
ainsi que l'hypothèse~(b') du théorème~\ref{th:principalE} si les fibres lisses de~$f$ la satisfont.
D'autre part, comme les fibres de~$f$ vérifient l'hypothèse~(a) du théorème~\ref{th:principalEun}, il résulte de \cite[Lemme~3.8]{owlnm} que les fibres de~$f'$ la vérifient aussi.
La propriété~$(\Psub{S})$ étant un invariant birationnel (cf.~remarque~\ref{rqs:conje}~(vi)),
il suffit donc,
pour démontrer le théorème~\ref{th:propPvraie} pour~$X$, de le démontrer pour~$X'$.
Quitte à remplacer~$X$ et~$f$ par~$X'$ et~$f'$, nous supposerons donc
désormais la variété~$X$ projective (ce qui nous permettra de considérer les produits symétriques de~$X$ et les restrictions des scalaires à la Weil de~$f$
sans sortir de la catégorie des schémas).

Fixons un ensemble fini $S \subset \Omega$ et des $0$\nobreakdash-cycles $y \in Z_0(C)$ et $(z_v)_{v \in \Omega} \in \prod_{v \in \Omega} Z_0(X \otimes_k k_v)$
tels que $f_*z_v=y$ dans $\CH_0(C\otimes_k k_v)$ pour tout $v \in \Omega$.
Supposons la famille $(z_v)_{v \in \Omega}$ orthogonale à $\Br_\vert(X/C)$ pour l'accouplement de Brauer--Manin.
Nous devons montrer qu'il existe un $0$\nobreakdash-cycle $z \in Z_0(X)$
tel que $f_*z=y$ dans $\CH_0(C)$;
si de plus
l'hypothèse~(b') du théorème~\ref{th:principalE} est satisfaite, nous devons montrer que pour tout entier $n>0$, le cycle~$z$ peut être choisi de telle façon que
 $z=z_v$ dans
$\CH_0(X \otimes_k k_v)/n$ pour $v \in S \cap \Omega_f$ et $z=z_v$ dans $\CH_0(X \otimes_k k_v)/N_{\kbar_v/k_v}(\CH_0(X \otimes_k \kbar_v))$ pour $v \in S \cap \Omega_\infty$.

Soit $B \subset \Br_\vert(X/C)$ un sous-groupe fini tel que $\Br_\vert(X/C)=B + f^*\Br(C)$ (cf.~lemme~\ref{lem:finitudebrvert}).
Quitte à agrandir~$S$, on peut supposer que $\Omega_\infty \subset S$ et que~$X$ et les éléments de~$B$ ont bonne réduction hors de~$S$ (cf.~\cite[p.~69]{ctsd94}), de sorte que

\begin{enumerate}
\item[(i)] $\mylangle b, u_v \myrangle=0$ pour tout $b \in B$, tout $v \in \Omega\setminus S$ et tout $u_v \in Z_0(X \otimes_k k_v)$.
\end{enumerate}

Notons~$g$ le genre de~$C$ et $M \subset C$ l'ensemble des points au-dessus desquels la fibre de~$f$ n'est pas lisse.
D'après le lemme~\ref{lem:effectivitecourbecdn}, quitte à agrandir encore~$S$, on peut supposer que

\begin{enumerate}
\item[(ii)] pour tout $v \in \Omega\setminus S$, tout diviseur sur $C \otimes_k k_v$ de degré $>2g$ est linéairement équivalent à un diviseur effectif
sans multiplicités,
de la forme $f_*z$ pour un $z \in Z_0^\eff(X \otimes_k k_v)$
et de support disjoint de $M \otimes_k k_v$.
\end{enumerate}

Soit $P \in X$ un point fermé.  Pour chaque $v \in S$, le
lemme~\ref{lem:effectiflocal} appliqué au morphisme $f \otimes_k k_v : X\otimes_k
k_v \to C \otimes_k k_v$ et au fermé constitué de la réunion des supports des $0$\nobreakdash-cycles $P \otimes_k k_v$
et~$z_v$ fournit une courbe $Z_v \subset X \otimes_k k_v$ et un entier~$N_{0,v}$.  Soit~$N$ un entier
assez grand pour que $\deg(z_v+NP)\geq N_{0,v}$ pour tout $v \in S$
et pour que $\deg(y)+N\deg(P)>2g$.
Par définition de~$N_{0,v}$, pour $v \in S$, le $0$\nobreakdash-cycle $z_v + NP$
est rationnellement équivalent, sur $X \otimes_k k_v$, à un $0$\nobreakdash-cycle effectif~$z_v'$ supporté par~$Z_v$ tel que $f_*z_v'$ soit sans multiplicités et de support
disjoint de $M \otimes_k k_v$.  Posons $z_v^1=z_v'$ pour $v \in S$ et $z_v^1=z_v+NP$ pour $v \in \Omega\setminus S$.
Comme les classes dans $\prod_{v \in \Omega} \CH_0(X \otimes_k k_v)$ des familles $(z_v)_{v \in \Omega}$ et $(z_v^1)_{v \in \Omega}$ diffèrent
par l'image d'un élément de $\CH_0(X)$, la famille $(z_v^1)_{v \in \Omega}$ est encore orthogonale à $\Br_\vert(X/C)$.
Quitte à remplacer~$z_v$ par~$z_v^1$ pour tout $v \in \Omega$ et~$y$ par $y+Nf_*P$, on peut donc supposer que

\begin{enumerate}
\item[(iii)] pour tout $v \in S$, le $0$\nobreakdash-cycle~$z_v$ est effectif et~$f_*z_v$ est sans multiplicités et de support disjoint de $M \otimes_k k_v$;
par ailleurs $\deg(y)>2g$ et pour tout $v \in \Omega_\infty$, le $0$\nobreakdash-cycle~$z_v$ est supporté par~$Z_v$.
\end{enumerate}

Le diviseur~$y$ sur~$C$ est alors très ample.  Il est donc linéairement équivalent à un diviseur effectif $y_0 \in Z_0(C)$ de support disjoint de~$M$
et sans multiplicités.

Comme les fibres de~$f$ au-dessus de $C \setminus M$ sont géométriquement irréductibles,
le théorème de Lang--Weil--Nisnevi\v{c}~\cite{langweil} entraîne l'existence d'un ensemble fini $S' \subset \Omega$ contenant~$S$ et, pour chaque place $v \in \Omega\setminus S'$,
d'un $0$\nobreakdash-cycle $z_v^2 \in Z_0(X\otimes_k k_v)$ effectif et sans multiplicités tel que $f_*z_v^2=y_0$ dans $Z_0(C \otimes_kk_v)$.
Pour $v \in S' \setminus S$, fixons, à l'aide de la propriété~(ii), un $0$\nobreakdash-cycle effectif~$z_v^2$ sur~$X \otimes_k k_v$ tel que $f_*z_v^2$
soit sans multiplicités, de support disjoint de $M \otimes_k k_v$ et linéairement équivalent à~$y$.
Enfin, pour $v \in S$, posons $z_v^2=z_v$.
La famille $(z_v^2)_{v \in \Omega}$ ainsi définie est encore orthogonale à~$B$ pour l'accouplement de Brauer--Manin,
grâce à la propriété~(i).  D'autre part, elle est orthogonale à $f^*\Br(C)$
en vertu de la formule de projection,
puisque $f_*z_v^2=y$ dans $\CH_0(C \otimes_k k_v)$ pour toute place $v \in \Omega$.
Par conséquent $(z_v^2)_{v \in \Omega}$ est orthogonale à $\Br_\vert(X/C)$.  Ainsi, compte tenu de~(iii),
on peut supposer, quitte à remplacer~$S$ par~$S'$ et~$z_v$ par~$z_v^2$ pour tout $v \in \Omega$,
que

\begin{enumerate}
\item[(iv)] pour tout $v \in \Omega$, le $0$\nobreakdash-cycle~$z_v$ est effectif et~$f_*z_v$ est sans multiplicités et de support disjoint de $M \otimes_k k_v$;
de plus, pour tout $v \in \Omega \setminus S$, l'égalité
$f_*z_v=y_0$ vaut dans $Z_0(C\otimes_k k_v)$.
\end{enumerate}

Les propriétés~(i) et~(ii) ne nous serviront plus.

La prochaine réduction concerne uniquement les places réelles;
le lecteur prêt à supposer~$k$ totalement imaginaire
peut l'ignorer.
Employons les notations du paragraphe~\ref{sec:vanhamel}.
D'après le corollaire~\ref{cor:existeubis}, il existe
des $0$\nobreakdash-cycles
$d_v \in Z_0(Z_v\otimes_{k_v} \kbar_v)$  pour $v \in \Omega_\infty$,
tous de degré~$0$,
tels que
si l'on pose
$z''_v=z_v$ pour $v \in \Omega_f$ et $z''_v=z_v+N_{\kbar_v/k_v}(d_v)$ pour $v \in \Omega_\infty$,
la famille $(f_{*+}z''_v)_{v \in \Omega}$
appartienne à l'image de la flèche diagonale
$\Picplus(C) \to \prod_{v \in \Omega} \Picplus(C\otimes_k k_v)$.
Comme $\deg(z_v)=\deg(z''_v)$ et comme, d'après~(iii), le cycle~$z_v$ est supporté par~$Z_v$ pour $v \in \Omega_\infty$,
il résulte de la définition de~$Z_v$ que~$z''_v$, pour $v \in \Omega_\infty$, est rationnellement équivalent à un $0$\nobreakdash-cycle effectif~$z'''_v$ tel que $f_*z'''_v$ soit sans multiplicités et de support
disjoint de $M \otimes_k k_v$.
Quitte à remplacer~$z_v$
 par~$z_v'''$
pour chaque $v \in \Omega_\infty$,
on peut maintenant supposer que

\begin{enumerate}
\item[(v)]
la famille $(f_{*+}z_v)_{v \in \Omega}$ appartient à l'image de l'application
diagonale
$$\Picplus(C) \to
\prod_{v \in \Omega} \Picplus(C\otimes_k k_v)\rlap{\text{.}}$$
\end{enumerate}

Notons $\abs{y} \subset \Sym_{C/k}=\displaystyle\coprod \Sym^d_{C/k}$ le système linéaire complet associé à~$y$. C'est un espace projectif sur~$k$.
Soit $F_M \subset \abs{y}$ le lieu des diviseurs dont le support rencontre~$M$
et $F_R \subset \abs{y}$ le complémentaire du lieu des diviseurs sans multiplicités.
Notons enfin $F \subset F_M \cup F_R$ l'ensemble des points situés sur au moins deux composantes irréductibles distinctes de $(F_M \cup F_R) \otimes_k \kbar$.
Les fermés~$F_M$ et~$F_R$ ne contiennent pas le point~$y_0$.
Il existe donc un ouvert dense $\abs{y}^0 \subset \abs{y}$ tel que pour tout $y_\infty \in \abs{y}^0$,
la droite de~$\abs{y}$ passant par~$y_0$ et par~$y_\infty$ ne rencontre pas~$F$.
Quitte à rétrécir~$\abs{y}^0$, on peut supposer~$\abs{y}^0$ disjoint de $F_M \cup F_R$ et du lieu des diviseurs dont le support rencontre le support de~$y_0$.

Par approximation faible dans l'espace projectif~$\abs{y}$, il existe $y_\infty \in \abs{y}^0(k)$ arbitrairement proche de $f_*z_v \in \abs{y}(k_v)$ pour $v \in S$.

Comme~$f_*z_v$ est sans multiplicités et de support disjoint de $M \otimes_k k_v$,
le morphisme $f_*:\Sym_{X/k} \to \Sym_{C/k}$ est lisse en~$z_v$.
Par le théorème des fonctions implicites, il s'ensuit
que pour chaque $v \in S$,
le point $y_\infty \in \Sym_{C/k}(k_v)$, qui est par hypothèse arbitrairement proche de~$f_*z_v$, se relève
en un $z_v^3 \in \Sym_{X/k}(k_v)$ arbitrairement proche de~$z_v$.
D'après le lemme~\ref{lem:continuite}, pour $v \in S \cap \Omega_f$ (resp.~$v \in S \cap \Omega_\infty$),
si~$z_v^3$ est suffisamment proche de~$z_v$, les classes de~$z_v^3$ et de~$z_v$ dans $\CH_0(X\otimes_k k_v)/n$ (resp.~dans $\CH_0(X\otimes_k k_v)/2$
et donc dans $\CH_0(X\otimes_k k_v)/N_{\kbar_v/k_v}(\CH_0(X\otimes_k\kbar_v))$)
seront égales.
En conclusion, quitte à choisir~$y_\infty$ assez proche de~$f_*z_v$ pour $v \in S$ puis à remplacer~$z_v$ par~$z_v^3$ pour $v \in S$, on peut supposer que

\begin{enumerate}
\item[(vi)] pour tout $v \in S$, l'égalité $f_*z_v=y_\infty$ vaut dans $Z_0(C \otimes_k k_v)$.
\end{enumerate}

Les diviseurs~$y_0$ et~$y_\infty$ sur~$C$ étant de supports disjoints,
la droite~$D$ de~$\abs{y}$ passant par~$y_0$ et par~$y_\infty$ détermine un système linéaire sans point base.
Il existe donc un morphisme $\pi : C \to \P^1_k$ tel que $\pi^{-1}(0)=y_0$ et $\pi^{-1}(\infty)=y_\infty$.
Notons $R \subset C$ le lieu de ramification de~$\pi$
et $p:W \to \P^1_k$ la restriction des scalaires à la Weil du morphisme $f:X \to C$ le long de~$\pi$ (cf.~\textsection\ref{sec:restrictionWeil}).
Comme $D \cap F = \emptyset$, toute fibre géométrique de~$\pi$ rencontrant~$M$ est étale et ne rencontre~$M$ qu'en un point.
Il~s'ensuit que $\pi(M) \cap \pi(R)=\emptyset$ et que~$\pi$ induit un isomorphisme $M \isoto \pi(M)$.
Comme au~\textsection\ref{sec:restrictionWeil}, notons~$W'$ une compactification projective de~$W$ telle que~$p$ s'étende
en un morphisme $p':W' \to \P^1_k$.  La variété~$W$ étant lisse par la proposition~\ref{prop:rwproprietes}~(i), on peut choisir~$W'$ lisse,
grâce à Hironaka.

D'après les propriétés~(iv) et~(vi) ci-dessus, le $0$\nobreakdash-cycle~$z_v$ définit,
quel que soit $v \in \Omega$,
une section de la seconde projection $X \times_C \pi^{-1}(t_v) \to \pi^{-1}(t_v)$ pour un $t_v \in \P^1(k_v)$ (à savoir, $t_v=\infty$ si $v \in S$ et $t_v=0$ sinon).
La donnée d'une telle section équivaut à celle d'un point $[z_v] \in W(k_v)$ vérifiant $p([z_v])=t_v$.

\begin{prop}
\mylabel{prop:pointorthogonal}
Le point adélique $([z_v])_{v \in \Omega} \in W'(\A_k)$ est orthogonal au groupe $\Br_\vert(W'/\P^1_k)$ pour l'accouplement de Brauer--Manin.
\end{prop}

Nous prouverons la proposition~\ref{prop:pointorthogonal} à la fin de ce paragraphe pour~$k$ totalement imaginaire et au \textsection\ref{sec:demproppointorthogonal} en général.
Admettons-la pour le moment et terminons la démonstration du théorème~\ref{th:propPvraie}.
Nous aurons besoin pour cela d'une variante de \cite[Theorem~4.1]{ctsksd98}, que nous énonçons avec des notations indépendantes:

\newcommand{\citationctsksdtemp}{$\cong$\cite[Theorem~4.1]{ctsksd98}}
\begin{thm}[\citationctsksdtemp]
\mylabel{th:variantectsksd}
Soit~$X$ une variété irréductible, projective et lisse sur un corps de nombres~$k$.  Soit $f:X\to \P^1_k$ un morphisme de fibre générique géométriquement irréductible,
dont la fibre au-dessus de tout point fermé de~$\P^1_k$ vérifie l'hypothèse~(a) du théorème~\ref{th:principalEun}.
Si $d>0$ est un entier, notons
$\sS^d \subset \Sym_{X/k}^d(\A_k)$
l'ensemble des familles de $0$\nobreakdash-cycles locaux effectifs de degré~$d$ orthogonales à $\Br_\vert(X/\P^1_k)$ pour
l'accouplement de Brauer--Manin
et $\sS^d_\vert \subset \sS^d$ l'ensemble
des familles $(z_v)_{v \in \Omega} \in \Sym_{X/k}^d(\A_k)$
pour lesquelles il existe un point fermé $t \in \P^1_k$, de degré~$d$, tel que la fibre~$X_t$ de~$f$ en~$t$ soit lisse
et que l'égalité $f_*z_v=t$ soit satisfaite dans $Z_0(\P^1_{\mkern-2muk_v}\mkern-.75mu)$ pour tout $v \in \Omega$.
Notons enfin~$d_0$ le nombre de fibres géométriques singulières de~$f$.
Si $d\geq d_0$, alors~$\sS^d_\vert$ est dense dans~$\sS^d$.
\end{thm}

Indiquons brièvement pourquoi
\cite[Theorem~4.1]{ctsksd98} résulte du théorème~\ref{th:variantectsksd}.
S'il existe une famille $(z_v)_{v \in \Omega} \in \prod_{v \in \Omega} \CH_0(X\otimes_k k_v)$ orthogonale à~$\Br(X)$ et vérifiant $\deg(z_v)=1$ pour tout~$v$,
les arguments généraux de réduction aux cycles effectifs que nous avons utilisés au début de la preuve du théorème~\ref{th:propPvraie}
entraînent l'existence, pour tout entier $n>0$, d'un $d>0$ premier à~$n$ tel que $\sS^d\neq\emptyset$. Il suffit alors, pour conclure, d'appliquer le

\begin{cor}
\mylabel{cor:thctsksd}
Dans la situation du théorème~\ref{th:variantectsksd}, supposons que les
fibres de~$f$ au-dessus d'un ouvert dense de~$\P^1_k$ vérifient
l'hypothèse~(b) du théorème~\ref{th:principalEun}.
Si pour tout $n>0$, il existe $d>0$ premier à~$n$ tel que $\sS^d\neq\emptyset$, alors~$X$ possède un $0$\nobreakdash-cycle de degré~$1$.
\end{cor}

\begin{proof}[Démonstration du corollaire~\ref{cor:thctsksd}]
Soit $V \subset \P^1_k$ un ouvert dense au-dessus duquel les fibres de~$f$ vérifient l'hypothèse~(b) du théorème~\ref{th:principalEun}.
Soient $P \in X$ un point fermé et~$d$ un entier premier à $\deg(P)$ tel que $\sS^{d}\neq\emptyset$.
La condition $\sS^{d}\neq\emptyset$ entraîne que $\sS^{d+\mathrm{deg}(P)}\neq\emptyset$.
Quitte à remplacer~$d$ par $d+m\deg(P)$ pour un~$m$ assez grand,
on peut donc supposer que $d\geq d_0$ et $d>\deg(\P^1_k \setminus V)$.
D'après le théorème~\ref{th:variantectsksd}, l'ensemble $\sS^d_\vert$ est alors non vide: il~existe un point fermé $t \in \P^1_k$
de degré~$d$ tel que la fibre~$X_t$ soit lisse et vérifie $X_t(\A_{k(t)})\neq\emptyset$.
Comme $d>\deg(\P^1_k \setminus V)$, le point~$t$ appartient à~$V$ et
par conséquent~$X_t$ admet un $0$\nobreakdash-cycle~$z$ de degré~$1$ sur~$k(t)$.
Une combinaison linéaire de~$z$ et de~$P$ est alors un $0$\nobreakdash-cycle sur~$X$ de degré~$1$ sur~$k$.
\end{proof}

\begin{proof}[Esquisse de démonstration du théorème~\ref{th:variantectsksd}]
Soit $M \subset \P^1_k$ l'ensemble des points fermés au-dessus desquels la fibre de~$f$ est singulière.
Un changement de variables permet de supposer~$M$ contenu dans $\A^1_k \subset \P^1_k$.
Fixons un entier $d \geq d_0$, une famille $(z_v)_{v \in \Omega} \in \sS^d$, un ensemble fini $S_0 \subset \Omega$ arbitrairement grand et, pour chaque $v \in S_0$, un
voisinage $\sU_v \subset \Sym^d_{X/k}(k_v)$
de~$z_v$.  Montrons que~$\sS^d_\vert$ rencontre $\prod_{v \in S_0} \sU_v \times \prod_{v \notin S_0} \Sym^d_{X/k}(k_v)$.
Quitte à remplacer les~$z_v$ pour $v \in S_0$ par d'autres éléments de~$\sU_v$, on peut supposer
que~$z_v$ est supporté par $f^{-1}(\A^1_k \setminus M)$ pour tout $v \in S_0$.
Des classes $A_{i,j} \in \Br(\A^1_k \setminus M)$ sont définies au début de la preuve de~\cite[Theorem~4.1]{ctsksd98}.
Comme $(z_v)_{v \in \Omega}$ est une famille de $0$\nobreakdash-cycles \emph{effectifs}
supportés par $f^{-1}(\A^1_k \setminus M)$ et orthogonaux à $\Br_\vert(X/\P^1_k)$ pour l'accouplement de Brauer--Manin,
on tire de \emph{op.\ cit.}, Lemma~4.5, l'existence
d'un ensemble fini $S_1 \subset \Omega$ contenant~$S_0$ et,
pour chaque $v \in S_1$, d'un $0$\nobreakdash-cycle \emph{effectif}~$z_v^2$ sur $X \otimes_k k_v$, de degré~$d$,
tels que $z_v^2=z_v$ pour $v \in S_0$ et que
\begin{align*}
\sum_{v \in S_1} \inv_v \mylangle A_{i,j}, z_v^2 \myrangle=0
\end{align*}
pour tous~$i$ et~$j$.  Autrement dit,
la condition~(4.3) de~\emph{op.\ cit.} (p.~19) est satisfaite.
Une fois ces cycles~$z_v^2$ construits,
le reste de la démonstration de \emph{op.\ cit.}, Theorem~4.1 s'applique mot pour mot.
Elle fournit
un point fermé $t \in \P^1_k$ de degré~$d$ dont l'image dans $\Sym^d_{\P^1_k/k}(k_v)$ est arbitrairement proche de~$f_*z_v$ pour $v \in S_0$
et qui est
tel que la fibre~$X_t$ soit lisse et vérifie $X_t(\A_{k(t)})\neq\emptyset$.
Pour $v \not\in S_0$, soit $z'_v \in \Sym^d_{X/k}(k_v)$ tel que $f_*z'_v=t$.
Pour $v \in S_0$, quitte à choisir~$t$ suffisamment proche de~$f_*z_v$, on peut supposer,
grâce au théorème des fonctions implicites, qu'il existe $z'_v \in \sU_v$ tel que $f_*z'_v=t$.
La~famille $(z'_v)_{v \in \Omega}$ appartient alors à~$\sS^d_\vert$ et à
$\prod_{v \in S_0} \sU_v \times \prod_{v \notin S_0} \Sym^d_{X/k}(k_v)$.
\end{proof}

Revenons à la démonstration du théorème~\ref{th:propPvraie}.
Nous sommes à présent en position d'appliquer le théorème~\ref{th:variantectsksd} et le corollaire~\ref{cor:thctsksd} à la fibration $p':W'\to\P^1_k$.
En effet, les fibres de~$p'$ vérifient les hypothèses~(a) et~(b) du théorème~\ref{th:principalEun}
en vertu des propositions~\ref{prop:rwproprietes} et~\ref{prop:rwtransmissionhyp}.
L'ensemble~$\sS^1$ du théorème~\ref{th:variantectsksd} étant non vide d'après la proposition~\ref{prop:pointorthogonal},
il résulte du corollaire~\ref{cor:thctsksd} que la variété~$W'$ possède un $0$\nobreakdash-cycle de degré~$1$ sur~$k$.
Il en va donc de même de~$W$ (cf.~\cite[p.~599]{ctfinitudechow}).
Notons $\pr_1:W \times_{\P^1_k}C \to W$ et $\pr_2:W \times_{\P^1_k}C \to C$ les deux projections et $\sigma:W \times_{\P^1_k}C \to X$ le morphisme d'adjonction
et fixons un $0$\nobreakdash-cycle  $w \in Z_0(W)$ de degré~$1$.
Le $0$\nobreakdash-cycle $z = \sigma_*\pr_1^*w \in Z_0(X)$ vérifie alors $f_*z = (f \circ \sigma)_* \pr_1^*w = \pr_{2*} \pr_1^* w = \pi^* p_*w$,
d'où $f_*z=y$ dans $\CH_0(C)$: la propriété~$(\Psub{\emptyset})$ est donc établie.

Supposons l'hypothèse~(b') du théorème~\ref{th:principalE} satisfaite
et démontrons~$(\Psub{S})$.
Fixons pour cela un entier $n>0$ et un point fermé $w_0 \in W$ tel que $p(w_0) \notin \pi(R)$.
Posons $d=1+\deg(w_0)nm$, où~$m$ est assez grand pour que $d>\deg(\pi(M) \cup \pi(R))$.
D'après la proposition~\ref{prop:pointorthogonal},
la famille de $0$\nobreakdash-cycles $(d[z_v])_{v \in \Omega}$ appartient à
l'ensemble~$\sS^d$ du théorème~\ref{th:variantectsksd}
appliqué
au morphisme $p':W'\to \P^1_k$.
La conclusion de ce théorème et le lemme~\ref{lem:continuite}
permettent d'en déduire l'existence d'un point fermé $t \in \P^1_k$
et d'une famille $(z'_v)_{v \in \Omega} \in \Sym^d_{W/k}(\A_k)$ tels que $p'_*z'_v=t$ dans $Z_0(\P^1_{k_v})$ pour tout $v \in \Omega$
et tels que les classes de~$d[z_v]$ et de~$z'_v$ dans $\CH_0(W' \otimes_k k_v)/(2n)$ coïncident pour $v \in S$.
Comme $d>\deg(\pi(M) \cup \pi(R))$, on a $t \in \P^1_k \setminus (\pi(M) \cup \pi(R))$.
D'après la proposition~\ref{prop:rwtransmissionhyp}~(ii),
il existe donc $z' \in Z_0(\Wsub{t})$ ayant même image que $(z'_v)_{v \in S}$
dans $\prod_{v \in S}\CH_0(\Wsub{t} \otimes_k k_v)/(2n\deg(w_0))$.
Le $0$\nobreakdash-cycle~$z'_v$ étant de degré~$d$, il existe
un entier~$b$ tel que $z''=z' + 2nbw_0 \in Z_0(W')$ soit de degré~$1$ sur~$k$.
Remarquons que pour tout $v \in S$, l'image de~$z''$ dans $\CH_0(W' \otimes_k k_v)/(2n)$ est égale à celle de~$[z_v]$.

Soit $U = C \setminus \pi^{-1}(\pi(R))$.
La variété $W \times_{\P^1_k} U$ est lisse, étant étale sur~$W$. Choisissons-en une compactification lisse~$W''$
telle que les restrictions à $W \times_{\P^1_k}U$ du morphisme d'adjonction $\sigma:W \times_{\P^1_k}C \to X$
et de la projection
$\pr_1:W \times_{\P^1_k}C \to W$
s'étendent en des morphismes $\sigma'':W'' \to X$ et
$\pr_1'':W'' \to W'$.
Comme~$\pr_1''$ est plat au-dessus de $p'^{-1}(\P^1_k \setminus \pi(R))$,
la classe $z=\sigma''_*\pr_1''^*z'' \in \CH_0(X)$
vérifie $f_*z=(f \circ \sigma'')_*\pr_1''^*z''=\pi^*p'_*z''=y$ dans $\CH_0(C)$.
De plus, on a $z=z_v$ dans $\CH_0(X\otimes_kk_v)/(2n)$ pour tout $v \in S$
puisque $\sigma''_*\pr_1''^*[z_v]=z_v$ dans $\CH_0(X \otimes_kk_v)$.
Il s'ensuit que $z=z_v$ dans $\CH_0(X\otimes_kk_v)/n$ pour tout $v \in S \cap \Omega_f$
et $z=z_v$ dans $\CH_0(X\otimes_k k_v)/N_{\kbar_v/k_v}(\CH_0(X\otimes_k\kbar_v))$ pour tout $v\in S \cap \Omega_\infty$.
Ainsi la propriété~$(\Psub{S})$ est-elle établie.

\bigskip
Pour compléter la démonstration des théorèmes~\ref{th:principalEun} et~\ref{th:principalE}, il nous faut encore prouver la proposition~\ref{prop:pointorthogonal}.

\begin{proofnoskip}[Démonstration de la proposition~\ref{prop:pointorthogonal} si~$k$ est totalement imaginaire]
L'absence de places réelles est cruciale pour la validité du lemme suivant.

\begin{lem}
\mylabel{lem:totimagpisurj}
Si~$k$ est totalement imaginaire, toute classe $A \in \Br(\P^1_k \setminus \pi(M))$ s'écrit $A=\pi_*A' + \delta$
pour un $A' \in \Br(C \setminus M)$ et un $\delta \in \Br(k)$,
où~$\pi_*$ désigne
l'application trace
(ou corestriction)
$\pi_*:\Br(C\setminus M) \to \Br(\P^1_k \setminus \pi(M))$ induite par~$\pi$.
\end{lem}

\begin{proof}
Notons
$j:\P^1_k \setminus \pi(M) \hookrightarrow \P^1_k$ et $j' : C\setminus M \hookrightarrow C$
les injections canoniques.
Comme~$\pi$ induit un isomorphisme $M \isoto \pi(M)$, les homomorphismes norme $N_{C/\P^1_k}:\pi_*\Gm\to\Gm$ et $N_{C/\P^1_k}:\pi_*j'_*\Gm\to j_*\Gm$
(cf.~\cite[6.5]{ega2})
s'inscrivent dans un diagramme commutatif de faisceaux étales sur~$\P^1_k$
$$
\xymatrix{
0 \ar[r] & \pi_*\Gm \ar[d]^{N_{C/\P^1_k}} \ar[r] & \pi_* j'_*\Gm \ar[d]^{N_{C/\P^1_k}} \ar[r] &
\vphantom{i_{\pi(m)*}}\smash[b]{\displaystyle\bigoplus_{m \in M} i_{\pi(m)*} \Z} \ar@{=}[d] \ar[r] & 0 \\
0 \ar[r] & \Gm \ar[r] & j_*\Gm \ar[r] & \vphantom{\Z}\smash[t]{\displaystyle\bigoplus_{m \in M} i_{\pi(m)*}\Z}\ar[r] & 0
}
$$
dont les lignes sont exactes.
Compte tenu que $\Br(\P^1_k)=\Br(k)$, que
$H^2(\P^1_k,j_*\Gm)=\Br(\P^1_k\setminus\pi(M))$
et que $H^2(C,j'_*\Gm)=\Br(C\setminus M)$
(cf.~\cite[Lemme~3.19]{owlnm}),
on obtient, en passant aux sections globales, un diagramme commutatif
$$
\xymatrix{
\Br(C) \ar[d] \ar[r] & \Br(C\setminus M) \ar[d]^{\pi_*} \ar[r] &
\vphantom{H^1(\Q/\Z)}\smash[b]{\displaystyle\bigoplus_{m \in M} H^1(k(m), \Q/\Z)} \ar@{=}[d] \ar[r] & H^3(C,\Gm)\ar[d]^{\pi_*} \\
\Br(k) \ar[r] & \Br(\P^1_k \setminus \pi(M)) \ar[r] &
\vphantom{H^1(\Q/\Z)}\smash[t]{\displaystyle\bigoplus_{m \in M} H^1(k(m), \Q/\Z)} \ar[r] & H^3(\P^1_k,\Gm)
}
$$
dont les lignes sont exactes.  Ainsi suffit-il,
pour démontrer le lemme,
d'établir l'injectivité
de $\pi_*:H^3(C,\Gm)\to H^3(\P^1_k,\Gm)$.

Comme $H^q(k,\Gm)=0$ pour tout $q\geq 3$
(cf.~\cite[Ch.~I, Theorem~4.20~(b)]{milneadt})
et que $H^q(C \otimes_k \kbar,\Gm)=0$ et $H^q(\P^1_\kbar,\Gm)=0$ pour $q \geq 2$ d'après le théorème de Tsen,
la suite spectrale de Hochschild--Serre fournit des isomorphismes canoniques $H^3(C,\Gm)=H^2(k,\Pic(C \otimes_k \kbar))$ et $H^3(\P^1_k,\Gm)=H^2(k,\Pic(\P^1_{\kbar}))$.
L'application $\pi_*:H^3(C,\Gm) \to H^3(\P^1_k,\Gm)$ s'identifie donc à la flèche obtenue en appliquant le foncteur $H^2(k,-)$ à l'homomorphisme
$\deg:\Pic(C \otimes_k\kbar)\to \Z$.
Il s'ensuit que son noyau est isomorphe à $H^2(k,\Pic^0(C \otimes_k \kbar))$.
Or
ce dernier groupe est nul,
d'après un théorème de Tate (cf.~\emph{op.\ cit.}, Ch.~I, Corollary~6.24),
car
$\Pic^0(C \otimes_k\kbar)$ est le groupe des $\kbar$\nobreakdash-points d'une variété abélienne et~$k$ est totalement imaginaire.
\end{proof}

Tout élément de $\Br_\vert(W'/\P^1_k)$ s'écrit $p'^*A$ pour un $A \in \Br(\P^1_k \setminus \pi(M))$
puisque les fibres de~$p'$ au-dessus de $\P^1_k \setminus \pi(M)$ contiennent toutes une composante irréductible géométriquement irréductible de multiplicité~$1$
(cf.~proposition~\ref{prop:rwproprietes}
et~\cite[Proposition~1.1.1]{ctsd94}).  Fixons un tel élément et
supposons~$k$ totalement imaginaire.
D'après le lemme~\ref{lem:totimagpisurj}, il existe $A' \in \Br(C \setminus M)$ et $\delta \in \Br(k)$ tels que $A=\pi_*A'+\delta$.

La proposition~\ref{prop:rwproprietes} fournit, pour $m \in M$, une bijection naturelle de l'ensemble des composantes irréductibles de $f^{-1}(m)$ sur l'ensemble des composantes
irréductibles de $p'^{-1}(\pi(m))$.
Cette bijection respecte les multiplicités et pour tout $m \in M$, la fermeture algébrique de~$k(m)$
dans chaque composante irréductible de~$f^{-1}(m)$ est $k(m)$\nobreakdash-isomorphe à la fermeture algébrique de~$k(m)$
dans la composante irréductible correspondante de~$p'^{-1}(\pi(m))$.
Il s'ensuit que la classe $f^*A' \in \Br(f^{-1}(C \setminus M))$ appartient au sous-groupe $\Br_\vert(X/C)$ (cf.~\cite[Proposition~1.1.1]{ctsd94}).
Comme par hypothèse la famille $(z_v)_{v \in \Omega}$ est orthogonale à $\Br_\vert(X/C)$, on a donc
\begin{align}
\mylabel{eq:zvorth}
\sum_{v \in \Omega} \inv_v \mylangle f^*A', z_v \myrangle=0\rlap{\text{.}}
\end{align}
D'autre part,
comme $f_*z_v=\pi^*p'_*[z_v]$, on a par adjonction
\begin{align*}
\sum_{v \in \Omega} \inv_v \mylangle f^*A', z_v \myrangle&=
\sum_{v \in \Omega} \inv_v \mylangle A', f_*z_v \myrangle=
\sum_{v \in \Omega} \inv_v \mylangle A', \pi^*p'_*[z_v]\myrangle\\&=
\sum_{v \in \Omega} \inv_v \mylangle p'^*\pi_*A', [z_v]\myrangle=
\sum_{v \in \Omega} \inv_v \mylangle p'^*A, [z_v]\myrangle\rlap{\text{,}}
\end{align*}
où la dernière égalité résulte de la loi de réciprocité globale.
Compte tenu de~(\ref{eq:zvorth}), cela conclut la démonstration de la proposition~\ref{prop:pointorthogonal} sous l'hypothèse que~$k$ est totalement imaginaire.
\end{proofnoskip}

\section{Places réelles: groupes \texorpdfstring{$\Picplus$}{Pic+} et \texorpdfstring{$\Brplus$}{Br+}}
\mylabel{sec:vanhamel}

\subsection{Introduction}

Ce paragraphe développe et adapte les idées originales de van~Hamel~\cite{vanhamel}
grâce auxquelles il a étendu
aux corps de nombres formellement réels
les résultats principaux de~\cite{ctreglees} et de~\cite{frossard}.

Soient~$X$ et~$C$ des variétés propres, lisses et géométriquement irréductibles sur un corps~$k$ de caractéristique nulle et $f:X \to C$ un morphisme
de fibre générique géométriquement irréductible.
Notons~$\eta$ le point générique de~$C$
et $\Brplus(C) \subset \Br(\eta)$ l'image réciproque de $\Br(X)$ par l'application $f^*:\Br(\eta) \to \Br(X_\eta)$.
Ce groupe est équipé d'une surjection $f^*_+:\Brplus(C) \to \Br_\vert(X/C)$, par laquelle se factorise l'application $f^*:\Br(C) \to \Br_\vert(X/C)$.
Rappelons d'autre part que l'on dispose d'un accouplement naturel $\Br(X) \times \CH_0(X) \to \Br(k)$ (cf.~(\ref{eq:accbrch})).

Dans ce paragraphe, nous allons définir, sous l'hypothèse que~$C$ est une courbe,
un groupe abélien $\Picplus(C)$, un accouplement $\Brplus(C) \times \Picplus(C) \to \Br(k)$
et un homomorphisme $f_{*+}:\CH_0(X) \to \Picplus(C)$ adjoint à droite de~$f^*_+$ par rapport aux deux accouplements considérés et
par lequel
$f_*:\CH_0(X) \to \CH_0(C)=\Pic(C)$
se factorise naturellement.

Ainsi, si~$k$ est un corps de nombres et si $(z_v)_{v \in \Omega}$ est une famille de $0$\nobreakdash-cycles locaux sur~$X$, la condition
\og{}$(z_v)_{v \in \Omega}$
est orthogonale à $\Br_\vert(X/C)$ pour l'accouplement de Brauer--Manin\fg{}
se reformule purement en termes de la famille image $(f_{*+}z_v)_{v \in \Omega} \in \prod_{v \in \Omega} \Picplus(C \otimes_k k_v)$
et du groupe $\Brplus(C)$.  L'intérêt d'une telle reformulation est que l'arithmétique des groupes $\Picplus(C)$ et $\Picplus(C \otimes_k k_v)$
est moins mystérieuse que celle des groupes $\CH_0(X)$ et
$\CH_0(X \otimes_k k_v)$: elle ne fait intervenir que la cohomologie galoisienne des tores, des variétés abéliennes et des modules galoisiens de type fini.
En particulier est-il possible
de caractériser
l'image de l'application naturelle $\Picplus(C) \to \prod_{v \in \Omega} \Picplus(C \otimes_k k_v)$
en établissant
l'analogue de la conjecture~$(E)$ pour le groupe $\Picplus(C)$
à partir des théorèmes standard de dualité arithmétique
(cf.~théorème~\ref{th:dualitearithmetique} ci-dessous).

Cette caractérisation nous a permis, dans la preuve des théorèmes~\ref{th:principalEun} et~\ref{th:principalE},
de supposer que la famille $(f_{*+}z_v)_{v \in \Omega} \in\prod_{v \in \Omega}\Picplus(C \otimes_k k_v)$
est l'image d'un élément de~$\Picplus(C)$,
quitte à modifier le choix des~$z_v$ aux places réelles
(hypothèse~(v) du paragraphe~\ref{sec:preuveP}).
Au paragraphe~\ref{sec:demproppointorthogonal}, nous en déduirons la validité de la proposition~\ref{prop:pointorthogonal} sans supposer~$k$ totalement imaginaire.
Plus précisément, en étudiant le comportement des groupes~$\Picplus$ par restriction des scalaires à la Weil, nous vérifierons que
si $(f_{*+}z_v)_{v \in \Omega}$ est l'image d'un élément de $\Picplus(C)$, alors
 $(p'_{*+}[z_v])_{v \in \Omega} \in \prod_{v \in \Omega} \Picplus(\P^1_{k_v})$
est l'image d'un élément de $\Picplus(\P^1_k)$.  Par adjonction il s'ensuivra que $([z_v])_{v \in \Omega}$ est
orthogonale à $p'^*_+\Brplus(\P^1_k)=\Br_\vert(W'/\P^1_k)$ pour l'accouplement de Brauer--Manin.

\subsection{Notations, définitions}
\mylabel{subsec:vhnotations}

Soit~$C$ une courbe propre, lisse et géométriquement irréductible sur un corps~$k$.
Supposons donné un morphisme $f:X\to C$ de fibre générique lisse et géométriquement irréductible, où~$X$ est une variété irréductible et lisse.
Notons $M \subset C$ l'ensemble des points fermés de~$C$ au-dessus desquels la fibre de~$f$ est singulière.
Pour $m \in M$, notons $(\Ysub{i})_{i \in I_m}$ la famille des composantes irréductibles de la fibre~$X_m$.
Pour $m \in M$ et $i \in I_m$, soit enfin $e_{m,i}$ la multiplicité de~$\Ysub{i}$ dans~$X_m$ et $K_{m,i}$ la fermeture algébrique de~$k(m)$ dans le corps des fonctions de~$\Ysub{i}$.

Si~$S$ est un schéma, notons $\Sm/S$ le site lisse-étale de~$S$, c'est-à-dire
la catégorie des $S$\nobreakdash-schémas lisses munie de la topologie étale.
On dispose d'une suite exacte de faisceaux sur $\Sm/C$
\begin{align}
\mylabel{se:weilsmc}
\xymatrix{
0\ar[r] & \Gm\ar[r] &j_*\Gm \ar[r]& \displaystyle\bigoplus_{m \in M} i_{m*}\Z\ar[r]&0
}
\end{align}
où $j: C \setminus M \hookrightarrow C$ et $i_m:\Spec(k(m)) \hookrightarrow C$ désignent les injections canoniques.
Pour $m \in M$, le
$\Z$\nobreakdash-module libre sur l'ensemble des composantes irréductibles de la fibre géométrique de~$f$ au-dessus de~$m$
définit, comme tout module galoisien, un schéma en groupes commutatifs étale sur~$k(m)$ et donc un faisceau en groupes abéliens sur~$\Sm/k(m)$; nous le noterons $\Psub{m}$.
Notons de plus $\alpha_m:i_{m*}\Z \to i_{m*}\Psub{m}$ le morphisme qui applique~$1$ sur la famille des multiplicités~$e_{m,i}$.
En composant la seconde flèche de~(\ref{se:weilsmc}) avec $\bigoplus\alpha_m$,
on obtient un morphisme
\begin{align}
\mylabel{eq:vplus}
j_*\Gm \longrightarrow \bigoplus_{m \in M}i_{m*}\Psub{m}\rlap{\text{.}}
\end{align}
Nous considérerons~(\ref{eq:vplus}) comme un complexe de faisceaux en groupes abéliens sur~$\Sm/C$ concentré en degrés~$0$ et~$1$
et noterons celui-ci~$\vplus(C)$.
Désignant par~$\rho$ le morphisme structural $\rho:C \to \Spec(k)$,
définissons des groupes abéliens $\Brplus(C)$ et $\Picplus(C)$ par les formules
$$
\Brplus(C)= H^2(k,\tau_{\leq 1}R\rho_* \vplus(C))
$$
et
$$\Picplus(C)=H^0(k,\RHom(\tau_{\leq 1}R\rho_*\vplus(C),\Gm))$$
(où le foncteur~$\RHom$ est considéré dans la catégorie dérivée des faisceaux abéliens sur~$\Sm/k$).
Le cup-produit fournit un accouplement canonique
\begin{align}
\mylabel{eq:accbrpluspicplus}
\Brplus(C) \times \Picplus(C) \to \Br(k)\rlap{\text{.}}
\end{align}

Rappelons que pour tout schéma~$S$, la cohomologie d'un faisceau~$\sF$ sur $\Sm/S$ coïncide avec la cohomologie de la restriction de~$\sF$ au petit site étale de~$S$
(cf.~\cite[III, Proposition~3.1]{milneet}). Le choix du site lisse-étale au lieu du petit site étale ne joue donc ici un rôle que dans le calcul du foncteur~$\RHom$.
Rappelons d'autre part que la formule des coefficients universels de Deligne~\cite[Théorème~1.5.2]{delignedualite}
fournit un isomorphisme canonique
\begin{align}
\mylabel{eq:coeffuniv}
\tau_{\leq 0}\RHom(\tau_{\leq 1}R\rho_*\Gm,\Gm) = \left(\tau_{\leq 1}R\rho_*\Gm\right)[1]
\end{align}
dans la catégorie dérivée des faisceaux en groupes abéliens sur~$\Sm/k$ (cf.~\cite[\textsection3.3]{vanhamellichtenbaumtate}).
La flèche naturelle $\Gm \to \vplus(C)$ et l'isomorphisme~(\ref{eq:coeffuniv}) induisent un morphisme d'\og{}oubli\fg{}
$\phi:\Picplus(C) \to \Pic(C)$ (dont nous verrons plus bas qu'il est surjectif).

\begin{prop}
\mylabel{prop:fetoileplus}
Il existe des flèches canoniques $f_{*+}:Z_0(X)\to \Picplus(C)$
et $f_+^*:\Brplus(C) \to \Br_\vert(X/C)$ telles que le diagramme
\begin{align}
\mylabel{diag:compatibleacc}
\begin{aligned}
\myxyin
\xymatrix@C=10ex{
\ar@<6ex>[d]^{f_{*+}}
\Br_\vert(X/C) \times \rlap{$Z_0(X)$}\phantom{\Picplus(C)}
\save+<10.125ex,0ex>\ar[r]^(.45){\mylangle-,-\myrangle}\restore
& \Br(k) \ar@{=}[d] \\
\ar@<2.06ex>[u]^{f^*_+}
\mkern36.3mu\rlap{$\Brplus(C)$}\mkern-36.3mu\phantom{\Br_\vert(X/C)} \times \Picplus(C)
 \ar[r]^(.67){\text{(\ref{eq:accbrpluspicplus})}} & \Br(k)
}
\myxyout
\end{aligned}
\end{align}
commute et que la composée $\phi \circ f_{*+}$ coïncide avec $f_*:Z_0(X) \to \CH_0(C)=\Pic(C)$.
\end{prop}

\begin{proof}
Notant encore~$j$ l'inclusion de $f^{-1}(C \setminus M)$ dans~$X$, on vérifie que la suite exacte des diviseurs de Weil sur~$X$ induit un triangle distingué
$$
\xymatrix{
\tau_{\leq 1}Rf_*\Gm \ar[r] & \tau_{\leq 1}Rf_*(j_*\Gm) \ar[r] & \displaystyle\bigoplus_{m \in M}i_{m*}\Psub{m} \ar[r]^(.72){+1} & \vphantom{.}
}
$$
dans la catégorie dérivée des faisceaux en groupes abéliens sur~$\Sm/C$.
Celui-ci s'inscrit dans un diagramme commutatif (sans la flèche en pointillés)
$$
\xymatrix@R=3ex{
\displaystyle\smash[b]{\bigoplus_{m \in M}i_{m*}\Psub{m}}[-1]\vphantom{\lower 1.5ex\hbox{}}
\ar@{=}[d] \ar[r] &
\vplus(C) \ar@{.>}[d] \ar[r] & j_*\Gm \ar[r]\ar[d] & \displaystyle
\smash[b]{\bigoplus_{m \in M}i_{m*}\Psub{m}}\vphantom{\lower 1.5ex\hbox{}}
\ar@{=}[d] \\
\displaystyle\smash[b]{\bigoplus_{m \in M}i_{m*}\Psub{m}}[-1]\vphantom{\lower 1.5ex\hbox{}}
 \ar[r] &
\tau_{\leq 1}Rf_*\Gm \ar[r] & \tau_{\leq 1}Rf_*(j_*\Gm) \ar[r] & \displaystyle\bigoplus_{m \in M}i_{m*}\Psub{m}\rlap{\text{.}}
}
$$
Le
groupe $\Hom(j_*\Gm, \bigoplus_{m\in M}i_{m*}\Psub{m}[-1])$ étant nul (pour des raisons de degré),
il existe une unique flèche en pointillés rendant ce diagramme commutatif.
Cette flèche induit,
\emph{via} les foncteurs $H^2(k,\tau_{\leq 1}R\rho_*-)$
et $H^0(k,\RHom(\tau_{\leq 1}R\rho_*-,\Gm))$,
des applications canoniques
$f^*_+:\Brplus(C) \to \Br_\vert(X/C)$
et
\begin{align}
\mylabel{eq:appcanvhpicplus}
\Hom(R(\rho\circ f)_*\Gm,\Gm) \to \Picplus(C)\rlap{\text{.}}
\end{align}
Notons~$f_{*+}$ la composée de~(\ref{eq:appcanvhpicplus}) et de
l'application
classe de cycle
$$c:Z_0(X) \to \Hom(R(\rho\circ f)_*\Gm,\Gm)$$
construite par van~Hamel dans~\cite[\textsection3.1]{vanhamellichtenbaumtate}.
Il résulte de \emph{op.\ cit.}, Lemma~3.1 et de la fonctorialité de~$c$ que~$f_{*+}$ et $f^*_+$ remplissent les conditions voulues.
\end{proof}

\begin{rmqs}
\mylabel{rq:picplusbrpluscarzero}
(i) Supposons~$k$ de caractéristique~$0$.  Alors $R^q\rho_*\Gm=0$ pour $q \geq 2$ d'après le théorème de Tsen
et le théorème de changement de base propre en cohomologie étale (cf.~\cite[p.~12]{vanhamellichtenbaumtate}).
Par conséquent $\tau_{\leq 1}R\rho_*\vplus(C)=R\rho_*\vplus(C)$ et donc $\Brplus(C)=H^2(C,\vplus(C))$.
Comme $H^2(C,j_*\Gm)=\Br(C\setminus M)$ (cf.~\cite[Lemme~3.19]{owlnm}), il s'ensuit que~$\Brplus(C)$ est le noyau de
l'application
$$
\Br(C \setminus M) \longrightarrow \bigoplus_{m \in M} \bigoplus_{i \in I_m} H^1(K_{m,i},\Q/\Z)
$$
composée de l'application résidu $\Br(C \setminus M) \to \bigoplus_{m \in M}H^1(k(m),\Q/\Z)$
et de la flèche induite par~$\bigoplus \alpha_m$.
D'après le diagramme~(\ref{diag:finitudebrvert}) (avec $Y=C$ et $Y^0=C \setminus M$), le groupe $\Brplus(C)$ s'identifie donc à l'ensemble des éléments de $\Br(C \setminus M)$
dont l'image réciproque dans $\Br(f^{-1}(C \setminus M))$ appartient au sous-groupe $\Br(X) \subset \Br(f^{-1}(C \setminus M))$.
En particulier l'application $f^*_+:\Brplus(C) \to \Br_\vert(X/C)$ est-elle surjective.

(ii) Si $M=\emptyset$, alors $\Picplus(C)=\Pic(C)$, le morphisme d'oubli~$\phi$ étant un isomorphisme.
Si de plus~$k$ est de caractéristique~$0$,
on a aussi $\Brplus(C)=\Br(C)$ d'après la remarque~\ref{rq:picplusbrpluscarzero}~(i).

(iii) Toujours sous l'hypothèse que~$k$ est de caractéristique~$0$, il résulte
de~\cite[Proposition~3.2]{vanhamellichtenbaumtate} que si~$X$ est propre, l'application $f_{*+}:Z_0(X) \to \Picplus(C)$ se factorise par $\CH_0(X)$.
\end{rmqs}

\subsection{Théorème de dualité arithmétique}
\mylabel{sec:thdualitearithmetique}

Supposons que~$k$ soit un corps de nombres. Soit~$\sC$ un modèle de~$C$ au-dessus d'un ouvert dense~$U$ du spectre de l'anneau des entiers de~$k$.
Si~$U$ est choisi assez petit,
les constructions du paragraphe~\ref{subsec:vhnotations} peuvent être effectuées
au-dessus de~$U$. Nous noterons
$\sF=\RHom(\tau_{\leq 1}R\rho_*\vplus(\sC),\Gm)$ l'objet de la catégorie dérivée des faisceaux en groupes abéliens sur $\Sm/U$
ainsi obtenu. En toute rigueur, $\rho$ désigne donc ici le morphisme structural du $U$\nobreakdash-schéma~$\sC$, etc.

Pour $v \in U$, posons
 $\Brplus(\sC \otimes_U \sO_v)=H^2(\sO_v,\tau_{\leq 1}R\rho_*\vplus(\sC))$.
Posons de plus
$\Picplus(\sC \otimes_U \sO_v)=H^0(\sO_v,\sF)$ et
 $\Picpluschapeau(\sC \otimes_U \sO_v)=\widehat{H^0(\sO_v,\sF)}$ et
notons $\PicplusAchapeau(C)$ le groupe
$$
\prod_{v \in \Omega_f}\smash{\rlap{\raise 5pt\hbox{$\mkern-7mu{}'$}}} \Picpluschapeau(C \otimes_k k_v) \times \prod_{v \in \Omega_\infty}\Coker\mkern-1mu\left(N_{\kbar_v/k_v}\!:
\Picplus(C \otimes_k \kbar_v) \to \Picplus(C \otimes_k k_v)\right)\!\rlap{\text{,}}
$$
où le symbole~$'$ désigne le produit restreint des $\Picpluschapeau(C\otimes_kk_v)$
par rapport aux images des flèches de restriction $\Picpluschapeau(\sC \otimes_U \sO_v) \to \Picpluschapeau(C \otimes_kk_v)$.
Autrement dit, les éléments de $\prod_{v \in \Omega_f}' \Picpluschapeau(C \otimes_k k_v)$
sont les familles
$(u_v)_{v \in \Omega_f} \in \prod_{v \in \Omega_f} \Picpluschapeau(C \otimes_k k_v)$ telles que
seul un nombre fini des~$u_v$ pour $v \in U$ ne proviennent pas de $\Picpluschapeau(\sC\otimes_U\sO_v)$.

Comme les groupes $\Brplus(C \otimes_k k_v)$ sont de torsion (cf.~remarque~\ref{rq:picplusbrpluscarzero}~(i)),
l'accouplement~(\ref{eq:accbrpluspicplus}) induit pour chaque $v \in \Omega$
un accouplement canonique
\begin{align}
\mylabel{eq:accbrpluspicpluslocal}
\Brplus(C \otimes_k k_v) \times \Picpluschapeau(C \otimes_k k_v) \to \Br(k_v)\rlap{\text{.}}
\end{align}
Celui-ci s'annule sur l'image de $\Brplus(\sC \otimes_U \sO_v) \times \Picpluschapeau(\sC \otimes_U \sO_v)$
pour tout $v \in U$,
puisque $\Br(\sO_v)=0$.
La somme, sur toutes les places $v \in \Omega$, des accouplements~(\ref{eq:accbrpluspicpluslocal}) suivis de l'invariant local $\inv_v:\Br(k_v) \hookrightarrow \Q/\Z$
détermine donc un accouplement canonique bien défini
\begin{align}
\mylabel{eq:accbrpluspicplusachapeau}
\Brplus(C) \times \PicplusAchapeau(C) \to \Q/\Z \rlap{\text{.}}
\end{align}
Enfin, posons $\Picchapeau(C)=\widehat{\Pic(C)}$ et $\PicAchapeau(C)=\CHzAchapeau(C)$.

\begin{thm}
\mylabel{th:dualitearithmetique}
Supposons que~$k$ soit un corps de nombres.
\begin{enumerate}
\item[(i)] L'image de l'application naturelle
$\Picpluschapeau(C) \to \prod_{v \in \Omega_f} \Picpluschapeau(C\otimes_kk_v)$
est incluse dans
$\prod_{v \in \Omega_f}' \Picpluschapeau(C \otimes_k k_v)$.
\item[(ii)] Si le groupe de Tate--Shafarevich de la jacobienne de~$C$ ne contient pas d'élément infiniment divisible non nul, le complexe
\begin{align*}
\xymatrix{
\Picpluschapeau(C) \ar[r] & \PicplusAchapeau(C) \ar[r] & \Hom(\Brplus(C),\Q/\Z)
}
\end{align*}
est une suite exacte.
\item[(iii)]
Sans hypothèse sur le groupe de Tate--Shafarevich de la jacobienne de~$C$, le groupe d'homologie du complexe ci-dessus s'injecte, \emph{via}~$\phi$, dans le groupe d'homologie
du complexe
\begin{align*}
\xymatrix{
\Picchapeau(C) \ar[r] & \PicAchapeau(C) \ar[r] & \Hom(\Br(C),\Q/\Z)\rlap{\text{.}}
}
\end{align*}
\end{enumerate}
\end{thm}

Notons que c'est l'assertion~(i) du théorème qui donne un sens à la première flèche du complexe apparaissant dans~(ii); la seconde flèche est définie par~(\ref{eq:accbrpluspicplusachapeau}).
D'autre part, l'assertion~(ii) est une conséquence immédiate de~(iii) et du théorème selon lequel le complexe apparaissant dans~(iii) est une suite exacte
si le groupe de Tate--Shafarevich de la jacobienne de~$C$ ne contient pas d'élément infiniment divisible non nul (cf.~remarque~\ref{rqs:conje}~(iv)).
Il nous suffira donc d'établir~(i) et~(iii) pour prouver le théorème.

\begin{proof}
Soient $T^*=\bigoplus_{m \in M}\rho_*\!\Coker(\alpha_m)$ et $Q^*=\bigoplus_{m \in M}(\rho \circ i_m)_*\Z$.
Quitte à rétrécir~$U$, on peut supposer que~$\rho \circ i_m$ est le morphisme structural d'un $U$\nobreakdash-schéma fini étale pour tout~$m$,
que~$T^*$ est un schéma en groupes sur~$U$ localement constant pour la topologie étale
et que $T=\RHom(T^*,\Gm)$ est représenté par un schéma en groupes de type multiplicatif.
Notons~$Q$ le tore quasi-trivial $\RHom(Q^*,\Gm)=\prod_{m \in M} (\rho \circ i_m)_*\Gm$.

\begin{lem}
\mylabel{lem:tridistdescvplus}
Dans la catégorie dérivée des faisceaux en groupes abéliens sur~$\Sm/U$,
il existe un triangle distingué canonique
\begin{align}
\mylabel{eq:tridistdescvplus}
\xymatrix{
\left(\tau_{\leq 0}\sF\right)[-1] \ar[r] & \tau_{\leq 1}R\rho_*\Gm \ar[r]^(.6)\alpha & T[1] \ar[r]^(.65){+1} & \vphantom{.}
}
\end{align}
tel que~$\alpha$ se factorise par~$Q$.
\end{lem}

\begin{proof}
De la définition de~$\vplus(\sC)$ et de la suite exacte~(\ref{se:weilsmc}),
on tire un triangle distingué
\begin{align}
\mylabel{eq:tridist}
\xymatrix{
\Gm \ar[r] & \vplus(\sC) \ar[r] & \displaystyle\bigoplus_{m \in M} \Coker(\alpha_m)[-1] \ar[r]^(.8){+1} & \vphantom{.}
}
\end{align}
dans la catégorie dérivée des faisceaux en groupes abéliens sur~$\Sm/\sC$, tel que la flèche
$\bigoplus_{m \in M} \Coker(\alpha_m)[-1] \to \Gm[1]$ issue
de ce triangle se factorise par $\bigoplus_{m\in M}i_{m*}\Z$.
Le~complexe $R\rho_*(\bigoplus_{m \in M} i_{m*}\Z)=Q^*$ étant concentré en degré~$0$,
on peut appliquer le foncteur~$\tau_{\leq 1}R\rho_*$ à~(\ref{eq:tridist}) et obtenir
un triangle distingué
\begin{align}
\mylabel{eq:tridistbrplus}
\xymatrix{
\tau_{\leq 1}R\rho_*\Gm \ar[r] & \tau_{\leq 1}R\rho_*\vplus(\sC) \ar[r] & T^*[-1] \ar[r]^(.65){\smash[t]{+1}} &\rlap{\text{.}}
}
\end{align}
Celui-ci induit, dualement, un triangle distingué
\begin{align}
\mylabel{eq:trin}
\xymatrix{
\sF \ar[r] & \RHom(\tau_{\leq 1}R\rho_*\Gm,\Gm) \ar[r]^(.73)\beta & T[2] \ar[r]^(.65){+1} & \vphantom{.}
}
\end{align}
tel que~$\beta$ se factorise par~$Q[1]$.
Comme~$T[2]$ est concentré en degré~$-2$, on tire finalement de~(\ref{eq:trin})
et de~(\ref{eq:coeffuniv})
le triangle distingué~(\ref{eq:tridistdescvplus})
(avec $\alpha=(\tau_{\leq 0}\beta)[-1]$).
\end{proof}

\begin{lem}
\mylabel{lem:descpicplus}
Si~$A$ désigne une extension de~$k$ ou l'un des anneaux~$\sO_v$ pour $v \in U$,
il existe une suite exacte canonique
\begin{align}
\mylabel{se:descpicplus}
\xymatrix@C=2.1em{
H^0(A,\Gm) \ar[r] & H^1(A,T) \ar[r] & \Picplus(\sC \otimes_U A) \ar[r] & \Pic(\sC \otimes_U A) \ar[r] & 0
}
\end{align}
fonctorielle en~$A$.
\end{lem}

\begin{proof}
Cette suite exacte se déduit du lemme~\ref{lem:tridistdescvplus}
car $H^1(A,Q)$ est nul; la nullité de $H^1(A,Q)$ vient de ce que
$\Pic(B)=0$ pour toute $A$\nobreakdash-algèbre finie étale~$B$.
\end{proof}

Revenons à la démonstration du théorème~\ref{th:dualitearithmetique}.
Comme~$T$ est un schéma en groupes de type multiplicatif sur~$U$, on peut supposer,
quitte à rétrécir~$U$,
 que~$T$ est lisse sur~$U$
et
que $H^1(k_v,T)$ est
fini pour tout $v \in U$ (cf.~\cite[Ch.~I, Corollary~2.4]{milneadt}).
La lissité de~$T$ entraîne que $H^1(\sO_v,T)=H^1(\Fv,T)$
pour tout $v \in U$ (cf.~\cite[Théorème~11.7]{grbr3}),
de sorte que le groupe $H^1(\sO_v,T)$ est lui aussi fini.
Il en résulte que la suite~(\ref{se:descpicplus}) pour $A=k_v$ ou $A=\sO_v$ avec $v\in U$,
privée de son premier terme,
reste exacte si l'on applique le foncteur $M \mapsto \widehat{M}$ à ses deux derniers termes.
D'où un diagramme commutatif à lignes exactes
\begin{align}
\mylabel{diag:ovkvpicplus}
\begin{aligned}
\xymatrix{
H^1(\sO_v,T) \ar[d] \ar[r] & \Picpluschapeau(\sC \otimes_U \sO_v) \ar[r] \ar[d] & \Picchapeau(\sC \otimes_U \sO_v) \ar[r] \ar[d]^\wr & 0 \\
H^1(k_v,T) \ar[r] & \Picpluschapeau(C \otimes_k k_v) \ar[r] & \Picchapeau(C\otimes_k k_v) \ar[r] & 0\rlap{\text{,}}
}
\end{aligned}
\end{align}
dans lequel la flèche verticale de droite est un isomorphisme puisque~$\sC$ est propre et lisse sur~$U$.
Le schéma en groupes de type multiplicatif~$T$ sur~$U$ est extension d'un schéma en groupes fini~$\Tsub{f}$ par un tore~$\Tsub{t}$.
Soit~$n$ le produit de l'exposant de~$\Tsub{f}$ et du degré d'une extension finie de~$k$ déployant $\Tsub{t} \otimes_Uk$.
Par le théorème~90 de Hilbert
et un argument de trace, l'entier~$n$ annule $H^1(k_v,T)$ pour tout $v \in U$.
Grâce au diagramme~(\ref{diag:ovkvpicplus}), on en déduit que le conoyau de la flèche de restriction
$\Picpluschapeau(\sC \otimes_U \sO_v) \to \Picpluschapeau(C \otimes_k k_v)$
est annulé par~$n$ pour tout $v \in U$.

Vérifions maintenant l'assertion~(i) du théorème.
Compte tenu du lemme~\ref{lem:divisible},
tout élément de $\Picpluschapeau(C)$ s'écrit $x+ny$ avec $x \in \Picplus(C)$ et $y \in \Picpluschapeau(C)$.
D'après ce qui précède, l'image de~$ny$ dans
$\prod_{v \in \Omega_f} \Picpluschapeau(C\otimes_kk_v)$
appartient à
$\prod_{v \in \Omega_f}' \Picpluschapeau(C \otimes_k k_v)$.
Il en va de même pour l'image de~$x$, puisque $\Picplus(C) = \varinjlim_{V\subset U} \Picplus(\sC \otimes_U V)$ si~$V$ parcourt les ouverts non vides de~$U$
(cf.~\cite[Corollaire~5.8]{grothsitetoposetale}).
D'où l'assertion~(i).

Avant de démontrer l'assertion~(iii), introduisons quelques notations.
Si~$\sG$ est un complexe borné de faisceaux étales en groupes abéliens sur~$U$,
notons $\Htilde^i(k_v,\sG)$ pour $v \in \Omega_\infty$
les groupes d'hypercohomologie de Tate de~$\sG$ (cf.~\cite[p.~324]{vanhamel} pour la définition).
Ces groupes sont nuls si~$v$ est complexe; ils sont annulés par~$2$ si~$v$ est réelle.
Pour $v \in \Omega_f$, posons $\Htilde^i(k_v,\sG)=H^i(k_v,\sG)$.
Enfin, notons $\prod'_{v \in \Omega} \Htilde^i(k_v,\sG)$ le produit restreint des groupes $\Htilde^i(k_v,\sG)$ pour $v \in \Omega$
par rapport aux images des flèches de restriction $H^i(\sO_v,\sG)\to H^i(k_v,\sG)=\Htilde^i(k_v,\sG)$ pour $v \in \Omega_f$.
Pour tout $i \in \Z$, la somme des accouplements locaux
\begin{align}
\Htilde^i(k_v,\RHom(\sG,\Gm)) \times \Htilde^{2-i}(k_v,\sG) \to \Br(k_v)
\end{align}
suivis des applications $\inv_v:\Br(k_v)\hookrightarrow \Q/\Z$
définit un accouplement canonique
\begin{align}
\mylabel{eq:accg}
\prod_{v \in \Omega}\smash{\rlap{\raise 5pt\hbox{$\mkern-4mu{}'$}}}\mkern4mu \Htilde^i(k_v,\RHom(\sG,\Gm)) \times H^{2-i}(k,\sG) \to \Q/\Z
\end{align}
fonctoriel en~$\sG$.

\begin{prop}
\mylabel{prop:surjideles}
L'homomorphisme
$$
\prod_{v \in \Omega}\smash{\rlap{\raise 5pt\hbox{$\mkern-4mu{}'$}}}\mkern4mu \Htilde^{-1}(k_v, \tau_{\leq 0}\RHom(\tau_{\leq 1}R\rho_*\Gm,\Gm))
\to \Hom(H^3(k,\tau_{\leq 1}R\rho_*\Gm),\Q/\Z)
$$
déduit de~(\ref{eq:accg}) est surjectif.
\end{prop}

\begin{proof}
Notons $\Pic_{\sC/U}$ le foncteur de Picard relatif de~$\sC$ sur~$U$.
Compte tenu du triangle distingué
$$
\xymatrix{
\Gm \ar[r] & \tau_{\leq 1}R\rho_*\Gm \ar[r] & \Pic_{\sC/U}[-1] \ar[r]^(.71){+1} & \vphantom{.}
}
$$
et de la nullité de $H^3(k,\Gm)$ (cf.~\cite[Ch.~I, Theorem~4.20~(b)]{milneadt}),
le groupe $\Hom(H^2(k,\Pic_{\sC/U}),\Q/\Z)$ se surjecte
sur $\Hom(H^3(k,\tau_{\leq 1}R\rho_*\Gm),\Q/\Z)$.
Il suffit donc de montrer que l'homomorphisme
$$
\beta:\prod_{v \in \Omega}\smash{\rlap{\raise 5pt\hbox{$\mkern-4mu{}'$}}}\mkern4mu \Htilde^0(k_v, \tau_{\leq 1}\RHom(\Pic_{\sC/U},\Gm))
\to \Hom(H^2(k,\Pic_{\sC/U}),\Q/\Z)
$$
déduit de~(\ref{eq:accg}) est surjectif.  À l'aide de la suite exacte
$$
\xymatrix{
0 \ar[r] & \Pic^0_{\sC/U} \ar[r] & \Pic_{\sC/U} \ar[r] & \Z \ar[r] & 0
}
$$
et de
la formule de Barsotti--Weil
$\tau_{\leq 1}\RHom(\Pic^0_{\sC/U},\Gm) = \Pic^0_{\sC/U}[-1]$
(cf.~\cite[Proposition~17.4, Theorem~18.1]{oort}),
on voit que~$\beta$ s'inscrit dans un diagramme commutatif
$$
\xymatrix@R=4ex@C=4em{
\ar[r]^(.56)\alpha\ar[d]
\displaystyle\smash[b]{\prod_{v \in \Omega}\smash{\rlap{\raise 5pt\hbox{$\mkern-4mu{}'$}}}\mkern4mu \Htilde^0(k_v, \Gm)}\vphantom{H^0(k_v)\tau_{\leq 1}}
& \Hom(H^2(k,\Z),\Q/\Z) \ar[d] \\
\ar[r]^(.56)\beta\ar@{->>}[d]^\delta
\displaystyle\smash[b]{\prod_{v \in \Omega}\smash{\rlap{\raise 5pt\hbox{$\mkern-4mu{}'$}}}\mkern4mu \Htilde^0(k_v, \tau_{\leq 1}\RHom(\Pic_{\sC/U},\Gm))}\vphantom{H^0(k_v)\tau_{\leq 1}}
& \Hom(H^2(k,\Pic_{\sC/U}),\Q/\Z) \ar[d] \\
\ar[r]^(.56)\gamma\displaystyle\prod_{v \in \Omega}\smash{\rlap{\raise 5pt\hbox{$\mkern-4mu{}'$}}}\mkern4mu \Htilde^{-1}(k_v, \Pic^0_{\sC/U})
& \Hom(H^2(k,\Pic^0_{\sC/U}),\Q/\Z)
}
$$
dont les colonnes sont exactes.
Le théorème~90 de Hilbert entraîne la surjectivité de~$\delta$.
Comme l'homomorphisme de restriction $H^2(k,\Pic^0_{\sC/U}) \to \bigoplus_{v \in \Omega_\infty} H^2(k_v,\Pic^0_{\sC/U})$ est un isomorphisme
(cf.~\cite[Ch.~I, Corollary~6.24]{milneadt}), la flèche~$\gamma$ est la somme directe, sur les places réelles~$v$ de~$k$,
des applications naturelles $$\gamma_v:\Htilde^{-1}(k_v,\Pic^0_{\sC/U}) \to \Hom(\Htilde^2(k_v,\Pic^0_{\sC/U}),\Q/\Z)\rlap{\text{.}}$$
Celles-ci sont bijectives d'après \cite[Ch.~I, Remark~3.7]{milneadt}.  Il en va donc de même de~$\gamma$.
Enfin, la flèche~$\alpha$ s'identifie d'après~\cite[\textsection2.3, Proposition~1]{serrecasselsfrohlich} à l'application de réciprocité de la théorie du corps de classes global,
qui est surjective (cf.~\cite[5.6]{tatecasselsfrohlich}).
La surjectivité de~$\beta$ résulte maintenant de celle de~$\alpha$, de~$\gamma$ et de~$\delta$.
\end{proof}

Considérons le diagramme commutatif
\begin{align}
\mylabel{diag:grand}
\begin{aligned}
\xymatrix{
& \displaystyle\smash[b]{\prod_{v \in \Omega}\smash{\rlap{\raise 5pt\hbox{$\mkern-4mu{}'$}}}\mkern4mu \Htilde^0(k_v, \tau_{\leq 1}R\rho_*\Gm)}
\vphantom{H^0(C\otimes_kk_v)} \ar[d] \ar@{->>}[r] & \Hom(H^3(C,\Gm),\Q/\Z) \ar[d] \\
H^1(k,T) \ar[d] \ar[r] &
\displaystyle\smash[b]{\prod_{v \in \Omega}\smash{\rlap{\raise 5pt\hbox{$\mkern-4mu{}'$}}}\mkern4mu H^1(k_v,T)}
\vphantom{H^1(k_v)} \ar[d] \ar[r] & \Hom(H^1(k,T^*),\Q/\Z) \ar[d] \\
\Picpluschapeau(C) \ar@{->>}[d] \ar[r] & \PicplusAchapeau(C) \ar[d] \ar[r] & \Hom(\Brplus(C),\Q/\Z) \ar[d] \\
\Picchapeau(C) \ar[r] & \PicAchapeau(C) \ar[r] & \Hom(\Br(C),\Q/\Z) \rlap{\text{.}}
}
\end{aligned}
\end{align}
La colonne de droite est la suite exacte
induite par le triangle distingué~(\ref{eq:tridistbrplus}).
Les autres flèches verticales proviennent du lemme~\ref{lem:tridistdescvplus}.
Les flèches horizontales de droite sont données par~(\ref{eq:accg}).
Il résulte de~\cite[Lemma~1.8]{vanhamel}
que
$$\prod_{v \in \Omega}{\rlap{\raise 5pt\hbox{$\mkern-4mu{}'$}}}\mkern4mu \widehat{\Htilde^0(k_v,\tau_{\leq 0}\sF)} = \PicplusAchapeau(C)\rlap{\text{;}}$$
de plus, on a $H^1(k_v,T)=\Htilde^1(k_v,T)$ pour tout $v \in \Omega$
puisque~$T$ est concentré en degré~$0$. Ces deux remarques donnent un sens à la seconde colonne.
Notons au passage que le groupe $\Htilde^0(k_v,\tau_{\leq 1}R\rho_*\Gm)$ n'est autre que $k_v^*$ si $v \in \Omega_f$; en revanche, si~$v$ est réelle, il dépend de
la jacobienne de~$C$.

La seconde ligne de~(\ref{diag:grand}) est exacte d'après le théorème de Poitou--Tate pour les groupes de type multiplicatif (cf.~\cite[Ch.~I, Theorem~4.20~(b)]{milneadt}).
La flèche verticale en bas à gauche est surjective en vertu des lemmes~\ref{lem:descpicplus} et~\ref{lem:chapeausurjectif}.
La colonne de droite est exacte par construction.
On a déjà établi l'exactitude de la seconde colonne au cran $\PicplusAchapeau(C)$ (voir le diagramme~(\ref{diag:ovkvpicplus})).
Enfin, la flèche horizontale du haut est surjective d'après la proposition~\ref{prop:surjideles}.
De tout cela il résulte, par une chasse au diagramme, que le groupe d'homologie de l'avant-dernière ligne s'injecte dans celui de la dernière ligne.
Le théorème~\ref{th:dualitearithmetique} est ainsi démontré.
\end{proof}

\begin{cor}
\mylabel{cor:existeu}
Soit $(u_v)_{v \in \Omega} \in \prod'_{v \in \Omega} \Picplus(C \otimes_k k_v)$ une famille orthogonale à~$\Brplus(C)$
pour~(\ref{eq:accbrpluspicplusachapeau}).
Supposons que l'image de $(u_v)_{v \in \Omega}$ dans $\prod_{v \in \Omega} \Pic(C \otimes_k k_v)$
provienne de $\Pic(C)$.
Alors il existe un élément~$u$ de~$\Picplus(C)$ dont l'image dans $\Picplus(C \otimes_k k_v)$
coïncide avec~$u_v$ pour $v \in \Omega_f$ et coïncide avec~$u_v$ à une norme de~$\kbar_v$ à~$k_v$ près pour $v \in \Omega_\infty$.
\end{cor}

\begin{proof}
D'après le théorème~\ref{th:dualitearithmetique},
l'image de $(u_v)_{v \in \Omega}$ dans $\PicplusAchapeau(C)$ provient d'un $\uchapeau \in \Picpluschapeau(C)$.
Considérons
le diagramme commutatif
\begin{align}
\mylabel{diag:picpluschapeauc}
\begin{aligned}
\xymatrix{
H^1(k,T) \ar@{=}[d] \ar[r] & \Picplus(C) \ar[r] \ar[d] & \Pic(C) \ar[r] \ar[d] & 0 \\
H^1(k,T) \ar[r] & \Picpluschapeau(C) \ar[r] & \Picchapeau(C)
}
\end{aligned}
\end{align}
fourni par le lemme~\ref{lem:descpicplus}.
Compte tenu du lemme~\ref{lem:noyauchapeau},
ses lignes sont exactes.
Par hypothèse, l'image de~$\uchapeau$ dans $\PicAchapeau(C)$ provient de $\Pic(C)$.
La flèche naturelle $\Picchapeau(C) \to \PicAchapeau(C)$ étant injective (cf.~remarque~\ref{rqs:conje}~(v)), il s'ensuit que
l'image de~$\uchapeau$ dans $\Picchapeau(C)$ provient elle aussi de~$\Pic(C)$.  Grâce au diagramme~(\ref{diag:picpluschapeauc}),
on en conclut que~$\uchapeau$ est l'image d'un $u \in \Picplus(C)$.

Pour $v \in \Omega_f$, la suite exacte~(\ref{se:descpicplus}) entraîne que $\Picplus(C \otimes_k k_v)$ ne contient pas d'élément infiniment divisible non nul,
vu que $H^1(k_v,T)$ est d'exposant fini, que $\Pic(C \otimes_k k_v)$ ne contient pas d'élément infiniment divisible non nul
et que le sous-groupe de torsion de $\Pic(C \otimes_k k_v)$ est d'exposant fini.
En d'autres termes, l'application naturelle $\Picplus(C \otimes_kk_v) \to \Picpluschapeau(C \otimes_k k_v)$ est injective pour tout $v \in \Omega_f$.
L'élément~$u$ vérifie donc les conditions requises.
\end{proof}

Le corollaire suivant est celui invoqué au \textsection\ref{sec:preuveP} afin d'accomplir la réduction~(v).

\begin{cor}
\mylabel{cor:existeubis}
Soient $y \in \Pic(C)$ et $(z_v)_{v \in \Omega} \in \prod_{v \in \Omega}Z_0(X \otimes_k k_v)$ une famille orthogonale à $\Br_\vert(X/C)$ pour l'accouplement de Brauer--Manin.
Pour $v \in \Omega_\infty$, soit $Z_v \subset X \otimes_k k_v$ une courbe dominant $C \otimes_k k_v$.
Supposons que l'on ait $f_*z_v=y$ dans $\Pic(C \otimes_k k_v)$ pour tout $v \in \Omega$.
Alors il existe des $0$\nobreakdash-cycles $d_v \in Z_0(Z_v \otimes_{k_v}\kbar_v)$ pour $v \in \Omega_\infty$ tels que, si l'on pose
$z'_v=z_v$ pour $v \in \Omega_f$ et $z'_v=z_v+N_{\kbar_v/k_v}(d_v)$ pour $v \in \Omega_\infty$,
la famille $(f_{*+}z'_v)_{v \in \Omega}$
appartienne à l'image de la flèche diagonale
$\Picplus(C) \to \prod'_{v \in \Omega} \Picplus(C\otimes_k k_v)$.
\end{cor}

\begin{proof}
Notons $i$ l'inclusion de $Z_v$ dans $X \otimes_k k_v$.
Pour tout $v \in \Omega_\infty$,
l'application $f_* \circ i_*:Z_0(Z_v \otimes_{k_v} \kbar_v) \to \Pic(C \otimes_k \kbar_v)$
 est surjective et
l'homomorphisme d'oubli $\phi:\Picplus(C \otimes_k \kbar_v) \to \Pic(C \otimes_k \kbar_v)$ est un isomorphisme
d'après le lemme~\ref{lem:descpicplus}.
L'application
$f_{*+} \circ i_*:Z_0(Z_v\otimes_{k_v} \kbar_v) \to \Picplus(C \otimes_k \kbar_v)$ est donc surjective.
Compte tenu de cette remarque, le corollaire~\ref{cor:existeubis} résulte
du corollaire~\ref{cor:existeu} appliqué à la famille $(f_{*+}z_v)_{v \in \Omega}$ et
de la commutativité du diagramme~(\ref{diag:compatibleacc}).
\end{proof}

\subsection{Démonstration de la proposition~\texorpdfstring{\ref{prop:pointorthogonal}}{4.7}}
\mylabel{sec:demproppointorthogonal}

Nous pouvons à présent établir la proposition~\ref{prop:pointorthogonal}.
Reprenons les notations du paragraphe~\ref{sec:preuveP}:
on dispose d'une famille $(z_v)_{v \in \Omega} \in \prod_{v \in \Omega} Z_0(X\otimes_k k_v)$
orthogonale à $\Br_\vert(X/C)$ pour l'accouplement de Brauer--Manin, d'un ensemble fini $S \subset \Omega$ et d'un morphisme fini $\pi:C \to \P^1_k$
tels que $f_*z_v=\pi^{-1}(\infty)$ pour tout $v \in S$ et $f_*z_v=\pi^{-1}(0)$ pour tout $v \in \Omega \setminus S$.
Par hypothèse, le lieu de ramification $R \subset C$ de~$\pi$ vérifie $\pi(M) \cap \pi(R)=\emptyset$ et~$\pi$ induit un isomorphisme $M \isoto \pi(M)$.
Notant $p:W \to \P^1_k$ la restriction des scalaires à la Weil de~$f$ le long de~$\pi$ et $[z_v] \in W(k_v)$ le $k_v$\nobreakdash-point de~$W$ défini par~$z_v$,
nous devons prouver que la famille $([z_v])_{v \in \Omega} \in \prod_{v \in \Omega} W(k_v)$ est orthogonale à $\Br_\vert(W/\P^1_k)$.

La variété~$W$ et le morphisme $p:W \to \P^1_k$ vérifient les hypothèses du paragraphe~\ref{subsec:vhnotations}.  Un complexe $\vplus(\P^1_k)$ et un groupe abélien
$\Picplus(\P^1_k)$ leur sont donc associés.
D'après la proposition~\ref{prop:rwproprietes}, il existe une flèche \og{}norme\fg{} canonique $N_{C/\P^1_k}:\pi_*\vplus(C) \to \vplus(\P^1_k)$.
Notons $\pi^*:\Picplus(\P^1_k) \to \Picplus(C)$ l'homomorphisme qui s'en déduit naturellement.

\begin{lem}
\mylabel{lem:descpicplusweil}
Soit $F \in \Pic(C)$ la classe d'une fibre de~$\pi$. Pour toute extension $K/k$, le morphisme $\pi^*$ s'inscrit dans une suite exacte canonique
$$
\xymatrix{
0 \ar[r] & \Picplus(\P^1_K) \ar[r]^(.45){\pi^*} & \Picplus(C \otimes_k K) \ar[r] & \Pic(C\otimes_k K)/\langle F \otimes_k K \rangle \ar[r] & 0\rlap{\text{.}}
}
$$
\end{lem}

\begin{proof}
D'après le lemme~\ref{lem:descpicplus} appliqué deux fois, il existe un diagramme commutatif à lignes exactes
$$
\xymatrix{
K^* \ar[r] & H^1(K,T) \ar[r] & \Picplus(C \otimes_kK) \ar[r] & \Pic(C\otimes_kK) \ar[r] & 0 \\
\ar@{=}[u] K^* \ar[r] & \ar@{=}[u] H^1(K,T) \ar[r] & \ar[u]_{\pi^*} \Picplus(\P^1_K) \ar[r] & \ar[u]_{\pi^*} \Pic(\P^1_K) \ar[r] & 0
}
$$
où~$T$ est le schéma en groupes de type multiplicatif sur~$k$ défini au début de la démonstration du théorème~\ref{th:dualitearithmetique}.
Le lemme~\ref{lem:descpicplusweil} en résulte.
\end{proof}

\begin{lem}
\mylabel{lem:compatibleweil}
Pour tout $v \in \Omega$, on a $\pi^* p_{*+}[z_v] = f_{*+}z_v$ dans $\Picplus(C\otimes_kk_v)$.
\end{lem}

\begin{proof}
Soient  $\sigma: W \times_{\P^1_k} C \to X$
et $\pr_1:W \times_{\P^1_k}C \to W$
le morphisme d'adjonction et la première projection.
Désignant indifféremment par~$j$ les immersions ouvertes $C \setminus M \hookrightarrow C$, $\P^1_k \setminus \pi(M) \hookrightarrow \P^1_k$,
$X \setminus f^{-1}(M) \hookrightarrow X$, $W \setminus p^{-1}(\pi(M)) \hookrightarrow W$ et $(W \times_{\P^1_k}C) \setminus \sigma^{-1}(f^{-1}(M)) \hookrightarrow W \times_{\P^1_k}C$,
notons $\theta: \pi_* Rf_*(j_*\Gm) \to Rp_*(j_*\Gm)$ la composée de la flèche naturelle
$\pi_* Rf_*(j_*\Gm) \to \pi_*Rf_*R\sigma_*(j_*\Gm)$, de l'isomorphisme canonique
$\pi_*Rf_*R\sigma_*(j_*\Gm) = R(\pi \circ f \circ \sigma)_* (j_*\Gm) = Rp_*  \pr_{1*} (j_*\Gm)$
et de la flèche $Rp_* \pr_{1*} (j_*\Gm) \to Rp_* (j_*\Gm)$ induite par la norme $\pr_{1*}(j_*\Gm) \to j_*\Gm$.
Le morphisme~$\theta$ s'insère dans un diagramme commutatif
\begin{align}
\begin{aligned}
\mylabel{diag:verifcompatzv}
\xymatrix{
\pi_* \vplus(C) \ar[d] \ar[r] & \pi_* j_*\Gm \ar[d]^{N_{C/\P^1_k}} \ar[r] & \pi_*Rf_*(j_*\Gm) \ar[d]^\theta \\
\vplus(\P^1_k) \ar[r] & j_*\Gm \ar[r] & Rp_*(j_*\Gm) \rlap{\text{.}}
}
\end{aligned}
\end{align}
(Pour établir la commutativité du carré de droite, il suffit de la vérifier sur les faisceaux de cohomologie de degré~$0$, où elle est évidente.)
Notons $\rho:\P^1_k \to \Spec(k)$ le morphisme structural de~$\P^1_k$
et posons $\sH_W=\RHom(R(\rho \circ p)_*(j_*\Gm),\Gm)$
et $\sH_X=\RHom(R(\rho \circ \pi \circ f)_*(j_*\Gm),\Gm)$.
Le diagramme~(\ref{diag:verifcompatzv}) induit pour tout $v \in \Omega$ le carré de gauche du diagramme commutatif
\begin{align}
\begin{aligned}
\mylabel{diag:verifcompatzvbis}
\xymatrix{
\Picplus(C \otimes_k k_v) & H^0(k_v,\sH_X) \ar[l] & Z_0((X \setminus f^{-1}(M)) \otimes_k k_v) \ar[l]_(.6)c \\
\ar[u]_{\pi^*} \Picplus(\P^1_{k_v}) & H^0(k_v,\sH_W) \ar[u]_{\RHom(R\rho_*\theta,\Gm)} \ar[l] & Z_0((W \setminus p^{-1}(\pi(M))) \otimes_k k_v) \ar[u]_{\sigma_* \pr_1^*} \ar[l]_(.6)c
}
\end{aligned}
\end{align}
dont les flèches étiquetées~$c$ sont les applications classe de cycle construites par van Hamel~\cite[\textsection3.1]{vanhamellichtenbaumtate}.
Par définition de~$f_{*+}$ (resp.~de~$p_{*+}$), la composée des flèches horizontales de la première (resp.~seconde) ligne de~(\ref{diag:verifcompatzvbis})
coïncide avec~$f_{*+}$ (resp.~$p_{*+}$).
Calculant l'image de $[z_v]\in Z_0((W \setminus p^{-1}(\pi(M))) \otimes_k k_v)$ dans $\Picplus(C \otimes_k k_v)$ par deux chemins différents dans ce diagramme,
on trouve ainsi que $f_{*+}z_v=\pi^*p_{*+}[z_v]$.
\end{proof}

Appliquant le lemme~\ref{lem:descpicplusweil} à $K=k$ et $K=k_v$ pour tout $v \in \Omega$, on obtient un diagramme commutatif
\begin{align}
\mylabel{diag:granddiag}
\begin{aligned}
\demicrochet
\xymatrix@C=1.35em{
0 \ar[r] &\Picplus(\P^1_k) \ar[r]^(.45){\pi^*} \ar[d] & \Picplus(C) \ar[r]\ar[d] & \Pic(C)/\langle F\rangle \ar@{^{ (}->}[d] \\
0 \ar[r] & \displaystyle\prod_{v \in \Omega} \Picplus(\P^1_{k_v}) \ar[r]^(.45){\pi^*} & \displaystyle\prod_{v \in \Omega} \Picplus(C\otimes_k k_v) \ar[r] &
\displaystyle\prod_{v \in \Omega} \Pic(C\otimes_k k_v)/\langle F \otimes_k k_v \rangle
}
\end{aligned}
\end{align}
dont les lignes sont exactes.
D'après le lemme~\ref{lem:compatibleweil}
et l'hypothèse~(v)
du paragraphe~\ref{sec:preuveP}, la famille $(\pi^*p_{*+}[z_v])_{v \in \Omega} \in \prod_{v \in \Omega} \Picplus(C \otimes_kk_v)$
appartient à l'image de la seconde flèche verticale de~(\ref{diag:granddiag}).
La flèche verticale de droite étant injective,
il s'ensuit, par une chasse au diagramme,
que $(p_{*+}[z_v])_{v \in \Omega}$ est l'image d'un élément
de $\Picplus(\P^1_k)$ par l'application diagonale
$\Picplus(\P^1_k) \to \prod_{v \in \Omega} \Picplus(\P^1_{k_v})$.
La famille
 $(p_{*+}[z_v])_{v \in \Omega}$
 est donc orthogonale à $\Brplus(\P^1_k)$ pour l'accouplement~(\ref{eq:accbrpluspicplusachapeau}).
Or il est équivalent que cette famille soit orthogonale à $\Brplus(\P^1_k)$ pour
l'accouplement~(\ref{eq:accbrpluspicplusachapeau})
ou que la famille $([z_v])_{v \in \Omega}$ soit orthogonale à $\Br_\vert(W/\P^1_k)$
pour l'accouplement de Brauer--Manin,
compte tenu de la commutativité du diagramme~(\ref{diag:compatibleacc})
et de la surjectivité de l'application $p^*_+:\Brplus(\P^1_k) \to \Br_\vert(W/\P^1_k)$ (cf.~remarque~\ref{rq:picplusbrpluscarzero}~(i)).
La proposition~\ref{prop:pointorthogonal} est donc établie.

\bibliographystyle{smfalphamodif}
\bibliography{zerocycles}
\end{document}